\documentclass[onefignum,onetabnum]{siamart171218}

\usepackage{lipsum}
\usepackage{amsfonts}
\usepackage{mathrsfs}
\usepackage{amssymb}
\usepackage{graphicx}
\usepackage{epstopdf}
\usepackage{algpseudocode}
\usepackage{tensor}
\usepackage{verbatim}
\usepackage{mathtools}
\usepackage{cleveref}
\usepackage{framed}
\usepackage{tikz}
\usetikzlibrary{calc, shapes, angles, quotes}
\usepackage{subcaption}
\usepackage{bm}
\usepackage{tcolorbox}

\ifpdf
\DeclareGraphicsExtensions{.eps,.pdf,.png,.jpg}
\else
\DeclareGraphicsExtensions{.eps}
\fi

\newlength{\leftstackrelawd}
\newlength{\leftstackrelbwd}
\def\leftstackrel#1#2{\settowidth{\leftstackrelawd}%
	{${{}^{#1}}$}\settowidth{\leftstackrelbwd}{$#2$}%
	\addtolength{\leftstackrelawd}{-\leftstackrelbwd}%
	\leavevmode\ifthenelse{\lengthtest{\leftstackrelawd>0pt}}%
	{\kern-.5\leftstackrelawd}{}\mathrel{\mathop{#2}\limits^{#1}}}

\newcommand{\bdd}[1]{ \boldsymbol{#1} }
\newcommand{\unitvec}[1]{\hat{\bdd{#1}}}

\newsiamremark{remark}{Remark}
\newsiamremark{hypothesis}{Hypothesis}
\crefname{hypothesis}{Hypothesis}{Hypotheses}
\newsiamthm{claim}{Claim}

\headers{Preconditioning Stokes Flow}{M. Ainsworth and C. Parker}

\title{A Mass Conserving Mixed \lowercase{$hp$}-FEM Scheme for Stokes Flow. Part III: Implementation and Preconditioning \thanks{Submitted to the editors DATE.}}

\author{Mark Ainsworth\thanks{ Division of Applied Mathematics, Brown University, Providence, RI
		(\email{mark\_ainsworth@brown.edu}, \email{charles\_parker@brown.edu}).} \and Charles Parker\footnotemark[2]  }

\usepackage{amsopn}

\DeclareMathOperator{\vcurl}{\mathbf{curl}}

\DeclareMathOperator{\dive}{div}

\DeclareMathOperator{\supp}{supp}

\ifpdf
\hypersetup{
  pdftitle={A Mass Conserving Mixed \lowercase{$hp$}-FEM Scheme for Stokes Flow. Part III: Preconditioning},
  pdfauthor={M. Ainsworth and C. Parker}
}
\fi

\begin{document}

\maketitle

% REQUIRED
\begin{abstract}		
	This is the third part in a series on a mass conserving, high order, mixed finite element method for Stokes flow. In this part, we study a block-diagonal preconditioner for the indefinite Schur complement system arising from the discretization of the Stokes equations using these elements. The underlying finite element method is uniformly stable in both the mesh size $h$ and polynomial order $p$, and we prove bounds on the eigenvalues of the preconditioned system which are independent of $h$ and grow modestly in $p$. The analysis relates the Schur complement system to an appropriate variational setting with subspaces for which exact sequence properties and inf-sup stability hold. Several numerical examples demonstrate agreement with the theoretical results.
\end{abstract}

% REQUIRED

% REQUIRED

\section{Introduction}
\label{sec:intro}

This paper is the third part in a series discussing a
mass conserving, high order, mixed
finite element method for Stokes flow on a simply connected polygon $\Omega$
with boundary $\Gamma = \partial \Omega$: Find $(\bdd{u}, p) \in
\bdd{H}^1_0(\Omega) \times L^2_0(\Omega)$ such that
\begin{subequations}
	\label{eq:stokes weak form}
	\begin{align} a(\bdd{u}, \bdd{v}) + b(\bdd{v}, p) &= (\bdd{f}, \bdd{v}) &
	&\forall \bdd{v} \in \bdd{H}^1_0 (\Omega) \label{eq:stokes weak form a} \\
	b(\bdd{u}, q) &= 0 & &\forall q \in L^2_0(\Omega), \label{eq:stokes weak form
		b}
	\end{align}
\end{subequations}
where $\bdd{u} = (u_1, u_2)$ is the fluid velocity, $p$ the pressure, $\bdd{f}
\in \bdd{L}^2(\Omega)$ the body force,
$a(\bdd{u}, \bdd{v}) := \nu(\nabla \bdd{u}, \nabla \bdd{v})$ and $b(\bdd{v}, p) := -(\dive \bdd{v}, p)$. Without loss of generality, by rescaling, we may reduce \cref{eq:stokes weak form} to the case where the kinematic viscosity $\nu = 1$. Here, $H^s(\Omega)$ and $H^s_0(\Omega)$ denote the usual Sobolev spaces \cite{Adams03},
$\bdd{H}^s(\Omega)$, $\bdd{H}^s_0(\Omega)$ the vector valued Sobolev spaces, i.e. $\bdd{H}^s(\Omega) := [H^{s}(\Omega)]^2$, and $L^2_0(\Omega)$ denotes the
(closed) subspace of square integrable functions with vanishing average value:
\begin{align*}
L^2_0(\Omega) := \left\{ q \in L^2(\Omega) : \int_{\Omega} q \ d\bdd{x} =
0\right\}.
\end{align*}
Problem \cref{eq:stokes weak form} is approximated a using mixed, high order,
finite element scheme on a mesh $\mathcal{T}$ as follows: Find $(\bdd{u}_{hk}, p_{hk}) \in \bdd{V}_{0}
\times Q_{0}$ such that
\begin{subequations}
	\label{eq:discrete stokes system}
	\begin{align}
	a(\bdd{u}_{hk}, \bdd{v}) + b(\bdd{v}, p_{hk}) &= (\bdd{f}, \bdd{v}) & &\forall
	\bdd{v} \in \bdd{V}_{0} \label{eq:discrete stokes system 1} \\ b(\bdd{u}_{hk},
	q) &= 0 & &\forall q \in Q_{0}, \label{eq:discrete stokes system 2}
	\end{align}
\end{subequations}
where the finite element spaces are chosen to be
\cite{AinCP19StokesI,AinCP19StokesII,Falk13}:
\begin{align*}
V &:= \{ v \in H^1(\Omega) : v|_{K} \in \mathcal{P}_{k}(K) \ \forall K \in
\mathcal{T}, \ v \text{ is ${C}^1$ at noncorner vertices} \}, \\ 
{Q} &:= \{ q \in
{L}^2(\Omega) : q|_{K} \in \mathcal{P}_{k-1}(K) \ \forall K \in \mathcal{T}, \
q \text{ is $C^0$ at noncorner vertices} \},
\end{align*}
$V_0 := V \cap H^1_0(\Omega)$, $\bdd{V}_0 = V_0 \times V_0$, $Q_0 = Q \cap L^2_0(\Omega)$, $\mathcal{P}_{k}$ denotes the space of all polynomials of degree at most
$k$, and a corner vertex is a vertex of the physical domain $\Omega$. The local degrees of freedom of the spaces $V$ and $Q$ are illustrated in \cref{fig:local dofs}.

\begin{figure}[ht]
	\centering
	\begin{subfigure}[b]{0.48\linewidth}
		\centering
		\begin{tikzpicture}
		\node[regular polygon, regular polygon sides=3, draw, minimum size=3cm]
		(m) at (0,0) {};

		\coordinate (L1) at (m.corner 1);
		\coordinate (L2) at (m.corner 2);
		\coordinate (L3) at (m.corner 3);
		
		% near vertex 1		
		\filldraw (barycentric cs:L1=0,L2=1,L3=0) circle (2pt) node[below]{}; 
		
		\draw (barycentric cs:L1=0,L2=1,L3=0) circle (5pt) node[below]{}; 
		
		% edge 3
		\filldraw (barycentric cs:L1=0,L2=1/3,L3=2/3) circle (2pt) node[below]{};
		
		\filldraw (barycentric cs:L1=0,L2=2/3,L3=1/3) circle (2pt) node[below]{};
		
		% vertex 2
		\filldraw (barycentric cs:L1=0,L2=0,L3=1) circle (2pt) node[above]{}; 
		
		\draw (barycentric cs:L1=0,L2=0,L3=1) circle (5pt) node[above]{}; 
		
		% edge 1
		\filldraw (barycentric cs:L1=1/3,L2=0,L3=2/3) circle (2pt) node[above]{}; 
		
		\filldraw (barycentric cs:L1=2/3,L2=0,L3=1/3) circle (2pt) node[above]{}; 
		
		% vertex 3
		\filldraw (barycentric cs:L1=1,L2=0,L3=0) circle (2pt) node[above]{}; 
		
		\draw (barycentric cs:L1=1,L2=0,L3=0) circle (5pt) node[above]{};

		% edge 2
		\filldraw (barycentric cs:L1=1/3,L2=2/3,L3=0) circle (2pt) node[left]{};
		
		\filldraw (barycentric cs:L1=2/3,L2=1/3,L3=0) circle (2pt) node[left]{};
		
		% interiors
		\filldraw[] (barycentric cs:L1=1/5,L2=1/5,L3=3/5) circle (2pt) node[above]{};
		
		\filldraw[] (barycentric cs:L1=1/5,L2=2/5,L3=2/5) circle (2pt) node[above]{};
		
		\filldraw[] (barycentric cs:L1=1/5,L2=3/5,L3=1/5) circle (2pt) node[above]{};
		
		\filldraw[] (barycentric cs:L1=2/5,L2=1/5,L3=2/5) circle (2pt) node[above]{};
		
		\filldraw[] (barycentric cs:L1=2/5,L2=2/5,L3=1/5) circle (2pt) node[above]{};
		
		\filldraw[] (barycentric cs:L1=3/5,L2=1/5,L3=1/5) circle (2pt) node[above]{};

		\end{tikzpicture}
		\caption{}
		\label{fig:v dofs}
	\end{subfigure}
	\hfill
	\begin{subfigure}[b]{0.48\linewidth}
		\centering
		\begin{tikzpicture}
		\node[regular polygon, regular polygon sides=3, draw, minimum size=3cm]
		(m) at (0,0) {};
		\node[regular polygon, regular polygon sides=3, minimum size=2cm]
		(m2) at (0,0) {};
		
		\coordinate (L1) at (m.corner 1);
		\coordinate (L2) at (m.corner 2);
		\coordinate (L3) at (m.corner 3);
		
		\coordinate (LL1) at (m2.corner 1);
		\coordinate (LL2) at (m2.corner 2);
		\coordinate (LL3) at (m2.corner 3);
		
		% near vertex 1		
		\filldraw (barycentric cs:L1=0,L2=1,L3=0) circle (2pt) node[below]{}; 
		
		% vertex 2
		\filldraw (barycentric cs:L1=0,L2=0,L3=1) circle (2pt) node[above]{}; 
		
		% vertex 3
		\filldraw (barycentric cs:L1=1,L2=0,L3=0) circle (2pt) node[above]{}; 
		
		% interiors
		\filldraw (barycentric cs:LL1=0,LL2=1/4,LL3=3/4) circle (2pt) node[below]{}; 
		
		\filldraw (barycentric cs:LL1=0,LL2=2/4,LL3=2/4) circle (2pt) node[below]{}; 
		
		\filldraw (barycentric 
		cs:LL1=0,LL2=3/4,LL3=1/4) circle (2pt) node[below]{}; 
		
		\filldraw (barycentric cs:LL1=1/4,LL2=0/4,LL3=3/4) circle (2pt) node[below]{}; 
		
		\filldraw (barycentric cs:LL1=1/4,LL2=1/4,LL3=2/4) circle (2pt) node[below]{}; 
		
		\filldraw (barycentric cs:LL1=1/4,LL2=2/4,LL3=1/4) circle (2pt) node[below]{}; 
		
		\filldraw (barycentric cs:LL1=1/4,LL2=3/4,LL3=0/4) circle (2pt) node[below]{}; 
		
		\filldraw (barycentric cs:LL1=2/4,LL2=0/4,LL3=2/4) circle (2pt) node[below]{};
		
		\filldraw (barycentric cs:LL1=2/4,LL2=1/4,LL3=1/4) circle (2pt) node[below]{};
		
		\filldraw (barycentric cs:LL1=2/4,LL2=2/4,LL3=0/4) circle (2pt) node[below]{};
		
		\filldraw (barycentric cs:LL1=3/4,LL2=0/4,LL3=1/4) circle (2pt) node[below]{};
		
		\filldraw (barycentric cs:LL1=3/4,LL2=1/4,LL3=0/4) circle (2pt) node[below]{};
		
		\end{tikzpicture}
		\caption{}
		\label{fig:q dofs}
	\end{subfigure}
	\caption{Local degrees of freedom for the finite element spaces (a) $V$ and (b) $Q$ in the case $k=5$. Dots indicate degrees of freedom corresponding to evaluation at the point located at the dot while circles indicate gradient evaluation. \label{fig:local dofs}}
\end{figure}
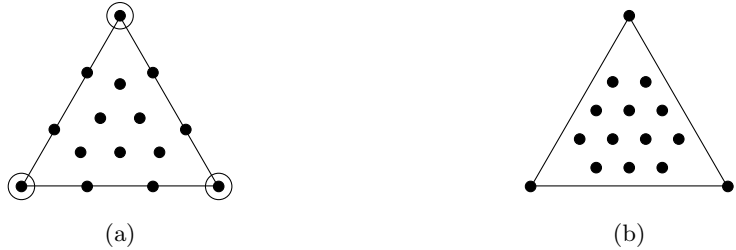

In Part I \cite{AinCP19StokesI}, it was shown that that these elements are uniformly
inf-sup stable in the mesh size $h$ and polynomial order $k$ if the mesh
$\mathcal{T}$ is corner-split which, roughly speaking, means that every element $K \in
\mathcal{T}$ has at most one edge lying on the domain boundary $\Gamma$; for a precise definition, see \cite[p. 12]{AinCP19StokesI}.
\begin{theorem}[Theorem 3.1 \& Corollary 3.2 \cite{AinCP19StokesI}]
	\label{thm:inf-sup global spaces} If the mesh $\mathcal{T}$ is corner-split,
	then for every $q \in Q_0$,  there exists a $\bdd{v} \in \bdd{V}_0$ such that
	$\dive \bdd{v} = q$ and
	\begin{align*}
	\|\bdd{v}\|_{\bdd{H}^1(\Omega)} \leq \beta^{-1} \|q\|_{L^2(\Omega)},
	\end{align*}
	where $0 < \beta < 1$ is independent of $k$ and $h$. Thus, the spaces
	$\bdd{V}_0 \times Q_0$ are uniformly inf-sup stable:
	\begin{align}
	\label{eq:inf-sup global spaces} \inf_{0 \neq q \in Q_0} \sup_{\bdd{0} \neq
		\bdd{v} \in \bdd{V}_0} \frac{b(\bdd{v}, q)}{\|\bdd{v}\|_{\bdd{H}^1(\Omega)}
		\|q\|_{L^2(\Omega)} } \geq \beta.
	\end{align}
\end{theorem}
Strictly speaking, \cite[Corollary 3.2]{AinCP19StokesI} shows that $\beta$ depends on the mesh-dependent quantity $\Theta(\mathcal{T})$ defined in \cite[eq. (3.2)]{AinCP19StokesI}, but is nevertheless bounded independently of the mesh size $h$ and polynomial degree $k$. Moreover, the finite element solution $\bdd{u}_{hk}$ will be pointwise
divergence free \cite[\S 1 and Theorem 2.6]{AinCP19StokesI}.  In Part II
\cite{AinCP19StokesII}, it was shown that these elements have optimal
approximation properties in both the mesh size $h$ and the polynomial order
$k$. On locally quasi-uniform meshes, the finite element solution to
\cref{eq:discrete stokes system} converges at the optimal algebraic rate to the solution to \cref{eq:stokes weak form} \cite[Theorem 2.2]{AinCP19StokesII}. Moreover, if the data $\bdd{f}$ belongs to a particular countably normed space, then the finite element method with properly geometrically graded meshes
converges exponentially fast as both the mesh is refined and the polynomial
degree is increased \cite[Corollary 2.5]{AinCP19StokesII}. The spaces
$\bdd{V}_0 \times Q_0$ are currently the only known triangular finite element spaces
that are uniformly inf-sup stable in $h$ and $k$, give pointwise divergence
free velocities, and posses optimal approximation properties.

In the current work, we turn to issues relating to the practical application of the
method. In particular, we give explicit bases for the spaces $\bdd{V}$ and $Q$
that result in an efficient preconditioner for the solution of the resulting
linear system for \cref{eq:discrete stokes system}, which may be used in
conjunction with an iterative solver for indefinite systems, such as MINRES
\cite{Paige75}.  The preconditioner consists of a standard static condensation,
or elimination of the interior degrees of freedom, along with an Additive Schwarz
preconditioner (ASM) \cite{SmBjGr04,ToWi06} for the resulting Schur complement
system associated with the interface degrees of freedom. Thanks to a judicious choice of basis, the condition number grow at most as $\log^3 k$ as $k$ is increased, and is uniform in the mesh size.

The current work finds inspiration in the early works of
\cite{Bram90,Klawonn98block,Silv94,Wathen93} for $h$-version methods,
\cite{Ain99,letallec97} for $hp$-version finite element methods, and
\cite{Klaw98,Mad93,Pav98,Pav99} for spectral element methods, each of which
developed block diagonal and/or block triangular preconditioners in terms of
existing preconditioners for second order elliptic problems.  Unfortunately,
these types of approaches do not readily extend to the mixed finite element
scheme \cref{eq:discrete stokes system} owing to the additional smoothness requirements imposed at
element vertices for both the velocity and pressure spaces. Our treatment of
these degrees of freedom is similar to the treatment of the second order
derivative degrees of freedom in preconditioning the stiffness matrix for
$H^2(\Omega)$-conforming methods \cite{AinCP19Precon} and the treatment of the
vertex degrees of freedom in preconditioning the mass matrix for $H^1(\Omega)$
problems \cite{AinJia19}.

\section{General Form of a Block-Diagonal Preconditioner}
\label{sec:precond and main theorem}
 
By fixing bases for the spaces $\bdd{V}_0$ and $Q$, we may express $\bdd{u} \in \bdd{V}_0$ and $p \in Q$ as
\begin{align*}
\bdd{u} = \vec{u}_{E}^{T} \vec{\Phi}_{E} +  \vec{u}_{I}^{T} \vec{\Phi}_{I} \quad \text{and} \quad p = \vec{p}_e^{T} \vec{\psi}_{e} + \vec{p}_{\iota}^{T} \vec{\psi}_{\iota},
\end{align*}
for suitable $\vec{u}_E$, $\vec{u}_I$, $\vec{p}_e$, $\vec{p}_{\iota}$, where $\vec{\Phi}_E$ is the vector of exterior velocity basis functions (vertex and edge functions), $\vec{\Phi}_{I}$ the vector of interior velocity basis functions, $\vec{\psi}_{e}$ the vector of exterior pressure basis functions, and $\vec{\psi}_{\iota}$ the vector of interior pressure basis functions. Here, the exterior pressure functions consist of vertex functions and a function corresponding to the average value over each element. The variational problem \cref{eq:discrete stokes system} in matrix form then reads
\begin{align}
\label{eq:full system matrix}
\begin{bmatrix}
\begin{array}{cc|cc}
\bdd{A}_{EE} & \bdd{B}_{E e} & \bdd{A}_{EI} & \bdd{B}_{E\iota} \\
\bdd{B}_{e E} & \bdd{0} &  \bdd{B}_{e I} & \bdd{0}\\
\hline 
\bdd{A}_{IE} & \bdd{B}_{I e} &  \bdd{A}_{II} & \bdd{B}_{I \iota} \\
\bdd{B}_{\iota E} & \bdd{0} & \bdd{B}_{\iota I} & \bdd{0}
\end{array}
\end{bmatrix}
\begin{bmatrix}
\begin{array}{c}
\vec{u}_E \\ \vec{p}_e \\ \hline \vec{u}_I \\ \vec{p}_\iota
\end{array}
\end{bmatrix}
&= \begin{bmatrix}
\begin{array}{c}
\vec{f}_E \\
\vec{0} \\
\hline
\vec{f}_I \\
\vec{0}
\end{array}
\end{bmatrix}.
\end{align}
The matrix appearing in \cref{eq:full system matrix} is symmetric but indefinite, owing to the zero subblocks. The pressure variable in problem \cref{eq:discrete stokes system} is unique up to a constant, meaning that the matrix in \cref{eq:full system matrix} has a one-dimensional null space. Nevertheless, the system \cref{eq:full system matrix} is consistent since the components of the load vector corresponding to pressure basis functions vanish identically and, a fortiori, are orthogonal to constant pressure modes. Consequently, the system \cref{eq:full system matrix} is uniquely solvable up to the addition of a constant in the pressure thanks to the inf-sup condition \cref{eq:inf-sup global spaces} and the uniform ellipticity of $a(\cdot, \cdot)$. 

The conditioning of the matrix, in common with standard $hp$-finite elements, degenerates rapidly with both the mesh size $h$ and the polynomial order $k$ of the elements. Indeed, almost every practical choice of basis function results in a rapid deterioration of the condition number $k$, even for symmetric, positive definite systems \cite{Ain03,Olsen95}. We seek a preconditioner for the symmetric, indefinite system \cref{eq:full system matrix} which controls the growth of the conditioning in both $h$ and $k$.
 
The first step towards preconditioning is to eliminate, or \textit{statically condense}, the interior degrees of freedom to arrive at the Schur complement system
\begin{align}
\label{eq:Schur complement system}
\bdd{S} \begin{bmatrix}
\vec{u}_E \\ \vec{p}_e
\end{bmatrix} = \begin{bmatrix}
\vec{f}_{E}^{*} \\ \vec{g}_{e}^{*}
\end{bmatrix} :=  \begin{bmatrix}
\vec{f}_E \\ \vec{0}
\end{bmatrix} - \begin{bmatrix}
\bdd{A}_{EI} & \bdd{B}_{E\iota} \\ 
\bdd{B}_{eI} & \bdd{0}
\end{bmatrix} \begin{bmatrix}
\bdd{A}_{II} & \bdd{B}_{I\iota} \\ 
\bdd{B}_{\iota I} & \bdd{0}
\end{bmatrix}^{-1} \begin{bmatrix}
\vec{f}_I \\ \vec{0}
\end{bmatrix},
\end{align}
where
\begin{align}
\label{eq:schur complement matrix}
\bdd{S} &= \begin{bmatrix}
\widetilde{\bdd{A}} & \widetilde{\bdd{B}}^{T} \\
\widetilde{\bdd{B}} & \bdd{0}
\end{bmatrix} = \begin{bmatrix}
\bdd{A}_{EE} & \bdd{B}_{Ee} \\ 
\bdd{B}_{eE} & \bdd{0}
\end{bmatrix} - \begin{bmatrix}
\bdd{A}_{EI} & \bdd{B}_{E \iota} \\ 
\bdd{B}_{eI} & \bdd{0}
\end{bmatrix} \begin{bmatrix}
\bdd{A}_{II} & \bdd{B}_{I \iota} \\ 
\bdd{B}_{\iota I} & \bdd{0}
\end{bmatrix}^{-1} \begin{bmatrix}
\bdd{A}_{IE} & \bdd{B}_{Ie} \\ 
\bdd{B}_{\iota E} & \bdd{0}
\end{bmatrix}
\end{align}
and we have used the fact (see \cref{lem:matrix relations}) that the $(2,2)$ block of Schur complement matrix $\bdd{S}$ reduces to the zero matrix. The inverse of the matrix appearing in \cref{eq:Schur complement system,eq:schur complement matrix} is well-defined by \cref{thm:inf-sup global interior spaces}. After the degrees of freedom on the element interfaces are in hand, the interior degrees of freedom can be recovered by back substitution using the relation
\begin{align*}
\begin{bmatrix}
\vec{u}_I \\ \vec{p}_{\iota}
\end{bmatrix} = \begin{bmatrix}
\bdd{A}_{II} & \bdd{B}_{I \iota} \\ 
\bdd{B}_{\iota I} & \bdd{0}
\end{bmatrix}^{-1} \left( \begin{bmatrix}
\vec{f}_I \\ \vec{0}
\end{bmatrix} - \begin{bmatrix}
\bdd{A}_{IE} & \bdd{B}_{Ie} \\ 
\bdd{B}_{\iota E} & \bdd{0}
\end{bmatrix} \begin{bmatrix}
\vec{u}_E \\ \vec{p}_e
\end{bmatrix}  \right).
\end{align*}
The element interface degrees of freedom are obtained by solving the Schur complement system \cref{eq:Schur complement system}. The matrix $\bdd{S}$ defined in \cref{eq:schur complement matrix} is symmetric and indefinite, and inherits the one dimensional null space from the full system matrix \cref{eq:full system matrix}, again corresponding to the constant pressure mode. Similarly, the right hand side in \cref{eq:Schur complement system} inherits the consistency of the load vector meaning that \cref{eq:Schur complement system} is uniquely solvable up to a constant pressure mode. The indefiniteness of the problem coupled with the presence of a low dimensional null space suggests using a MINRES iterative solver \cite{Paige75} in conjunction with a suitable preconditioner. 

We seek a block diagonal matrix of the form
\begin{align}
\label{eq:preconditioner form}
\bdd{P} = \begin{bmatrix}
\bar{\bdd{A}} & \bdd{0} \\
\bdd{0} & \bar{\bdd{M}}
\end{bmatrix}
\end{align}
to precondition $\bdd{S}$, where $\bar{\bdd{A}}$ and $\bar{\bdd{M}}$ are symmetric positive definite matrices. The convergence of the MINRES algorithm with preconditioner $\bdd{P}^{-1}$ depends on the location of the nonzero eigenvalues of $\bdd{P}^{-1} \bdd{S}$ \cite[Remark 4.13 and \S 4.2.4]{Elman14}. In particular, let $\delta$, $\Delta$, $\theta$, and $\Theta$ be nonnegative constants such that
\begin{align*}
%\label{eq:intro spectral equivalences}
\delta \leq \frac{\vec{u}_E^T \widetilde{\bdd{A}} \vec{u}_E }{\vec{u}_E^T \bar{\bdd{A}} \vec{u}_E } \leq \Delta \ \ \ \forall \bdd{u} \in \bdd{V}_0 \quad \text{and} \quad \theta \leq \frac{\vec{q}_e^T \widetilde{\bdd{B}} \widetilde{\bdd{A}}^{-1} \widetilde{\bdd{B}}^T \vec{q}_e}{ \vec{q}_e^T \bar{\bdd{M}} \vec{q}_e} \leq \Theta \ \ \ \forall q \in Q_0.
\end{align*}
Then, by \cite[Theorem 4.7 and eq. (4.37)]{Elman14}, the eigenvalues of $\bdd{P}^{-1} \bdd{S}$ lie in the set
\begin{align}
\label{eq:intro evals location}
\left[-\Theta^2, \frac{1}{2}\left( \delta - \sqrt{\delta^2 + 4\delta \theta^2}\right)  \right] \cup \{0\} \cup \left[ \delta, \frac{1}{2} \left( \Delta + \sqrt{\Delta^2 + 4\Delta \Theta^2} \right) \right].
\end{align}
In order to use variational techniques like Additive Schwarz Methods to construct $\bar{\bdd{A}}$ and $\bar{\bdd{M}}$, we must first identify the appropriate variational setting of the Schur complement system \cref{eq:Schur complement system}. In particular, the Schur complement is posed over the subspaces spanned by the external degrees of freedom of $\bdd{V}_0 \times Q_0$, which are rather non-standard owing to the additional continuity imposed at noncorner vertices. \Cref{sec:exact sequences} gives a precise characterization of these spaces including new results showing that they form a discrete exact sequence property (\cref{thm:full complex}) and that they, like the spaces $\bdd{V}_0 \times Q_0$, are uniformly inf-sup stable in both $h$ and $k$ (\cref{thm:inf-sup global boundary spaces}).

\Cref{sec:stokes extension} defines the Stokes extension operator and its relation to the subspace splittings. \Cref{sec:schur complement variational} uses the results of the previous two sections to relate the matrix form of the Schur complement system to a variational problem. In \cref{sec:basis functions}, we present an explicit set of basis functions on the reference
element for the spaces $\bdd{V}$ and $Q$ and then detail how these are used in
the construction of the global basis functions. We
develop the additive Schwarz theory and construct the matrices $\bar{\bdd{A}}$ and $\bar{\bdd{M}}$ in \cref{sec:asm theory}, which is then applied to two numerical examples demonstrating in \cref{sec:numerics}. \Cref{sec:technical lemmas} contains technical lemmas related to the additive Schwarz theory.

\section{Subspace Splittings, Exact Sequences, and Stability}
\label{sec:exact sequences}

A key property of the mixed finite element pair $\bdd{V}_0 \times Q_0$ is the exactness of the sequence \cite[Theorem 2.6]{AinCP19StokesI} and \cite[\S 3.2]{Falk13}:
\begin{align}
\label{eq:exact sequence full spaces}
0 \xrightarrow{ \ \ \ \subset\ \ \ } \Sigma_ {0} \xrightarrow{\ \ \vcurl \ \ } \bdd{V}_{0} \xrightarrow{ \ \ \dive \ \ } Q_{0} \xrightarrow{\qquad} 0,
\end{align}
where $\vcurl = (\partial_y, -\partial_x)^T$ and $\Sigma_0$ is the space of $H^2(\Omega)$-conforming piecewise polynomials (see  \cite[\S 2]{AinCP19StokesI}) given by
\begin{align*}
\Sigma_0 &= \{ \phi \in {H}_0^2(\Omega) : \phi|_{K} \in \mathcal{P}_{k+1}(K) \ \forall K \in \mathcal{T}, \ \phi \text{ is $C^2$ at noncorner vertices} \}.
\end{align*}
The exact sequence property \cref{eq:exact sequence full spaces} was used in \cite[Theorem 2.2]{AinCP19StokesII} to obtain optimal error estimates for the velocity that were independent of the pressure error. In the remainder of this section, we seek exact sequences analogous to \cref{eq:exact sequence full spaces} that respect the separation of interior and exterior degrees of freedom. Such sequences will be used to both identify the variational problem associated with the Schur complement system \cref{eq:Schur complement system} and prove its uniform stability.

Before we begin, we introduce some notation. Let $\mathcal{V}$ denote the set of all element vertices, and partition $\mathcal{V}$ into: $\mathcal{V}_{C}$, the set of element vertices located at a vertex of the polygonal domain $\Omega$; $\mathcal{V}_{B}$, the set of remaining element vertices on the domain boundary $\Gamma$ which are not corner vertices; and $\mathcal{V}_I$, the set of element vertices in the interior of domain $\Omega$. Let $\mathcal{E}$ be the set of all element edges. Given an element $K \in \mathcal{T}$, $\mathcal{E}_K$ denotes the edges of $K$ and $\mathcal{V}_K$ denotes the vertices of $K$. Likewise, given a vertex $\bdd{a} \in \mathcal{V}$, $\mathcal{E}_{\bdd{a}}$ denotes the set of edges having $\bdd{a}$ as an endpoint and $\mathcal{T}_{\bdd{a}}$ the set of elements having $\bdd{a}$ as a vertex. We assume that $\mathcal{T}$ is a partition of the domain $\Omega$ into triangles such that the nonempty intersection of any two distinct elements from $\mathcal{T}$ is either a single common vertex or a single common edge of both elements, and there exists $\kappa > 0$ independent of $\mathcal{T}$ such that
\begin{align}
\label{eq:shape regularity}
\rho_K \geq \kappa h_K \quad \forall K \in \mathcal{T},
\end{align}
where $h_K := \mathrm{diam}(K)$ and $\rho_K$ is the diameter of the largest inscribed circle of $K$. The mesh size $h$ denotes the diameter of the largest element, i.e. $h := \max_{K \in \mathcal{T}} h_K$.

\subsection{Interior Subspaces}

We first examine the subspaces associated with the interior degrees of freedom given by
\begin{align*}
\Sigma_I &= \{ \phi \in \Sigma_0 : \phi|_{\partial K} = \partial_n \phi|_{\partial K} = 0, \ \forall K \in \mathcal{T} \} \\
\bdd{V}_I &= \{ \bdd{v} \in \bdd{V}_0 : \bdd{v}|_{\partial K} = \bdd{0}, \ \forall K \in \mathcal{T} \} \\
Q_{I} &= \left\{ q \in Q_0 : \int_{K} q \ d\bdd{x} = 0, \ q|_{K}(\bdd{a}) = 0, \ \forall \bdd{a} \in \mathcal{V}_K, \ \forall K \in \mathcal{T} \right\},
\end{align*}
which, in turn, may be decomposed into contributions from individual elements:
\begin{align}
\label{eq:interior subspace element decomposition}
\Sigma_I = \bigoplus_{K \in \mathcal{T}} \Sigma_{I}(K), \quad
\bdd{V}_I = \bigoplus_{K \in \mathcal{T}} \bdd{V}_I(K), \quad \text{and} \quad \quad Q_{I} = \bigoplus_{K \in \mathcal{T}} Q_{I}(K),
\end{align}
where
\begin{align*}
\Sigma_I(K) &:= \mathcal{P}_{k+1}(K) \cap H^2_0(K), \qquad \bdd{V}_{I}(K) := \bm{\mathcal{P}}_{k}(K) \cap \bdd{H}^1_0(K), \\
Q_{I}(K) &:= \left\{ q \in \mathcal{P}_{k-1}(K) \cap L^2_0(K) : q(\bdd{a}) = 0, \ \forall \bdd{a} \in \mathcal{V}_K \right\}.
\end{align*}
Both the element-level interior spaces and the corresponding interior spaces on a mesh form exact sequences:
\begin{theorem}
	\label{thm:interior space element exact sequence}
	The following sequences are exact:
	\begin{align}
	\label{eq:exact sequence interior element} 
	0 \xrightarrow{ \ \ \ \subset\ \ \ } \Sigma_ {I}(K) \xrightarrow{\ \ \vcurl \ \ } \bdd{V}_{I}(K) \xrightarrow{ \ \ \dive \ \ } Q_{I}(K) \xrightarrow{\qquad} 0, \quad \forall K \in \mathcal{T},
	\end{align}
	and
	\begin{align}
	\label{eq:exact sequence interior global} 
	0 \xrightarrow{ \ \ \ \subset\ \ \ } \Sigma_ {I} \xrightarrow{\ \ \vcurl \ \ } \bdd{V}_{I} \xrightarrow{ \ \ \dive \ \ } Q_{I} \xrightarrow{\qquad} 0.
	\end{align}
\end{theorem}
\begin{proof}
	In view of $\vcurl H^2_0(K) \subset \bdd{H}^1_0(K)$, we have the relations $\vcurl \Sigma_{I}(K) \subset \bdd{V}_I(K)$,  $\vcurl \Sigma_I(K) \subseteq \ker \dive$, and by \cite[Theorem 3.4]{AinCP19StokesI}, $\dive \bdd{V}_I(K) = Q_{I}(K)$. Here, we consider div as a linear operator $\bdd{V}_I(K) \to Q_{I}(K)$. Moreover, for $\phi \in \Sigma_I(K)$, $\vcurl \phi \equiv 0$ if and only if $\phi \equiv 0$ and so $\dim \vcurl \Sigma_I(K) = \dim \Sigma_I(K)$. The dimension counts $\dim \Sigma_I(K) = \frac{k^2 - 7k + 12}{2}$, $\dim \bdd{V}_I(K) = k^2 - 3k + 2$, and $\dim Q_I(K) = \frac{k^2 + k - 8}{2}$
	reveal that $\dim \Sigma_I(K) + \dim Q_I(K) - \dim \bdd{V}_I(K) = 0$.
	By the rank-nullity theorem, we have
	\begin{align*}
	\dim \bdd{V}_I(K) = \dim \mathrm{Im} \dive + \dim \ker \dive \geq \dim Q_I(K) + \dim \Sigma_I(K) = \dim \bdd{V}_I(K).
	\end{align*}
	and so $\ker \dive = \vcurl \Sigma_I(K)$. Thus, the element-level sequence \cref{eq:exact sequence interior element} is exact. The exactness of the global spaces \cref{eq:exact sequence interior global} may be proved along similar lines using the exactness of the element level sequence \cref{eq:exact sequence interior element}.
\end{proof}
\Cref{thm:interior space element exact sequence} gives a useful decomposition of the interior spaces in terms of the $\vcurl$ operator:
\begin{corollary}
	\label{cor:V I decomposition}
	The spaces $\bdd{V}_I(K)$ and $\bdd{V}_I$ admit the following decompositions:
	\begin{align*}
	\label{eq:V I decomposition curl curl perp element}
	\bdd{V}_I(K) &= \vcurl \Sigma_I(K) \oplus \{ \bdd{v} \in \bm{\mathcal{P}}_{k-1}(K) \cap \bdd{H}^1_0(K), \ a(\bdd{v}, \vcurl \phi) = 0, \ \forall \phi \in \Sigma_I(K) \}
	\end{align*}
	and
	\begin{align}
	\label{eq:V I decomposition curl curl perp}
	\bdd{V}_I &= \vcurl \Sigma_I \oplus \{ \bdd{v} \in \bdd{V}_0 : \bdd{v}|_{\partial K} = \bdd{0}, \ \forall K \in \mathcal{T}, \ a(\bdd{v}, \vcurl \phi) = 0, \ \forall \phi \in \Sigma_I \}.
	\end{align}
\end{corollary}

The next result concerns the stability of the interior mixed finite element pair $\bdd{V}_{I} \times Q_{I}$.
\begin{theorem}
	\label{thm:inf-sup global interior spaces}
	Let $K \in \mathcal{T}$ and let $\beta$ be the discrete inf-sup constant appearing in \cref{eq:inf-sup global spaces}. If $q \in Q_{I}(K)$, then there exists $\bdd{v} \in \bdd{V}_I(K)$ such that $\dive \bdd{v} = q$ and
	\begin{align}
	\label{eq:invert div element}
	h_K^{-1} \|\bdd{v}\|_{\bdd{L}^2(K)} + |\bdd{v}|_{\bdd{H}^1(K)} \leq \beta^{-1} \|q\|_{L^2(K)}.
	\end{align}
	Consequently, (i) the spaces $\bdd{V}_I \times Q_{I}$ are uniformly inf-sup stable:
	\begin{align}
	\label{eq:inf-sup global interior spaces}
	\inf_{0 \neq q \in Q_{I}} \sup_{\bdd{0} \neq \bdd{v} \in \bdd{V}_I} \frac{b(\bdd{v}, q)}{|\bdd{v}|_{\bdd{H}^1(\Omega)} \|q\|_{L^2(\Omega)} } \geq \beta,
	\end{align}
	and (ii) the matrix $\begin{bmatrix}
	\bdd{A}_{II} & \bdd{B}_{I\iota} \\ 
	\bdd{B}_{\iota I} & \bdd{0}
	\end{bmatrix}$ is invertible, and hence the Schur complement $\bdd{S}$ appearing in \cref{eq:Schur complement system} is well-defined by the formula \cref{eq:schur complement matrix}.
\end{theorem}
\begin{proof}
	Let $q \in Q_{I}(K)$. Then, \cite[Theorem 3.5]{AinCP19StokesI} gives the existence of $\bdd{v} \in \bdd{V}_{I}(K)$ with $\dive \bdd{v} = q$ satisfying the estimate \cref{eq:invert div element}. 
	
	Now let $q \in Q_{I}$ be decomposed as in \cref{eq:interior subspace element decomposition} so that $q = \sum_{K \in \mathcal{T}} q_K$ with $q_K \in Q_{I}(K)$. By the first statement in the theorem, there exists $\bdd{v}_K \in \bdd{V}_I(K)$ with $\dive \bdd{v}_K = -q_K$ satisfying \cref{eq:invert div element}. Hence, $\bdd{v} := \sum_{K \in \mathcal{T}} \bdd{v}_K$ satisfies $\dive \bdd{v} = -q$ and
	\begin{align*}
	|\bdd{v}|_{\bdd{H}^1(\Omega)}^2 &= \sum_{K \in \mathcal{T}} |\bdd{v}_K|_{\bdd{H}^1(\Omega)}^2 \leq \beta^{-2} \sum_{K \in \mathcal{T}} \|q_K\|_{L^2(\Omega)}^2 = \beta^{-2} \|q\|_{L^2(\Omega)}^2,
	\end{align*}
	from which \cref{eq:inf-sup global interior spaces} immediately follows.
	(ii) follows at once thanks to the ellipticity of $a(\cdot, \cdot)$ and the inf-sup condition \cref{eq:inf-sup global interior spaces}.
\end{proof}

\subsection{Boundary Subspaces}

The subspaces $\tilde{\Sigma}_E$, $\tilde{\bdd{V}}_E$, and $\tilde{Q}_E$ are defined as follows
\begin{subequations}
	\label{eq:tilde spaces definition}
	\begin{align}
	\label{eq:tilde sigma e defintion}
	\tilde{\Sigma}_E &:= \{ \phi \in \Sigma_0 : a(\vcurl \phi, \vcurl \psi ) = 0, \ \forall \psi \in \Sigma_I \} \\
	\label{eq:tilde V E definition}
	\tilde{\bdd{V}}_E &:= \{ \bdd{v} \in \bdd{V}_0 : \dive \bdd{v} \in \tilde{Q}_E, \ a(\bdd{v}, \vcurl \psi) = 0, \ \forall \psi \in \Sigma_I \} \\
	\label{eq:tilde q e definition}
	\tilde{Q}_E &:= \{ q \in Q_0 : (q, r) = 0, \ \forall r \in Q_{I} \},
	\end{align}
\end{subequations}
and correspond to degrees of freedom on the element boundaries. More precisely, we have:
\begin{theorem}
	\label{thm:tilde interior decomposition}
	The following decompositions hold:
	\begin{align}
	\label{eq:tilde interior decomposition}
	\Sigma_0 = \Sigma_I \oplus \tilde{\Sigma}_E, \quad \bdd{V}_0 = \bdd{V}_I \oplus \tilde{\bdd{V}}_E, \quad \text{and} \quad Q_0 = Q_{I} \oplus \tilde{Q}_E.
	\end{align}
\end{theorem}
\begin{proof}
	The decompositions of $\Sigma_{0}$ and $Q_0$ follow immediately using the orthogonality conditions in the definition of the spaces \cref{eq:tilde sigma e defintion,eq:tilde q e definition}. The decomposition of the velocity space $\bdd{V}_0$ is more involved. We first use the exact sequence \cref{eq:exact sequence full spaces} to write:
	\begin{align}
	\label{eq:v0 subspace splitting curl curl perp}
	\bdd{V}_0 &= \vcurl \Sigma_0 \oplus (\vcurl \Sigma_0)^{\perp},
	\end{align}
	where $(\vcurl \Sigma_0)^{\perp} := \{ \bdd{u} \in \bdd{V}_0 : a(\bdd{u}, \vcurl \psi ) = 0, \ \forall \psi \in \Sigma_0 \}.$
	Now, let $\bdd{v} \in \bdd{V}_0$ be given. By the decomposition \cref{eq:v0 subspace splitting curl curl perp} and the decomposition of $\Sigma_0$ in \cref{eq:tilde interior decomposition}, there exists $\phi_I \in \Sigma_I$, $\tilde{\phi}_E \in \tilde{\Sigma}_E$, and $\bdd{v}_{\perp} \in (\vcurl \Sigma_0)^{\perp}$ such that $\bdd{v} = \vcurl (\phi_I + \tilde{\phi}_E) + \bdd{v}_{\perp}$. We decompose the divergence analogously: $\dive \bdd{v} = q_{I} + \tilde{q}_E$ with $q_{I} \in Q_{I}$ and $\tilde{q}_E \in \tilde{Q}_E$. Thanks to the exact sequence \cref{eq:exact sequence interior global} and the decomposition \cref{eq:V I decomposition curl curl perp}, there exists $\bdd{w} \in \bdd{V}_I$ such that $\dive \bdd{w} = q_{I}$ and $\bdd{w} \in \{ \bdd{v} \in \bdd{V}_I : \ a(\bdd{v}, \vcurl \psi) = 0, \ \forall \psi \in \Sigma_I \}$. Then, $\bdd{v}_I := \vcurl \phi_I + \bdd{w}$ satisfies $\bdd{v}_I \in \bdd{V}_I$ and $\dive \bdd{v}_I = q_{I}$. Consequently, $\tilde{\bdd{v}}_E := \bdd{v} - \bdd{v}_I = \vcurl \tilde{\phi}_E + \bdd{v}_{\perp} - \bdd{w}$ satisfies $\dive \tilde{\bdd{v}}_E = \dive (\bdd{v} - \bdd{v}_I) = \tilde{q}_E \in \tilde{Q}_E$ and
	\begin{align*}
	a(\tilde{\bdd{v}}_E, \vcurl \psi) = a(\vcurl \tilde{\phi}_E, \vcurl \psi) + a(\bdd{v}_{\perp}, \vcurl \psi) + a(\bdd{w}, \vcurl \psi) = 0, \ \forall \psi \in \Sigma_I
	\end{align*} 
	by construction. Thus, $\tilde{\bdd{v}}_E \in \tilde{\bdd{V}}_E$, which completes the proof.
\end{proof}

\cref{thm:tilde interior decomposition} means that the decompositions in the columns of the following complex \cref{eq:full complex} are valid. The next result shows that the rows form exact sequences:
\begin{theorem}
	\label{thm:full complex}
	Each row the of the following complex is an exact sequence
	\begin{subequations}
	\label{eq:full complex}
	\begin{alignat}{9}
	\label{eq:exact sequence full spaces 2}
	&0 \quad&& \xrightarrow{ \ \ \ \subset\ \ \ } \quad && \Sigma_ {0} \quad && \xrightarrow{\ \ \vcurl \ \ } \quad && \bdd{V}_{0} \quad && \xrightarrow{ \ \ \dive \ \ } \quad && Q_{0} \quad && \xrightarrow{\qquad} \quad & 0 \\
	& && && \shortparallel_{\phantom{0}} && && \shortparallel_{\phantom{0}} && && \shortparallel_{\phantom{0}} && && \notag \\
	\label{eq:exact sequence interior spaces 2}
	&0 && \xrightarrow{ \ \ \ \subset\ \ \ } && \Sigma_ {I} && \xrightarrow{\ \ \vcurl \ \ } && \bdd{V}_{I} && \xrightarrow{ \ \ \dive \ \ } && Q_{I} && \xrightarrow{\qquad} && 0 \\
	&  && && \oplus_{\phantom{0}} && && \oplus_{\phantom{0}} && && \oplus_{\phantom{0}} && && \notag \\
	\label{eq:exact sequence exterior spaces}
	&0 && \xrightarrow{ \ \ \ \subset\ \ \ } && \tilde{\Sigma}_ {E} && \xrightarrow{\ \ \vcurl \ \ } && \tilde{\bdd{V}}_{E} && \xrightarrow{ \ \ \dive \ \ } && \tilde{Q}_{E} && \xrightarrow{\qquad} && 0 
	\end{alignat}
	\end{subequations}
	where the exterior spaces $\tilde{\Sigma}_E$, $\tilde{\bdd{V}}_E$, and $\tilde{Q}_E$ are given by \cref{eq:tilde spaces definition}.
\end{theorem}
\begin{proof}
\cite[Theorem 2.6]{AinCP19StokesI} gives the exactness of \cref{eq:exact sequence full spaces 2}
	while \cref{thm:interior space element exact sequence} gives the exactness of \cref{eq:exact sequence interior spaces 2}. Moreover, the decomposition \cref{eq:tilde interior decomposition} and the exactness the sequences \cref{eq:exact sequence full spaces 2,eq:exact sequence interior spaces 2}  imply that $\dim \tilde{\Sigma}_E + \dim \tilde{Q}_E - \dim \tilde{\bdd{V}}_E = 0.$ Since $\vcurl \tilde{\Sigma}_E \subset \tilde{\bdd{V}}_E$ and $\dive \tilde{\bdd{V}}_E \subseteq \tilde{Q}_E$, we conclude that the sequence \cref{eq:exact sequence exterior spaces} is exact using analogous arguments to those used in \cref{thm:interior space element exact sequence}.
\end{proof}
The exactness of the final row in \cref{eq:full complex} gives the following analogue of \cref{cor:V I decomposition} for the exterior velocity space:
\begin{corollary}
	\label{cor:tilde V E decomposition}
	The exterior velocity space $\tilde{\bdd{V}}_E$ admits the following decomposition: $
	\tilde{\bdd{V}}_E = \vcurl \tilde{\Sigma}_{E} \oplus \{ \bdd{v} \in \bdd{V}_0 : \dive \bdd{v} \in \tilde{Q}_E, \ a(\bdd{v}, \vcurl \phi) = 0, \ \forall \phi \in \Sigma_0 \}$.
\end{corollary}

\Cref{thm:inf-sup global spaces,thm:inf-sup global interior spaces} show that the mixed finite element pairs appearing in the first two rows of \cref{eq:full complex} are uniformly inf-sup stable. The next result shows that the boundary spaces are also stable with \textit{the same inf-sup constant as for the full velocity and pressure spaces}:
\begin{theorem}
	\label{thm:inf-sup global boundary spaces}
	Let $\beta$ be the discrete inf-sup constant defined in \cref{eq:inf-sup global spaces}. If $q \in \tilde{Q}_E$, then there exists a $\bdd{v} \in \tilde{\bdd{V}}_E$ such that $\dive \bdd{v} = q$ and
	\begin{align}
	\label{eq:right continuous inverse boundary space}
	|\bdd{v}|_{\bdd{H}^1(\Omega)} \leq \beta^{-1} \|q \|_{L^2(\Omega)}.
	\end{align}
	Consequently, the spaces $\tilde{\bdd{V}}_{E} \times \tilde{Q}_E $ are uniformly inf-sup stable:
	\begin{align}
	\label{eq:inf-sup global boundary spaces}
	\inf_{0 \neq q \in \tilde{Q}_E} \sup_{\bdd{0} \neq \bdd{v} \in \tilde{\bdd{V}}_{E}} \frac{b(\bdd{v}, q)}{|\bdd{v}|_{\bdd{H}^1(\Omega)} \|q\|_{L^2(\Omega)} } \geq \beta.
	\end{align}
\end{theorem}
\begin{proof}
	Let $q \in \tilde{Q}_E$ be given. By \cref{thm:inf-sup global spaces}, there exists a $\bdd{w} \in \bdd{V}_0$ such that $\dive \bdd{w} = q $ and $\|\bdd{w}\|_{\bdd{H}^1(\Omega)} \leq \beta^{-1} \|q\|_{L^2(\Omega)}$, where $\beta$ is independent of $h$ and $k$. According to \cref{thm:tilde interior decomposition,eq:tilde V E definition}, there exists functions $\bdd{w}_I \in \bdd{V}_I$, $\tilde{\bdd{w}}_E \in \tilde{\bdd{V}}_E$ such that $\bdd{w} = \bdd{w}_I + \tilde{\bdd{w}}_E$. $Q_{I} \ni \dive \bdd{w}_I = \dive (\bdd{w} - \tilde{\bdd{w}}_E) \in \tilde{Q}_E$
	since $\dive \bdd{w} = q \in \tilde{Q}_E$, and so $\dive \bdd{w}_I = 0$. By the exact sequence property \cref{eq:exact sequence interior global}, $\bdd{w}_I = \vcurl \phi_I$ for some $\phi \in \Sigma_I$, and thus $\bdd{w} = \vcurl \phi_I + \tilde{\bdd{w}}_E$. Note that this decomposition of $\bdd{w}$ is $a(\cdot, \cdot)$ orthogonal by definition: $a(\vcurl \phi_I, \tilde{\bdd{w}}_E) = 0$ and
	\begin{align*}
	%\label{eq:w norm decomposition}
	|\bdd{w}|_{H^1(\Omega)}^2 = a(\vcurl \phi_I, \vcurl \phi_I) + a(\tilde{\bdd{w}}_E, \tilde{\bdd{w}}_E) = |\vcurl \phi_I|_{H^1(\Omega)}^2 + |\tilde{\bdd{w}}_E|_{H^1(\Omega)}^2.
	\end{align*}	
	Define $\bdd{v} := \tilde{\bdd{w}}_E$. Then, $
	\dive \bdd{v} = \dive \tilde{\bdd{w}}_E = \dive (\tilde{\bdd{w}}_E + \vcurl \phi_I) = \dive \bdd{w} = q$,
	and $|\bdd{v}|_{H^1(\Omega)} \leq |\bdd{w}|_{H^1(\Omega)} \leq \beta^{-1} \|q\|_{L^2(\Omega)}$.
	\cref{eq:inf-sup global boundary spaces,eq:right continuous inverse boundary space} follow at once.
\end{proof}

\section{Stokes Extension Operator}
\label{sec:stokes extension}

Let $\bdd{V} := V \times V$ denote the discrete velocity space in the absence of essential boundary conditions and $Q_{I}^{\perp}$ be the orthogonal complement of $Q_I$ in $Q$ with the corresponding projection $\tilde{\Pi} : Q \to Q_I^{\perp}$, 
\begin{align}
\label{eq:tilde pi definition}
(\tilde{\Pi} q, r) = (q, r), \ \forall r \in Q_I^{\perp} := \{ q \in Q : (q, r) = 0 \ \forall r \in Q_I \}.
\end{align}
It is worthwhile noting that \cref{eq:tilde pi definition} means that $\tilde{Q}_E = Q_I^{\perp} \cap L^2_0(\Omega)$, so that the space $Q_I^{\perp}$ corresponds to boundary degrees of freedom. Let $K \in \mathcal{T}$. Then, thanks to \cref{thm:inf-sup global interior spaces}, there exist $\bdd{u}_{S,K} \in \bm{\mathcal{P}}_{k}(K)$ and $p_{S,K} \in \mathcal{P}_{k-1}(K)$ satisfying
\begin{subequations}
	\label{eq:stokes extension u and p}
	\begin{align}
	a_K(\bdd{u}_{S,K}, \bdd{v}) + b_K(\bdd{v}, p_{S,K}) &= 0 & &\forall \bdd{v} \in \bdd{V}_I(K) \label{eq:stokes extension u and p 1} \\
	b_K(\bdd{u}_{S,K}, q) &= 0 & &\forall q \in Q_{I}(K) \label{eq:stokes extension u and p 2} \\
	\bdd{u}_{S,K} &= \bdd{u} & &\text{on } \partial K \label{eq:stokes extension u and p 3} \\
	p_{S,K}(\bdd{a}) &= p|_{K}(\bdd{a}) & &\bdd{a} \in \mathcal{V}_K  \label{eq:stokes extension u and p 4} \\
	\int_{K} p_{S,K} \ d\bdd{x} &= \int_{K} p \ d\bdd{x}, \label{eq:stokes extension u and p 5}
	\end{align}
\end{subequations}
where $a_K(\cdot,\cdot)$ and $b_K(\cdot, \cdot)$ denote the restrictions of the bilinear forms to element $K$. We define the Stokes extension map $\bdd{V} \times Q \ni (\bdd{u}, p) \mapsto \mathscr{E}(\bdd{u}, p) =: (\bdd{u}_S, p_S)$ by the rule $\bdd{u}_{S} = \bdd{u}_{S,K}$ and $p_S = p_{S,K}$ on each element $K \in \mathcal{T}$. 

The first result deals with the Stokes extension of a given velocity field paired  a zero pressure:
\begin{theorem}
	\label{thm:extension velocity continuity}
	Let $\bdd{\Pi}_{\bdd{V}} : \bdd{V} \to \bdd{V}$, $\Pi_{Q} : \bdd{V} \to Q$ be defined by the rule
	\begin{align}
	\label{eq:pi dagger definition}
	\bdd{V} \ni \bdd{u} \mapsto (\bdd{\Pi}_{\bdd{V}} \bdd{u}, \Pi_{Q} \bdd{u}) := \mathscr{E}(\bdd{u}, 0). 
	\end{align}
	Then, $\Pi_{Q} \bdd{u} \in Q_I$ and 
	\begin{align}
	\label{eq:stokes extension continuity u dagger}
	\|\bdd{\Pi}_{\bdd{V}} \bdd{u}\|_{\bdd{H}^1(K)} +  \|\Pi_{Q} \bdd{u} \|_{L^2(K)} \leq C  \|\bdd{u}\|_{\bdd{H}^{1/2}(\partial K)}, \quad \forall K \in \mathcal{T},
	\end{align}
	where $\|\cdot\|_{\bdd{H}^{1/2}(\partial K)}$ is the usual trace norm and $C$ is independent of $k$ and $\bdd{u}$. In particular, if $\bdd{u} \in \tilde{\bdd{V}}_E$, then $\bdd{\Pi}_{\bdd{V}} \bdd{u} = \bdd{u}$. Moreover, the following equivalence of semi-norms holds:
	\begin{align}
	\label{eq:stokes ext equiv harmonic ext}
	|\bdd{u}|_{\bdd{H}^{1/2}(\partial K)} \leq |\bdd{\Pi}_{\bdd{V}} \bdd{u}|_{\bdd{H}^{1}(K)} \leq C \beta^{-1} |\bdd{u}|_{\bdd{H}^{1/2}(\partial K)}, \quad \forall K \in \mathcal{T},
	\end{align}
	where $C$ is independent of $k$, $h_K$, $\beta$, and $\bdd{u}$. 		
\end{theorem}
\begin{proof}
	Let $K \in \mathcal{T}$ and $\bdd{u} \in \bdd{V}$ be given. Conditions \cref{eq:stokes extension u and p 4,eq:stokes extension u and p 5} imply that $\Pi_{Q} \bdd{u} \in Q_{I}$. Thanks to \cite[Theorem 7.4]{BCMP91}, there exists $\bdd{w} \in \bm{\mathcal{P}}_{k}(K)$ such that
	\begin{align}
	\label{eq:stokes extension proof w}
	\bdd{w}|_{\partial K} = \bdd{u}|_{\partial K} \quad \text{and} \quad \|\bdd{w}\|_{\bdd{H}^{1}(K)} \leq C \|\bdd{u}\|_{\bdd{H}^{1/2}(\partial K)},
	\end{align}
	with $C$ independent of $k$.
	In particular, $\bdd{\Pi}_{\bdd{V}} \bdd{u} - \bdd{w} = \bdd{u}_I$ where $\bdd{u}_I \in \bdd{V}_I(K)$ satisfies
	\begin{align*}
	\begin{aligned}
	a_K(\bdd{u}_{I}, \bdd{v}) + b_K(\bdd{v}, \Pi_{Q} \bdd{u}) &= -a_K(\bdd{w}, \bdd{v}) & &\forall \bdd{v} \in \bdd{V}_I(K) \\ %\label{eq:stokes extension2 u and p 1} \\
	b_K(\bdd{u}_{I}, q) &= -b_K(\bdd{w}, q) & &\forall q \in Q_{I}(K). %\label{eq:stokes extension2 u and p 2}
	\end{aligned}
	\end{align*}
	Using \cite[Corollary 4.1]{Gir12} and \cref{thm:inf-sup global interior spaces}, we conclude that
	\begin{align*}
	\|\bdd{u}_I\|_{\bdd{H}^1(K)} + \|\Pi_{Q} \bdd{u} \|_{L^2(K)} \leq C \|\bdd{w}\|_{\bdd{H}^1(K)}
	%\|\bdd{\Pi}_{\bdd{V}} \bdd{u}\|_{\bdd{H}^1(K)} +  \|\Pi_{Q} \bdd{u} \|_{L^2(K)} &\leq 
	%\|\bdd{u}_I\|_{\bdd{H}^1(K)} + \|\bdd{w}\|_{\bdd{H}^1(K)} + \|\Pi_{Q} \bdd{u} \|_{L^2(K)} \leq C \|\bdd{w}\|_{\bdd{H}^1(K)}.
	\end{align*}
	\Cref{eq:stokes extension continuity u dagger} now follows from the triangle inequality and \cref{eq:stokes extension proof w}.
	
	Now let $\bdd{u} \in \tilde{\bdd{V}}_E$. Then, $\dive \bdd{\Pi}_{\bdd{V}} \bdd{u} \in \tilde{Q}_E$ by \cref{eq:stokes extension u and p 2} and $\bdd{w} := \bdd{u} -\bdd{\Pi}_{\bdd{V}} \bdd{u} \in \bdd{V}_I$ by \cref{eq:stokes extension u and p 3}. Moreover, $\dive \bdd{w} \in Q_I \cap \tilde{Q}_E$ since $\dive \bdd{u}, \dive \bdd{\Pi}_{\bdd{V}} \bdd{u} \in \tilde{Q}_E$. Thus, $\dive \bdd{w} = 0$ and $\bdd{w} = \vcurl \phi$ with $\phi \in \Sigma_I$ by the exact sequence property \cref{eq:exact sequence interior global}. For any $\psi \in \Sigma_I$,
	\begin{align*}
	a(\vcurl \phi, \vcurl \psi) = a(\bdd{u}, \vcurl \psi) - a(\bdd{\Pi}_{\bdd{V}}\bdd{u}, \vcurl \psi) = 0
	\end{align*}
	by the definition of $\tilde{\bdd{V}}_E$ and choosing $\bdd{v} = \vcurl \psi$ for $\psi \in \Sigma_I$ in \cref{eq:stokes extension u and p 1}. Since $a(\vcurl \cdot, \vcurl \cdot)$ is coercive on $\Sigma_I$, $\phi \equiv 0$, and $\bdd{u} = \bdd{\Pi}_{\bdd{V}} \bdd{u}$. 
	
	The equivalence \cref{eq:stokes ext equiv harmonic ext} is proved by arguing as in \cite[Theorem 4.1]{Bram90}. 
\end{proof}
The next result complements \cref{thm:extension velocity continuity}:
\begin{theorem}
	\label{thm:extension pressure continuity}
	For $p \in Q$, $\mathscr{E}(\bdd{0}, p) = (\bdd{0}, \tilde{\Pi} p)$ where $\tilde{\Pi}$ is defined in \cref{eq:tilde pi definition}, and
	\begin{align}
	\label{eq:stokes extension continuity tilde p}
	\|\tilde{\Pi} p \|_{L^2(K)} \leq \|p\|_{L^2(K)}.
	\end{align}	
	In particular, if $p \in Q_I^{\perp}$, then $\mathscr{E}(\bdd{0}, p) = (\bdd{0}, p)$.
\end{theorem}
\begin{proof}
	Let $p \in Q$ and consider the Stokes extension $(\tilde{\bdd{u}}, \tilde{p}) := \mathscr{E}(\bdd{0}, p)$. Since $Q = Q_I \oplus Q_{I}^{\perp}$, the pressure $\tilde{p}$ may be written in the form $\tilde{p} = p_{I} + \tilde{\Pi} p$. In particular, $p_{I} \in Q_{I}$ satisfies
	\begin{align*}
	b_K(\bdd{v}, p_{I}) = b_K(\bdd{v}, \tilde{p}) - b_K(\bdd{v}, \tilde{\Pi} p) = b_K(\bdd{v}_K, \tilde{p}), \quad \forall \bdd{v} \in \bdd{V}_I(K), \ \forall K \in \mathcal{T},
	\end{align*}
	where we used the fact that $b_K(\bdd{v}, \tilde{\Pi} p) = 0$ since $\dive \bdd{V}_I(K) = Q_{I}(K) \perp Q_{I}^{\perp}$.
	Hence,
	\begin{align*}
	\begin{aligned}
	a_K(\tilde{\bdd{u}}, \bdd{v}) + b_K(\bdd{v}, p_{I}) &= 0& &\forall \bdd{v} \in \bdd{V}_I(K) \\ %\label{eq:stokes extension2 u and p 1} \\
	b_K(\tilde{\bdd{u}}, q) &= 0 & &\forall q \in Q_{I}(K), 		 %\label{eq:stokes extension2 u and p 2}
	\end{aligned}
	\end{align*}
	\Cref{eq:stokes extension continuity u dagger} then gives $(\tilde{\bdd{u}}, p_{I}) = \mathscr{E}(\bdd{0}, 0)$; or, equally well, $\tilde{\bdd{u}} = \bdd{0}$ and $\tilde{p} = \tilde{\Pi} p$. The estimate \cref{eq:stokes extension continuity us ps} immediately follows since $\tilde{\Pi}$ is a projection. If $p \in Q_{I}^{\perp}$, then $\mathscr{E}(\bdd{0}, p) = (\bdd{0}, \tilde{\Pi} p)= (\bdd{0}, p)$.
\end{proof}
Combining \cref{thm:extension velocity continuity,thm:extension pressure continuity} leads to the following result:
\begin{corollary}
	\label{cor:extension continuity element}
	The Stokes extension operator $\mathscr{E}(\cdot, \cdot)$ is linear and continuous: For $K \in \mathcal{T}$, 
	\begin{align}
	\label{eq:stokes extension continuity us ps}
	\|\mathscr{E}(\bdd{u}, p)\|_{\bdd{H}^1(K) \times L^2(K)} \leq C  \|\bdd{u}\|_{\bdd{H}^{1/2}(\partial K)} +  \|\tilde{\Pi} p \|_{L^2(K)} \quad \forall (\bdd{u}, p) \in \bdd{V} \times Q,
	\end{align}
	where $C$ is independent of $k$. Moreover, $\ker \mathscr{E} = \bdd{V}_I \times Q_I$ and $\mathscr{E}(\bdd{u}, p) = (\bdd{\Pi}_{\bdd{V}} \bdd{u}, \Pi_{Q} \bdd{u} + \tilde{\Pi} p)$.
\end{corollary}
\begin{proof}
	The linearity of $\mathscr{E}(\cdot, \cdot)$ is immediate from the definition \cref{eq:stokes extension u and p}, and \cref{eq:stokes extension continuity us ps} then follows from \cref{eq:stokes extension continuity u dagger,eq:stokes extension continuity tilde p} using the triangle inequality. A simple consequence of \cref{eq:stokes extension u and p 3,eq:stokes extension u and p 4,eq:stokes extension u and p 5} is that $\ker \mathscr{E} \subseteq \bdd{V}_I \times Q_I$. Moreover, \cref{eq:stokes extension continuity us ps} gives that $\bdd{V}_I \times Q_I \subseteq \ker \mathscr{E}$. Thus, $\ker \mathscr{E} = \bdd{V}_I \times Q_{I}$.
\end{proof}

\section{Variational Form of the Schur Complement System}
\label{sec:schur complement variational}

The results of the previous two sections are used to study the Schur complement system \cref{eq:Schur complement system}. The first result relates the Schur complement matrix $\bdd{S}$ to the discrete Stokes extension map:
\begin{lemma}
	\label{lem:matrix relations}
	For all $(\bdd{u}, p), (\bdd{v}, q) \in \bdd{V}_0 \times Q$, the Stokes extension satisfies
	\begin{align}
	\label{eq:inner product matrix}
	a(\bdd{u}_{S}, \bdd{v}_{S}) + b(\bdd{v}_{S}, p_{S}) + b(\bdd{u}_{S}, q_{S}) = \begin{bmatrix}
	\vec{v}_E \\ \vec{q}_e \end{bmatrix}^T
	\bdd{S}
	\begin{bmatrix}
	\vec{u}_E \\ \vec{p}_{e} \end{bmatrix}
	\end{align}
	where $\bdd{S} = \begin{bmatrix}
	\widetilde{\bdd{A}} & \widetilde{\bdd{B}}^T \\
	\widetilde{\bdd{B}} & \widetilde{\bdd{C}} \end{bmatrix}$. Consequently, the following identities hold:
	\begin{subequations}
		\label{eq:matrix relations}
		\begin{align}
		\vec{u}_E^{T} \widetilde{\bdd{A}} \vec{v}_E &= a(\bdd{\Pi}_{\bdd{V}} \bdd{u}, \bdd{\Pi}_{\bdd{V}} \bdd{v}) & & \forall \bdd{u}, \bdd{v} \in \bdd{V}_0, \label{eq:matrix relations 1} \\
		\vec{p}_{e}^{T} \widetilde{\bdd{B}} \vec{u}_E  &= b(\bdd{\Pi}_{\bdd{V}} \bdd{u}, \tilde{\Pi} p) & & \forall p \in Q, \ \bdd{u} \in \bdd{V}_0, \label{eq:matrix relations 2} \\
		\widetilde{\bdd{C}} &= \bdd{0} \quad \text{and} \quad \vec{g}_{e}^{*} = \vec{0}, \label{eq:matrix relations 3}
		\end{align}	
	\end{subequations}
	where $\bdd{\Pi}_{\bdd{V}}$ is defined in \cref{eq:pi dagger definition} and $\tilde{\Pi}$ is defined in \cref{eq:tilde pi definition}.
\end{lemma}
\begin{proof}
	Let $(\bdd{u}, p) \in \bdd{V}_0 \times Q$. The Stokes extension $(\bdd{u}_S, p_S) = \mathcal{E}(\bdd{u}, p)$ may be written as $\bdd{u}_{S} = \vec{\Phi}_{E}^T \vec{u}_E + \vec{\Phi}_{I}^{T} \vec{u}_{I}^{*}$ and $ p_{S} = \vec{\psi}_{e}^{T} \vec{p}_e + \vec{\psi}_{\iota}^{T} \vec{p}_{\iota}^{*}$ so that $\bdd{u} = \vec{\Phi}_E^T \vec{u}_E + \vec{\Phi}_I^T \vec{u}_I$ and $p = \vec{\psi}_e^T \vec{p}_e + \vec{\psi}_{\iota}^T \vec{p}_{\iota}$ for suitable $\vec{u}_E$, $\vec{u}_I$, $\vec{p}_e$, and $\vec{p}_{\iota}$. Thanks to \cref{eq:stokes extension u and p 1,eq:stokes extension u and p 2}, $\vec{u}_I^{*}$ and $\vec{p}_{\iota}^{*}$ are given by
	\begin{align}
	% check if this is already defined
	\label{eq:dagger interior relations}
	\begin{bmatrix}
	\vec{u}_I^{*} \\ \vec{p}_{\iota}^{*}
	\end{bmatrix} = - \begin{bmatrix}
	\bdd{A}_{II} & \bdd{B}_{I \iota} \\
	\bdd{B}_{\iota I} & \bdd{0}
	\end{bmatrix}^{-1} \begin{bmatrix}
	\bdd{A}_{IE} & \bdd{B}_{Ie} \\
	\bdd{B}_{\iota E} & \bdd{0}
	\end{bmatrix}
	\begin{bmatrix}
	\vec{u}_E \\ \vec{p}_{e}
	\end{bmatrix}.
	\end{align}
	Analogous relations hold replacing $(\bdd{u}, p)$ by $(\bdd{v}, q) \in \bdd{V}_0 \times Q$. Now,
	\begin{align*}
	&a(\bdd{u}_{S}, \bdd{v}_{S}) + b(\bdd{v}_{S}, p_{S}) + b(\bdd{u}_{S}, q_{S}) \\
	&\qquad = \begin{bmatrix}
	\begin{array}{c}
	\vec{v}_E \\ \vec{q}_e \\ \hline \vec{v}_I^{*} \\ \vec{q}_{\iota}^{*}
	\end{array}
	\end{bmatrix}^{T} 
	\begin{bmatrix}
	\begin{array}{cc|cc}
	\bdd{A}_{EE} & \bdd{B}_{E e} & \bdd{A}_{EI} & \bdd{B}_{E\iota} \\
	\bdd{B}_{e E} & \bdd{0} &  \bdd{B}_{e I} & \bdd{0}\\
	\hline 
	\bdd{A}_{IE} & \bdd{B}_{I e} &  \bdd{A}_{II} & \bdd{B}_{I \iota} \\
	\bdd{B}_{\iota E} & \bdd{0} & \bdd{B}_{\iota I} & \bdd{0}
	\end{array}
	\end{bmatrix}
	\begin{bmatrix}
	\begin{array}{c}
	\vec{u}_E \\ \vec{p}_e \\ \hline \vec{u}_I^{*} \\ \vec{p}_{\iota}^{*}
	\end{array}
	\end{bmatrix}.
	\end{align*}
	and then \cref{eq:dagger interior relations}, we obtain \cref{eq:inner product matrix}.
	Identities \cref{eq:matrix relations} are then obtained from \cref{eq:inner product matrix} as follows:
	\begin{enumerate}
		\item[(a)] Choose $p = q = 0$. Then, $(\bdd{\Pi}_{\bdd{V}} \bdd{u}^{\dagger}, \Pi_{Q} \bdd{u}), ( \bdd{\Pi}_{\bdd{V}} \bdd{v}, \Pi_{Q} \bdd{v}) \in \bdd{V} \times Q_{I}$ by \cref{thm:extension velocity continuity}, so $b(\bdd{\Pi}_{\bdd{V}} \bdd{v}, \Pi_{Q} \bdd{u}) + b(\bdd{\Pi}_{\bdd{V}} \bdd{u}, \Pi_{Q} \bdd{v}) = 0$ by \cref{eq:stokes extension u and p 2};  \cref{eq:matrix relations 1} follows.
		
		\item[(b)] Choose $p = 0$ and $\bdd{v} = \bdd{0}$. By \cref{thm:extension pressure continuity}, $\mathscr{E}(\bdd{0}, q) = (\bdd{0}, \tilde{\Pi} q)$, and \cref{eq:matrix relations 2} follows.
		
		\item[(c)] Choose $\bdd{u} = \bdd{v} = \bdd{0}$. By \cref{thm:extension pressure continuity}, $\mathscr{E}(\bdd{0}, p) = (\bdd{0}, \tilde{\Pi} p)$, $\mathscr{E}(\bdd{0}, q) = (\bdd{0}, \tilde{\Pi} q)$, and so $
		\vec{q}_e^T \widetilde{\bdd{C}} \vec{p}_e = 0$. Furthermore, 
		\begin{align*}
		\vec{q}_e^{T} \vec{g}_e^{*} = \vec{q}_e^{T} \widetilde{\bdd{B}} \vec{u}_{E} = b(\bdd{\Pi}_{\bdd{V}} \bdd{u}, \tilde{\Pi} q) = -b(\bdd{u} - \bdd{\Pi}_{\bdd{V}} \bdd{u}, \tilde{\Pi} q), \quad \forall {q} \in {Q}
		\end{align*}
		by \cref{eq:discrete stokes system 2}. Since $\dive(\bdd{u} - \bdd{\Pi}_{\bdd{V}} \bdd{u}) \in \dive \bdd{V}_I = Q_{I} \perp Q_I^{\perp}$, $b(\bdd{u} - \bdd{\Pi}_{\bdd{V}} \bdd{u}, \tilde{q}) = 0$, which completes \cref{eq:matrix relations 3}.
	\end{enumerate}
\end{proof}

The main result of this section relates the Schur complement problem \cref{eq:Schur complement system} to a Stokes problem posed on the boundary spaces $\tilde{\bdd{V}}_E \times \tilde{Q}_E$:
\begin{theorem}
	\label{thm:schur complement variational form}
	The Schur complement system \cref{eq:Schur complement system}, is equivalent to the following variational problem: Find $(\bdd{u}, p) \in \tilde{\bdd{V}}_E \times \tilde{Q}_E$ such that
	\begin{subequations}
		\label{eq:variational form schur complement}
		\begin{align}
		a(\bdd{u}, \bdd{v}) + b(\bdd{v}, p) &= (\bdd{f}, \bdd{v}) & &\forall \bdd{v} \in \tilde{\bdd{V}}_E \\
		b(\bdd{u}, {q}) &= 0 & &\forall {q} \in \tilde{Q}_E.
		\end{align}
	\end{subequations}
	Moreover, the nonzero eigenvalues of the generalized eigenvalue problem $\widetilde{\bdd{B}} \widetilde{\bdd{A}}^{-1} \widetilde{\bdd{B}}^T \vec{q}_e$ $= \lambda \widetilde{\bdd{M}} \vec{q}_e$ are contained in the interval $[\beta^2, 1]$, where $\widetilde{\bdd{M}}$ is the matrix associated with the $L^2(\Omega)$-inner product on ${Q}_I^{\perp}$ and $\beta$ is the inf-sup constant in \cref{eq:inf-sup global spaces}. In particular, the nonzero eigenvalues $\lambda$ are uniformly bounded away from zero in $h$ and $k$.
\end{theorem}
\begin{proof}
	Let $(\bdd{u}, p), (\bdd{v}, q) \in \tilde{\bdd{V}}_E \times Q_I^{\perp}$. Substituting the identities in \cref{thm:extension velocity continuity,thm:extension pressure continuity} into \cref{eq:matrix relations} gives
	\begin{align*}
	\begin{bmatrix}
	\vec{v}_E \\ \vec{q}_e \end{bmatrix}^T
	\bdd{S}
	\begin{bmatrix}
	\vec{u}_E \\ \vec{p}_{e} \end{bmatrix} &= \vec{v}_E^T \widetilde{\bdd{A}} \vec{u}_E^T + \vec{q}_e^T \widetilde{\bdd{B}} \vec{u}_E + \vec{v}_E \widetilde{\bdd{B}}^T \vec{p}_e  \\
	&= a(\bdd{\Pi}_{\bdd{V}} \bdd{u}, \bdd{\Pi}_{\bdd{V}} \bdd{v}) + b(\bdd{\Pi}_{\bdd{V}} \bdd{v}, \tilde{\Pi} p) + b(\bdd{\Pi}_{\bdd{V}} \bdd{u}, \tilde{\Pi} q)  \\
	&= a(\bdd{u}, \bdd{v}) + b(\bdd{v}, p) + b(\bdd{u}, q).
	\end{align*}
	\cref{eq:variational form schur complement} now follows from \cref{eq:matrix relations 3} on noting that $\tilde{Q}_E = Q_I^{\perp} \cap L^2_0(\Omega)$. Arguing as in \cite[Theorem 3.22]{Elman14}, the eigenvalue bound follows from the inf-sup condition \cref{eq:inf-sup global boundary spaces}. 
\end{proof}

\section{Basis Functions}
\label{sec:basis functions}

We first define a basis $\{\phi_i \}$ for the space of scalar-valued functions $V$. The basis is constructed so that the exclusion of particular functions gives a basis for $V_0 = V \cap H^1_0(\Omega)$, which simplifies both the enforcement of homogeneous boundary conditions and the implementation of the preconditioner. A basis $\{ \bdd{\Phi}_i \}$ for the velocity space $\bdd{V}_0 = V_0 \times V_0$ is then obtained using functions of the form $\phi_j \unitvec{e}_1$ and $\phi_j \unitvec{e}_2$. For the pressure space, we only give a basis for $Q$ since the space $Q_0$ is not used in the actual implementation.

\subsection{Basis Functions on a Reference Triangle}

We begin by defining basis functions for the pressure and velocity spaces on the reference triangle $\hat{T}$ shown in \cref{fig:reference triangle}.

\subsubsection{Pressure Basis Functions}
Let $\{B^{k}_{\alpha} \}_{\alpha \in \mathcal{I}}$ denote the Bernstein polynomials \cite{Lai07}:
\begin{align}
\label{eq:bernstein poly}
B^{k}_{\alpha} = \frac{k!}{\alpha_1 ! \alpha_2 ! \alpha_3! } \lambda_1^{\alpha_1} \lambda_2^{\alpha_2} \lambda_3^{\alpha_3},
\end{align}
where $\mathcal{I} = \{ \alpha \in \mathbb{Z}_{+}^{3} : |\alpha| = k \}$ and $\{\lambda_i, \ 1 \leq i \leq 3\}$, are the barycentric coordinates on the reference triangle $\hat{T}$. The set $\{ B_{\alpha}^{k} \}_{\alpha \in \mathcal{I}}$ forms a basis for $\mathcal{P}_{k}(\hat{T})$ \cite{Lai07}. Each Bernstein polynomial $B_{\alpha}^{k}$ can be identified with the domain point $\bdd{x}_{\alpha} = \frac{\alpha_1}{ k } \hat{\bdd{a}}_1 +  \frac{\alpha_2}{ k } \hat{\bdd{a}}_2 + \frac{\alpha_3}{ k } \hat{\bdd{a}}_3$
on the reference triangle. Let $\mathcal{I}_0 = \{ \alpha \in \mathcal{I} : \alpha_i < k \}$ denote the subset corresponding to interior (non-vertex) points. Fix any $\beta \in \mathcal{I}_0$; since all the Bernstein polynomials \cref{eq:bernstein poly} share the same average value, the set $\{ B_{\alpha}^{k} - B_{\beta}^{k} \}_{\alpha \in \mathcal{I} \setminus \{\beta\} }$ is a basis for $\mathcal{P}_{k}(\hat{T}) \cap L^2_0(\hat{T})$. This set can be partitioned into:
\begin{enumerate}
	\item[(i)] Vertex functions: $	\hat{\psi}_{i} := B^{k}_{ke_i} - B^{k}_{\beta }$, $1 \leq i \leq 3$,	satisfying $\int_{\hat{T}} \hat{\psi}_i d\bdd{x} = 0$ and  $\hat{\psi}_{i}(\hat{\bdd{a}}_{j}) = \delta_{ij}$ for $1 \leq i, j \leq 3$.
	
	\item[(ii)] Interior functions: $\hat{\psi}_{\iota,\alpha} := B^{k}_{\alpha} - B_{\beta}^{k}$, $\alpha \in \mathcal{I}_0 \setminus \{\beta\}$, satisfying $\int_{\hat{T}} \hat{\psi}_{\iota, \alpha} d\bdd{x} = 0$ and $\hat{\psi}_{\iota,\alpha}(\hat{\bdd{a}}_i) = 0$, $1 \leq i \leq 3$.
\end{enumerate}
In order to obtain a basis for $\mathcal{P}_{k}(\hat{T})$, we supplement this set with one additional function:
\begin{enumerate}
	\item[(iii)] Average value function
	\begin{align}
	\label{eq:pressure average value func}
	\hat{\psi}_{\hat{T}} := 1 - \sum_{i=1}^{3} \hat{\psi}_{i}.
	\end{align} 
	satisfying $|\hat{T}|^{-1} \int_{\hat{T}} \hat{\psi}_{\hat{T}} d \bdd{x} = 1$ and $\hat{\psi}_{\hat{T}}(\bdd{a}_i) = 0$, $1 \leq i \leq 3$.
\end{enumerate}
In summary, there are 3 vertex functions, one average value function and $\frac{1}{2}(k+1)(k+2) - 4$ interior functions which total $\frac{1}{2}(k+1)(k+2) = \dim \mathcal{P}_k(\hat{T})$, and form a basis for the pressure space $Q = \mathcal{P}_k(\hat{T})$ on the reference element.

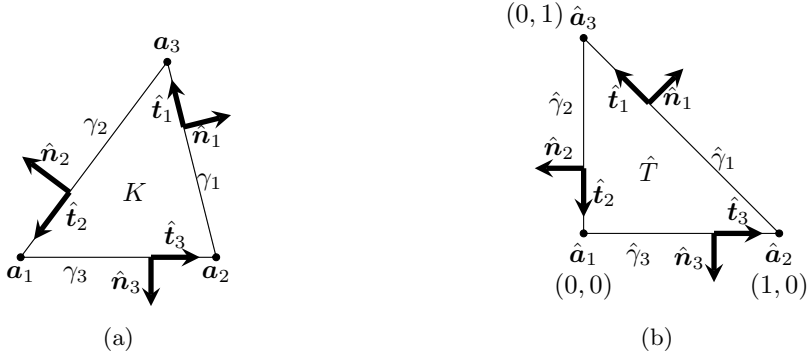
\begin{figure}[hbt]
	\centering
	\begin{subfigure}[b]{0.45\linewidth}
		\centering
		\begin{tikzpicture}[scale=0.65]
		\filldraw (-2,0) circle (2pt) node[align=center,below]{$\bdd{a}_1$}
		-- (2,0) circle (2pt) node[align=center,below]{$\bdd{a}_2$}	
		-- (1, 4) circle (2pt) node[align=center,above]{$\bdd{a}_3$}
		-- (-2,0);
		
		\coordinate (a1) at (-2,0);
		\coordinate (a2) at (2,0);
		\coordinate (a3) at (1, 4);
		
		\coordinate (e113) at ($(a2)!1/3!(a3)$);
		\coordinate (e123) at ($(a2)!2/3!(a3)$);
		\coordinate (e1231) at ($(a2)!2/3+1/sqrt(17)!(a3)$);
		
		\coordinate (e213) at ($(a3)!1/3!(a1)$);
		\coordinate (e223) at ($(a3)!2/3!(a1)$);
		\coordinate (e2231) at ($(a3)!2/3+1/sqrt(17)!(a1)$);
		
		\coordinate (e313) at ($(a1)!1/3!(a2)$);
		\coordinate (e323) at ($(a1)!2/3!(a2)$);
		\coordinate (e3231) at ($(a1)!2/3+1/4!(a2)$);
		
		% edge labels
		\draw ($(e113)+(0.2,0.2)$) node[align=center]{$\gamma_1$};
		\draw ($(e213)+(0,0)$) node[align=center,left]{$\gamma_2$};
		\draw ($(e313)+(-0.2,0)$) node[align=center,below]{$\gamma_3$};
		
		% tangents and normals
		% t0
		\draw[line width=2, -stealth] (e123) -- (e1231);
		\draw ($($(e123)!0.5!(e1231)$)+(-0.3,-0.1)$) node[align=center]{$\unitvec{t}_1$};
		% n0
		\draw[line width=2, -stealth] (e123) -- ($(e123)!1!-90:(e1231)$);
		\draw ($($(e123)!0.5!-90:(e1231)$)+(0.,-0.3)$) node[align=center]{$\unitvec{n}_1$};
		% t1
		\draw[line width=2, -stealth] (e223) -- (e2231);
		\draw ($($(e223)!0.5!(e2231)$)+(0.1,0)$) node[align=center, right]{$\unitvec{t}_2$};
		% n1
		\draw[line width=2, -stealth] (e223) -- ($(e223)!1!-90:(e2231)$);
		\draw ($($(e223)!0.5!-90:(e2231)$)+(0.2,0)$) node[align=center, above]{$\unitvec{n}_2$};
		% t2
		\draw[line width=2, -stealth] (e323) -- (e3231);
		\draw ($($(e323)!0.5!(e3231)$)+(0,0)$) node[align=center, above]{$\unitvec{t}_3$};
		% n2
		\draw[line width=2, -stealth] (e323) -- ($(e323)!1!-90:(e3231)$);
		\draw ($($(e323)!0.5!-90:(e3231)$)+(0,0)$) node[align=center, left]{$\unitvec{n}_3$};
		
		% element label
		\draw (1/3, 4/3) node(T){$K$};
		
		\end{tikzpicture}	
		\caption{}
		\label{fig:general triangle}
	\end{subfigure}
	\hfill
	\begin{subfigure}[b]{0.45\linewidth}
		\centering
		\begin{tikzpicture}[scale=0.65]
		\filldraw (0,0) circle (2pt) node[align=center,below]{$\hat{\bdd{a}}_1$ \\ $(0,0)$}
		-- (4,0) circle (2pt) node[align=center,below]{$\hat{\bdd{a}}_2$\\ $(1,0)$}	
		-- (0,4) circle (2pt) node[align=center,above]{$\hat{\bdd{a}}_3$}
		-- (0,0);
		
		\draw (-1, 4) node[align=center, above]{$(0, 1)$}; 
		
		\coordinate (a1) at (0,0);
		\coordinate (a2) at (4,0);
		\coordinate (a3) at (0,4);
		
		\coordinate (e113) at ($(a2)!1/3!(a3)$);
		\coordinate (e123) at ($(a2)!2/3!(a3)$);
		\coordinate (e1231) at ($(a2)!2/3+1/sqrt(32)!(a3)$);
		
		\coordinate (e213) at ($(a3)!1/3!(a1)$);
		\coordinate (e223) at ($(a3)!2/3!(a1)$);
		\coordinate (e2231) at ($(a3)!2/3+1/4!(a1)$);
		
		\coordinate (e313) at ($(a1)!1/3!(a2)$);
		\coordinate (e323) at ($(a1)!2/3!(a2)$);
		\coordinate (e3231) at ($(a1)!2/3+1/4!(a2)$);
		
		% edge labels
		\draw ($(e113)+(0.2,0.2)$) node[align=center]{$\hat{\gamma}_1$};
		\draw ($(e213)+(0,0)$) node[align=center,left]{$\hat{\gamma}_2$};
		\draw ($(e313)+(-0.2,0)$) node[align=center,below]{$\hat{\gamma}_3$};
		
		% tangents and normals
		% t0
		\draw[line width=2, -stealth] (e123) -- (e1231);
		\draw ($($(e123)!0.5!(e1231)$)+(-0.25,-0.25)$) node[align=center]{$\unitvec{t}_1$};
		% n0
		\draw[line width=2, -stealth] (e123) -- ($(e123)!1!-90:(e1231)$);
		\draw ($($(e123)!0.5!-90:(e1231)$)+(0.25,-0.25)$) node[align=center]{$\unitvec{n}_1$};
		% t1
		\draw[line width=2, -stealth] (e223) -- (e2231);
		\draw ($($(e223)!0.5!(e2231)$)+(0,0)$) node[align=center, right]{$\unitvec{t}_2$};
		% n1
		\draw[line width=2, -stealth] (e223) -- ($(e223)!1!-90:(e2231)$);
		\draw ($($(e223)!0.5!-90:(e2231)$)+(0,0)$) node[align=center, above]{$\unitvec{n}_2$};
		% t2
		\draw[line width=2, -stealth] (e323) -- (e3231);
		\draw ($($(e323)!0.5!(e3231)$)+(0,0)$) node[align=center, above]{$\unitvec{t}_3$};
		% n2
		\draw[line width=2, -stealth] (e323) -- ($(e323)!1!-90:(e3231)$);
		\draw ($($(e323)!0.5!-90:(e3231)$)+(0,0)$) node[align=center, left]{$\unitvec{n}_3$};

		% element label
		\draw (4/3, 4/3) node(T){$\hat{T}$};
		\end{tikzpicture}		
		\caption{}
		\label{fig:reference triangle}
	\end{subfigure}
	\caption{Notation for (a) general triangle $K$ and (b) reference triangle $\hat{T}$.}
\end{figure}

\subsubsection{Velocity Basis Functions}

The construction of the basis functions for the velocity space $V$ is more complicated owing to the higher continuity requirement. In particular, the basis functions $\{ \hat{\phi}_{k}^{\beta} \}$, $|\beta| =1$, $k \in \{1,2,3\}$, associated with the derivative degrees of freedom at the vertices should satisfy $D^{\alpha} \hat{\phi}_{k}^{\beta}(\hat{\bdd{a}}_l) = \delta_{ \alpha \beta} \delta_{kl}$, $|\alpha| = 1$, $l \in \{1,2,3\}$. In order to construct these functions, we begin by considering the vector valued function given by
\begin{align}
\label{eq:vec J definition}
\vec{J}_1 := \frac{ \lambda_1^2 }{ P_{k-3}^{(3,3)}(-1)}  \begin{bmatrix}
\lambda_2 P_{k-3}^{(3,3)}(\lambda_{2} - \lambda_{1}) \\
\lambda_3 P_{k-3}^{(3,3)}(\lambda_{3} - \lambda_{1})
\end{bmatrix},
\end{align}
where ${P}_{k}^{(3,3)}$ is the Jacobi polynomial of degree $k$ \cite{Sz39}. The first component of $\vec{J}_1$ vanishes on edge $\gamma_2$ and the gradient at $\hat{\bdd{a}}_1$ is given by $\begin{bmatrix}
1 & 0
\end{bmatrix}^T$, while the second component vanishes on edge $\hat{\gamma}_3$ and has gradient $\begin{bmatrix}
0 & 1
\end{bmatrix}^T$ at $\hat{\bdd{a}}_1$. The factor $\lambda_1^2$ means that both components of $\vec{J}_1$ and their gradients vanish on the edge $\hat{\gamma}_1$. In summary, since $x = \lambda_2$ and $y = \lambda_3$, we have
\begin{align}
\label{eq:j1 derivative property}
\vec{J}_1(\hat{\bdd{a}}_k) = \vec{0} \quad \text{and} \quad \begin{bmatrix}
\frac{\partial}{\partial x} \\
\frac{\partial}{\partial y}
\end{bmatrix} \vec{J}_1^T(\hat{\bdd{a}}_k) =  \begin{bmatrix}
\frac{\partial}{\partial \lambda_2} \\  \frac{\partial}{\partial \lambda_3}
\end{bmatrix} \vec{J}_1^T(\hat{\bdd{a}}_k) = \delta_{kl} \begin{bmatrix}
1 & 0 \\ 0 & 1
\end{bmatrix}.
\end{align}
Defining $\vec{J}_2$ and $\vec{J}_3$ by cyclic permutations of the indices, we conclude that $\vec{J}_2$ and $\vec{J}_3$ vanish at the vertices and that, for $k \in \{1,2,3\}$,
\begin{align}
\label{eq:J other properties}
\begin{bmatrix}
\frac{\partial}{\partial \lambda_3} \\ \frac{\partial}{\partial \lambda_1}
\end{bmatrix} \vec{J}_2^T(\hat{\bdd{a}}_k) = \delta_{k2} \begin{bmatrix}
1 & 0 \\ 0 & 1
\end{bmatrix} \quad \text{and} \quad \begin{bmatrix}
\frac{\partial}{\partial \lambda_1} \\ \frac{\partial}{\partial \lambda_2}
\end{bmatrix} \vec{J}_3^T(\hat{\bdd{a}}_k) = \delta_{k3} \begin{bmatrix}
1 & 0 \\ 0 & 1
\end{bmatrix}.
\end{align}
Substituting the identities
\begin{align*}
\begin{bmatrix}
\frac{\partial}{\partial \lambda_3} \\ \frac{\partial}{\partial \lambda_1}
\end{bmatrix} = \begin{bmatrix}
-1 & 1 \\ 
-1 & 0
\end{bmatrix} \begin{bmatrix}
\frac{\partial}{\partial x} \\ \frac{\partial}{\partial y}
\end{bmatrix} \quad \text{and} \quad \begin{bmatrix}
\frac{\partial}{\partial \lambda_1} \\ \frac{\partial}{\partial \lambda_2}
\end{bmatrix} = \begin{bmatrix}
0 & -1 \\ 
1 & -1
\end{bmatrix} \begin{bmatrix}
\frac{\partial}{\partial x} \\ \frac{\partial}{\partial y}
\end{bmatrix}
\end{align*}
in \cref{eq:J other properties} and rearranging gives
\begin{align}
% try to combine into one line
\begin{bmatrix}
\frac{\partial}{\partial x} \\ \frac{\partial}{\partial y}
\end{bmatrix} \left( \begin{bmatrix}
-1 & -1 \\
1 & 0
\end{bmatrix} \vec{J}_2 \right)^T(\hat{\bdd{a}}_{k}) &= \delta_{k2} \begin{bmatrix}
1 & 0 \\ 0 & 1
\end{bmatrix} \label{eq:j2 derivative property} \\
\begin{bmatrix}
\frac{\partial}{\partial x} \\ \frac{\partial}{\partial y}
\end{bmatrix} \left( \begin{bmatrix}
0 & 1 \\
-1 & -1
\end{bmatrix} \vec{J}_3 \right)^T(\hat{\bdd{a}}_{k}) &= \delta_{k3} \begin{bmatrix}
1 & 0 \\ 0 & 1
\end{bmatrix}. \label{eq:j3 derivative property}
\end{align}	
Armed with \cref{eq:j1 derivative property,eq:j2 derivative property,eq:j3 derivative property}, we define the basis functions for the velocity space as follows:

\begin{enumerate}
	\item[(i)] $C^0$ vertex functions: $\hat{\phi}_{i} = \lambda_i^2(3 - 2\lambda_i)$, $\leq i \leq 3$, 
	satisfying $\hat{\phi}_i(\hat{\bdd{a}}_j) = \delta_{ij}$, $D\hat{\phi}_i(\hat{\bdd{a}}_j) = \bdd{0}$, and $\hat{\phi}_i|_{\hat{\gamma}_i} = 0$ for $1 \leq i, j \leq 3$.

	\item[(ii)] $C^1$ vertex functions
	\begin{align*}
	%\label{eq:c1 vertex reference}
	\begin{bmatrix}
	\hat{\phi}_{1}^{(1,0)} \\
	\hat{\phi}_{1}^{(0,1)}
	\end{bmatrix} &= \begin{bmatrix}
	1 & 0 \\
	0 & 1
	\end{bmatrix} \vec{J}_1; & \begin{bmatrix}
	\hat{\phi}_{2}^{(1,0)} \\
	\hat{\phi}_{2}^{(0,1)}
	\end{bmatrix} &= \begin{bmatrix}
	-1 & -1 \\
	1 & 0
	\end{bmatrix} \vec{J}_2; & \begin{bmatrix}
	\hat{\phi}_{3}^{(1,0)} \\
	\hat{\phi}_{3}^{(0,1)}
	\end{bmatrix} &= \begin{bmatrix}
	0 & 1 \\
	-1 & -1
	\end{bmatrix} \vec{J}_3
	\end{align*}
	which, thanks to \cref{eq:j1 derivative property,eq:j2 derivative property,eq:j3 derivative property}, 
	satisfy $D^{\beta} \hat{\phi}_{i}^{\alpha}(\hat{\bdd{a}}_j) = \delta_{\alpha\beta} \delta_{ij}$ and $\hat{\phi}_i|_{\hat{\gamma}_i} = 0$ for $1 \leq i, j \leq 3$, $|\alpha| = 1$, $|\beta| \leq 1$.
	
	\item[(iii)] Edge functions: Let $\hat{\gamma}$ be the edge connecting vertices $\hat{\bdd{a}}_i$ and $\hat{\bdd{a}}_j$; then the basis functions associated with the edge are defined by $\hat{\phi}_{\hat{\gamma}, l} = \lambda_{i}^2 \lambda_{j}^2 r_l(\lambda_{j} - \lambda_i)$ where $\{ r_l \}$ is any basis for $\mathcal{P}_{k-4}((-1,1))$. These functions satisfy $D^{\alpha} \hat{\phi}_{\hat{\gamma}, l}(\hat{\bdd{a}}_k)$ $= 0$ for $|\alpha| \leq 1$, $1 \leq k \leq 3$ and $\hat{\phi}_{\hat{\gamma}, l}|_{\hat{\gamma}'} = \delta_{ \hat{\gamma} \hat{\gamma'}}$ for $\hat{\gamma}' \in \mathcal{E}_{\hat{T}}$.
	
	\item[(iv)] Interior functions: The basis functions associated with the element interior are defined by
	$\hat{\phi}_{I, l} = \lambda_1 \lambda_2 \lambda_3 s_{l}$
	where $\{ s_l \}$ is any basis for $\mathcal{P}_{k-3}(\hat{T})$. These functions satisfy $\hat{\phi}_{I, l}|_{\partial \hat{T}} = 0$ and $D\hat{\phi}_{I, l}(\hat{\bdd{a}}_k) = \bdd{0}$ for $1 \leq k \leq 3$.
\end{enumerate}
It is easily seen that the above functions are linearly independent. Furthermore, there are 3 functions per vertex, $\dim \mathcal{P}_{k-4}((-1,1)) = k-3$ functions per edge, and $\dim \mathcal{P}_{k-3}(\hat{T}) = \frac{1}{2}(k-2)(k-1)$ interior functions which total $\frac{1}{2}(k+1)(k+2) = \dim \mathcal{P}_k(\hat{T})$. Hence, the above functions also form a basis for $\mathcal{P}_k(\hat{T})$.

\subsection{Basis Functions on a Mesh}
\label{sec:basis functions mesh}

We now define the global basis functions for the spaces $Q$ and $V$. The lower continuity requirements imposed at corner vertices $\mathcal{V}_C$ means that extra care must be taken when defining the global basis functions associated with $\mathcal{V}_C$.

\subsubsection{Pressure Basis Functions}

The pressure space $Q$ requires $C^0$ continuity at all vertices except at corner vertices, where the functions are allowed to be discontinuous. This means that \textit{each element has its own degree of freedom} at vertices $\bdd{a} \in \mathcal{V}_C$, whilst at the remaining vertices $\bdd{a} \in \mathcal{V} \setminus \mathcal{V}_C$, all elements \textit{share a single degree of freedom} at the common vertex as shown in \cref{fig:pressure vertex degrees of freedom example}. Consequently, any given vertex $\bdd{a} \in \mathcal{V}$ is associated with either (a) \textit{a single} basis function supported on the patch $\mathcal{T}_{\bdd{a}}$ if $\bdd{a} \in \mathcal{V} \setminus \mathcal{V}_C$, or (b) \textit{a collection} of basis functions, each of which is supported on a single element $K \in \mathcal{T}_{\bdd{a}}$ if $\bdd{a} \in \mathcal{V}_C$. The set of supports of the pressure functions associated with a vertex $\bdd{a} \in \mathcal{V}$ is defined by
\begin{align*}
\Omega_{\bdd{a}} = \begin{cases}
\{ \mathcal{T}_{\bdd{a}} \} & \bdd{a} \in \mathcal{V}_C \\
\{ K \in \mathcal{T}_{\bdd{a}} \} & \bdd{a} \in \mathcal{V} \setminus \mathcal{V}_C.
\end{cases}
\end{align*}
That is, the cardinality of these sets is $|\Omega_{\bdd{a}}| = 1$ for noncorner vertices (since there is only one vertex basis function associated to $\bdd{a}$) whilst $|\Omega_{\bdd{a}}| \geq 2$ for corner vertices, thanks to the assumption that the mesh is corner-split into at least two elements. The corresponding global vertex functions $\{\psi_{\bdd{a}}^{\omega} : \bdd{a} \in \mathcal{V}, \ \omega \in \Omega_{\bdd{a}} \} $ are defined to be pull-backs in the usual way:
\begin{align}
\label{eq:global pressure vertex}
\psi_{\bdd{a}}^{\omega} &= \begin{cases}
\hat{\psi}_{i} \circ \bdd{F}_K^{-1} & \text{on } K \subseteq \omega, \\
0 & \text{otherwise},
\end{cases} & \psi_K &= \begin{cases}
\hat{\psi}_{\hat{T}} \circ \bdd{F}_K^{-1} & \text{on } K, \\
0 & \text{otherwise},
\end{cases}
\end{align}
where $\hat{\bdd{a}}_i = \bdd{F}_K^{-1}(\bdd{a})$. 

\begin{figure}
	\centering
	\begin{subfigure}{0.48\linewidth}
		\begin{tikzpicture}
		\coordinate(v0) at (0, 0);
		\coordinate(v1) at (3, 1);
		\coordinate(v2) at (1, 2);
		\coordinate(v3) at (0, 3);
		\coordinate(v4) at (-2, 2);	
		
		\coordinate(i0) at (1.5, 0.5);	
		\coordinate(i1) at (1.5, 1.25);
		\coordinate(i2) at (0.5, 1);
		\coordinate(i3) at (-0.5, 2);
		\coordinate(i4) at (-1, 1);
		
		\draw (v0) -- (v1) -- (v2) -- (v3) -- (v4) -- (v0);
		
		\filldraw[blue] (i0) circle (1.5pt);
		\filldraw[blue] (i1) circle (1.5pt);
		\filldraw[blue] (i2) circle (1.5pt);
		\filldraw[blue] (i3) circle (1.5pt);
		\filldraw[blue] (i4) circle (1.5pt);
		
		\draw (v0) -- (i2);
		\draw (v1) -- (i1);
		\draw (i0) -- (i1);
		\draw (i0) -- (i2);
		\draw (i1) -- (i2);
		\draw (v2) -- (i1);
		\draw (v2) -- (i2);
		\draw (v2) -- (i3);
		\draw (v3) -- (i3);
		\draw (v4) -- (i3);
		\draw (i3) -- (i4);
		\draw (i3) -- (i2);
		\draw (i4) -- (i2);

		\filldraw[red] ($(v0) + (-0.05, 0.2)$) circle (1.5pt); 
		\filldraw[red]  ($(v0) + (0.2, 0.15)$) circle (1.5pt);
		\filldraw[red] ($(v1) + (-0.5, -0.04)$) circle (1.5pt); 
		\filldraw[red] ($(v1) + (-0.5, 0.16)$) circle (1.5pt); 
		\filldraw[red] ($(v2) + (0.25, -0.2)$) circle (1.5pt); 
		\filldraw[red] ($(v2) + (0.01, -0.23)$) circle (1.5pt); 
		\filldraw[red] ($(v2) + (-0.2, -0.1)$) circle (1.5pt);
		\filldraw[red] ($(v2) + (-0.25, 0.1)$) circle (1.5pt);
		\filldraw[red] ($(v3) + (0.05, -0.23)$) circle (1.5pt);
		\filldraw[red] ($(v3) + (-0.21, -0.2)$) circle (1.5pt);
		\filldraw[red] ($(v4) + (0.4, 0.1)$) circle (1.5pt);
		\filldraw[red] ($(v4) + (0.3, -0.1)$) circle (1.5pt);
		
		\filldraw[red] (1, 3) circle (1.5pt);
		\draw (1, 3) node[align=center, right]{Corner dof};
		\filldraw[blue] (1, 2.5) circle (1.5pt);
		\draw (1, 2.5) node[align=center, right]{Noncorner dof};
		\end{tikzpicture}
		\caption{}
		\label{fig:pressure vertex degrees of freedom example}
	\end{subfigure}
	\hfill
	\begin{subfigure}{0.48\linewidth}
		\begin{tikzpicture}
		\coordinate(v0) at (0, 0);
		\coordinate(v1) at (3, 1);
		\coordinate(v2) at (1, 2);
		\coordinate(v3) at (0, 3);
		\coordinate(v4) at (-2, 2);	
		
		\coordinate(i0) at (1.5, 0.5);	
		\coordinate(i1) at (1.5, 1.25);
		\coordinate(i2) at (0.5, 1);
		\coordinate(i3) at (-0.5, 2);
		\coordinate(i4) at (-1, 1);
		
		\draw (v0) -- (v1) -- (v2) -- (v3) -- (v4) -- (v0);
		
		\filldraw[red] (v0) circle (2.5pt); 
		\filldraw[red] (v1) circle (2.5pt);
		\filldraw[red] (v2) circle (2.5pt);
		\filldraw[red] (v3) circle (2.5pt);
		\filldraw[red] (v4) circle (2.5pt);
		
		\filldraw[green] (i0) circle (2.5pt);
		\filldraw[green] (i4) circle (2.5pt);
		
		\filldraw[blue] (i1) circle (2.5pt);
		\filldraw[blue] (i2) circle (2.5pt);
		\filldraw[blue] (i3) circle (2.5pt);

		\draw (v0) -- (i2);
		\draw (v1) -- (i1);
		\draw (i0) -- (i1);
		\draw (i0) -- (i2);
		\draw (i1) -- (i2);
		\draw (v2) -- (i1);
		\draw (v2) -- (i2);
		\draw (v2) -- (i3);
		\draw (v3) -- (i3);
		\draw (v4) -- (i3);
		\draw (i3) -- (i4);
		\draw (i3) -- (i2);
		\draw (i4) -- (i2);

		% interior vertices
		% i1
		\draw[line width=2, -stealth] (i1) -- ($(i1)+(0.5,0)$);
		\draw[line width=2, -stealth] (i1) -- ($(i1)+(0,0.5)$);
		% i2
		\draw[line width=2, -stealth] (i2) -- ($(i2)+(0.5,0)$);
		\draw[line width=2, -stealth] (i2) -- ($(i2)+(0,0.5)$);
		% i3
		\draw[line width=2, -stealth] (i3) -- ($(i3)+(0.5,0)$);
		\draw[line width=2, -stealth] (i3) -- ($(i3)+(0,0.5)$);
		
		% noncorner boundary vertex
		% i0
		\draw[line width=2, -stealth] (i0) -- ($(i0)!0.5/sqrt(2.5)!(v1)$);
		\draw[line width=2, -stealth] (i0) -- ($(i0)!0.5/sqrt(2.5)!-90:(v1)$);
		% i4
		\draw[line width=2, -stealth] (i4) -- ($(i4)!0.5/sqrt(2)!(v0)$);
		\draw[line width=2, -stealth] (i4) -- ($(i4)!0.5/sqrt(2)!-90:(v0)$);
		
		% corner boundary vertex
		% v0
		\draw[line width=2, -stealth] (v0) -- ($(v0)!0.5/sqrt(2.5)!(i0)$);
		\draw[line width=2, -stealth] (v0) -- ($(v0)!0.5/sqrt(1.25)!(i2)$);
		\draw[line width=2, -stealth] ($(v0)!0.5/sqrt(2)!(i4)$) -- (v0);
		
		% v1
		\draw[line width=2, -stealth] ($(v1)!0.5/sqrt(2.5)!(i0)$) -- (v1);
		\draw[line width=2, -stealth] (v1) -- ($(v1)!0.5/sqrt(2.3125)!(i1)$);
		\draw[line width=2, -stealth] (v1) -- ($(v1)!0.5/sqrt(5)!(v2)$);
		
		% v2
		\draw[line width=2, -stealth] ($(v2)!0.5/sqrt(5)!(v1)$) -- (v2);
		\draw[line width=2, -stealth] (v2) -- ($(v2)!0.5/sqrt(0.8125)!(i1)$);
		\draw[line width=2, -stealth] (v2) -- ($(v2)!0.5/sqrt(1.25)!(i2)$);
		\draw[line width=2, -stealth] (v2) -- ($(v2)!0.5/sqrt(2.25)!(i3)$);
		\draw[line width=2, -stealth] (v2) -- ($(v2)!0.5/sqrt(2)!(v3)$);
		
		% v3
		\draw[line width=2, -stealth] ($(v3)!0.5/sqrt(2)!(v2)$) -- (v3);
		\draw[line width=2, -stealth] (v3) -- ($(v3)!0.5/sqrt(1.25)!(i3)$);
		\draw[line width=2, -stealth] (v3) -- ($(v3)!0.5/sqrt(5)!(v4)$);
		
		% v4
		\draw[line width=2, -stealth] ($(v4)!0.5/sqrt(5)!(v3)$) -- (v4);
		\draw[line width=2, -stealth] (v4) -- ($(v4)!0.5/sqrt(2.25)!(i3)$);
		\draw[line width=2, -stealth] (v4) -- ($(v4)!0.5/sqrt(2)!(i4)$);
		
		% legend
		% functional value
		\filldraw[black] (1, 3) circle (2.5pt);
		\draw (1, 3) node[align=center, right]{Functional dof};
		% derivative dof
		\draw[line width=2, -stealth] (1, 2.5) -- (1.5, 2.5);
		\draw (1.5, 2.5) node[align=center, right]{Derivative dof};
		\end{tikzpicture}
		\caption{}
		\label{fig:velocity example mesh}
	\end{subfigure}
	\caption{(a) The global pressure vertex degrees of freedom and (b) the global velocity vertex degrees of freedom on a mesh of an example domain. Observe that in (a) there are multiple pressure vertex degrees of freedom at corner vertices but only a single degree of freedom at interior vertices and in (b) there are three or more derivative degrees of freedom at corner vertices (red), two derivative degrees of freedom aligned with the domain boundary at noncorner boundary vertices (green), and two derivative degrees of freedom aligned with the coordinate axes at interior vertices (blue).}
\end{figure}
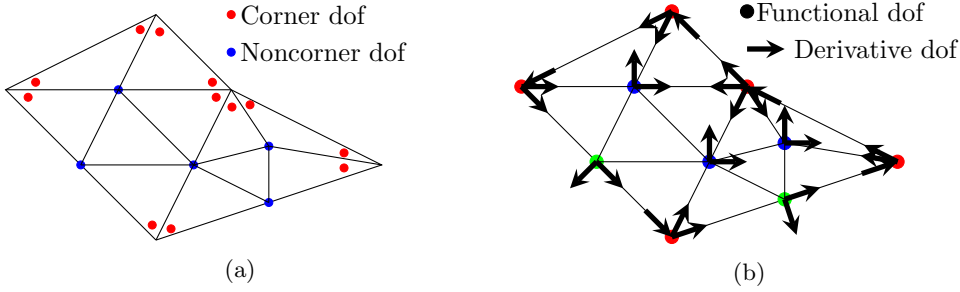

The average value functions and interior functions are simpler. Each element $K \in \mathcal{T}$ has a single function $\psi_K$, corresponding to the average value over $K$, defined by \cref{eq:global pressure vertex}.
Similarly, each element $K \in \mathcal{T}$ has $\frac{k}{2}(k+1) - 4$ interior functions also defined to be pull-backs.

\subsubsection{Velocity Basis Functions}

The velocity space $V$ imposes $C^1$ continuity at all vertices except at corner vertices, where only $C^0$-continuity is required to ensure $V \subset H^1(\Omega)$. This means that at corner vertices $\bdd{a} \in \mathcal{V}_C$, each element $K \in \mathcal{T}_{\bdd{a}}$ has two degrees of freedom for the gradient corresponding to the two tangential derivatives corresponding to the two edges of $K$ that meet at $\bdd{a}$. To enforce continuity between two neighboring elements in $\mathcal{T}_{\bdd{a}}$, the tangential derivative corresponding to the common edge must be \textit{shared between the two elements}. In other words, each corner vertex $\bdd{a} \in \mathcal{V}_C$ has \textit{one derivative degree of freedom for each edge} $\gamma \in \mathcal{E}_{\bdd{a}}$. For the remaining noncorner vertices $\bdd{a} \in \mathcal{V} \setminus \mathcal{V}_C$, all elements in $\mathcal{T}_{\bdd{a}}$ \textit{share two degrees of freedom} at the common vertex, corresponding to any two linearly independent directional derivatives as in \cref{fig:velocity example mesh}. Consequently, a given vertex $\bdd{a} \in \mathcal{V}$ is associated with either (a) \textit{two} basis functions supported on the patch $\mathcal{T}_{\bdd{a}}$ if $\bdd{a} \in \mathcal{V} \setminus \mathcal{V}_C$, or (b) \textit{a collection} of basis functions, each of which is associated to an edge $\gamma \in \mathcal{E}_{\bdd{a}}$ and supported on the pair of elements sharing the common edge $\gamma$ if $\bdd{a} \in \mathcal{V}_C$.

The set of unit vectors defining the directional derivative degrees of freedom at a vertex $\bdd{a} \in \mathcal{V}$ are chosen as follows:
\begin{align}
\label{eq:derivative directions}
D_{\bdd{a}} = \begin{cases}
\{ \unitvec{e}_1, \unitvec{e}_2 \} & \bdd{a} \in \mathcal{V}_I \\
\{ \unitvec{t}, \unitvec{n} \} &  \bdd{a} \in \mathcal{V}_B \\
\{ \unitvec{t}_{\gamma} : \gamma \in \mathcal{E}_{\bdd{a}} \} &  \bdd{a} \in \mathcal{V}_C
\end{cases}
\end{align}
where $\unitvec{t}$ and $\unitvec{n}$ are the unit tangent and normal vectors at a noncorner boundary vertex $\bdd{a} \in \mathcal{V}_B$ and $\unitvec{t}_{\gamma}$ denotes a unit tangent vector on an edge $\gamma \in \mathcal{E}$ as illustrated in \cref{fig:velocity example mesh}.
For a given vertex $\bdd{a} \in \mathcal{V}$ and unit vector $\unitvec{\mu} \in D_{\bdd{a}}$, the global basis function $\phi_{\bdd{a}}^{\mu}$ has support
\begin{align*}
\supp \phi_{\bdd{a}}^{\mu} = \begin{cases}
\{ K \in \mathcal{T}_{\bdd{a}} : \exists \gamma \in \mathcal{E}_K \text{ with } \unitvec{t}_{\gamma} = \pm \unitvec{\mu} \} & \bdd{a} \in \mathcal{V}_C,  \\
\mathcal{T}_{\bdd{a}} & \bdd{a} \in \mathcal{V} \setminus \mathcal{V}_C.
\end{cases}
\end{align*}
The global $C^1$ vertex functions come in pairs as follows: Given a noncorner vertex $\bdd{a} \in \mathcal{V}_I \cup \mathcal{V}_B$, let $\unitvec{\mu}_1$, $\unitvec{\mu}_2$ be unit vectors such that $D_{\bdd{a}} = \{ \unitvec{\mu}_1, \unitvec{\mu}_2 \}$ as in \cref{eq:derivative directions}, and define the basis functions by
\begin{align}
\label{eq:vertex c1 reference to physical}
\begin{bmatrix}
\phi_{\bdd{a}}^{\mu_1} \\
\phi_{\bdd{a}}^{\mu_2}
\end{bmatrix}   =
\begin{bmatrix}
\unitvec{\mu}_1 & \unitvec{\mu}_2
\end{bmatrix}^{-1} D\bdd{F}_K \begin{bmatrix}
\hat{\phi}_{i}^{(1,0)} \circ \bdd{F}_K^{-1} \\
\hat{\phi}_{i}^{(0,1)} \circ \bdd{F}_K^{-1} \\
\end{bmatrix}  \quad \text{on } K \in \mathcal{T}_{\bdd{a}},
\end{align}
where $\hat{\bdd{a}}_i = \bdd{F}_K^{-1}(\bdd{a})$. The above construction ensures that the basis functions are $C^1$ continuous at the vertex $\bdd{a}$: i.e. $\unitvec{\mu}_i \cdot \nabla \phi_{\bdd{a}}^{\mu_j}(\bdd{a}) = \delta_{ij}$. The case of a corner vertex $\bdd{a} \in \mathcal{V}_C$ is more complicated since, as mentioned above, each edge $\gamma \in \mathcal{E}_{\bdd{a}}$ contributes one independent basis function at the vertex, also defined by the expression \cref{eq:vertex c1 reference to physical}, which is supported on the edge patch $\{ K \in \mathcal{T}_{\bdd{a}} : \gamma \in \mathcal{E}_K \}$. The unit vectors $\unitvec{\mu}_1, \unitvec{\mu}_2$ in \cref{eq:vertex c1 reference to physical}  associated with such an element $K \in \mathcal{T}_{\bdd{a}}$ are taken to be the pair of unit tangent vectors on the two edges of $K$ having an endpoint at $\bdd{a}$. This means that the basis functions $\phi_{\bdd{a}}^{\mu_1}$ and $\phi_{\bdd{a}}^{\mu_2}$ are the only $C^1$ vertex functions supported on $K$.

The remaining $C^0$ vertex functions, edge functions, and interior functions are again defined to be pull-backs of the corresponding functions on the reference element in the usual way: i.e. $\phi_{\bdd{a}} = \hat{\phi}_{i} \circ \bdd{F}_{K}^{-1}$ on $K \in \mathcal{T}_{\bdd{a}}$ and $\phi_{\bdd{a}} = 0$ otherwise,
where $\hat{\bdd{a}}_i = \bdd{F}_K^{-1}(\bdd{a})$. Similarly, there are $k-3$ edge functions per edge $\gamma \in \mathcal{E}$, supported on the patch of elements containing that edge $\{ K \in \mathcal{T} : \gamma \in \mathcal{E}_K \}$, and there are $\frac{1}{2}(k-1)(k-2)$ interior functions per element $K \in \mathcal{T}$.

\subsubsection{Velocity Basis Functions with Homogeneous Boundary Conditions}

The above construction gives a basis for $V$ in the absence of essential boundary conditions. If nonhomogeneous essential boundary conditions are imposed, then the values of the following basis functions will be constrained by the boundary data:
\begin{itemize}
	
	\item the $C^0$ vertex function $\phi_{\bdd{a}}$, $\bdd{a} \in \mathcal{V} \setminus \mathcal{V}_I$ at each vertex on the domain boundary;
	
	\item the $C^1$ vertex function at each noncorner boundary vertex corresponding to the tangential derivative degree of freedom, i.e. $\phi_{\bdd{a}}^{t}$ for $\bdd{a} \in \mathcal{V}_B$;
	
	\item the pair of $C^1$ vertex functions at each corner boundary vertex corresponding to the tangential derivatives along the domain boundary edges: $\phi_{\bdd{a}}^{t_{\gamma}}$, $\phi_{\bdd{a}}^{t_{\gamma'}}$ for $\bdd{a} \in \mathcal{V}_C$ where $\gamma, \gamma' \in \mathcal{E}_{\bdd{a}} \cap \Gamma$; and
	
	\item all $k-3$ edge functions for each edge on the domain boundary.
\end{itemize}
If homogeneous essential boundary conditions are imposed, then a basis for $V_0 = V \cap H^1_0(\Omega)$ is obtained by taking the following functions:
\begin{itemize}
	\item the $C^0$ vertex function at each interior vertex, i.e. $\phi_{\bdd{a}}$, $\bdd{a} \in \mathcal{V}_I$;
	
	\item the following $C^1$ vertex functions: $\phi_{\bdd{a}}^{\mu}$, $\bdd{a} \in \mathcal{V}$, $\unitvec{\mu} \in \mathring{D}_{\bdd{a}}$ where
	\begin{align}
	\label{eq:derivative directions v0}
	\mathring{D}_{\bdd{a}} = \begin{cases}
	\{ \unitvec{e}_1, \unitvec{e}_2 \} & \bdd{a} \in \mathcal{V}_I \\
	\{ \unitvec{n} \} &  \bdd{a} \in \mathcal{V}_B \\
	\{ \unitvec{t}_{\gamma} : \gamma \in \mathcal{E}_{\bdd{a}} \cap \mathcal{E}_I \} &  \bdd{a} \in \mathcal{V}_C
	\end{cases}
	\end{align}
	and $\mathcal{E}_I$ denotes the set of interior edges; 
	
	\item all $k-3$ edge functions on each interior edge; and
	
	\item all interior functions on each element.
\end{itemize}
Condition \cref{eq:derivative directions v0} means that we keep both $C^1$ vertex functions for each interior vertex, the $C^1$ vertex function associated with the the outward normal of $\Gamma$ for each noncorner boundary vertex, and each $C^1$ vertex function corresponding to an interior edge unit tangent vector at corner vertices.

\section{Constructing the Preconditioner Using Additive Schwarz Theory}
\label{sec:asm theory}

In \cref{sec:precond and main theorem}, we constructed the stiffness matrix for the Stokes problem \cref{eq:full system matrix} using bases for the spaces $\bdd{V}_0$ and $Q$ and performed static condensation to arrive at the Schur complement system \cref{eq:Schur complement system}. In \cref{sec:schur complement variational}, it was shown that the algebraic Schur complement system \cref{eq:Schur complement system} was related to the mixed finite element problem \cref{eq:variational form schur complement} posed on the spaces $\tilde{\bdd{V}}_E \times \tilde{Q}_E$. The alert reader will have noticed a slight discrepancy in the treatment of the average pressure mode over the domain $\Omega$: in \cref{sec:basis functions mesh,sec:precond and main theorem}, the average pressure modes were included in the discretization (and it was pointed out that these modes span the kernel of the Schur complement) whereas in \cref{sec:schur complement variational}, the pressure space $\tilde{Q}_E = Q_I^{\perp} \cap L^2_0(\Omega)$ was used, which factors out the singular mode. In order to construct a preconditioner in the form \cref{eq:preconditioner form}, we formulate an Additive Schwarz Method (ASM) over the spaces $\tilde{\bdd{V}}_E \times Q_I^{\perp}$ rather than the seemingly more natural choice $\tilde{\bdd{V}}_E \times \tilde{Q}_E$ suggested by \cref{thm:schur complement variational form}.

\subsection{Pressure ASM}

We decompose the pressure space $Q_I^{\perp}$ as follows:
\begin{align}
\label{eq:QI direct sum decomposition}
Q_I^{\perp} = \bigoplus_{\substack{\bdd{a} \in \mathcal{V} \\ \omega \in \Omega_{\bdd{a}}}} \tilde{Q}_{\bdd{a},\omega} \oplus \bigoplus_{K \in \mathcal{T}} \tilde{Q}_{K},   
\end{align}
where (i) the vertex spaces $\tilde{Q}_{\bdd{a}, \omega} := \mathrm{span} \{ \tilde{\psi}_{\bdd{a}}^{\omega} \}$,  $\bdd{a} \in \mathcal{V},  \omega \in \Omega_{\bdd{a}}$ with $\tilde{\psi}_{\bdd{a}}^{\omega} := \tilde{\Pi} \psi_{\bdd{a}}^{\omega}$ are equipped with the inner product $	m_{\bdd{a}, \omega}( p, q ) := |\omega| k^{-4} p(\bdd{a}) q(\bdd{a})$, and (ii) the element average spaces $\tilde{Q}_{K} := \mathrm{span} \tilde{\psi}_K$, $K \in \mathcal{T}$, with $\tilde{\psi}_K := \tilde{\Pi} \psi_K$ are equipped with the inner product
\begin{align*}
m_{K}(p, q) := \frac{1}{|K|} \left( \int_{K} p \ d\bdd{x} \right) \left( \int_{K} q \ d\bdd{x} \right), \quad \forall p,q \in \tilde{Q}_K. 
\end{align*}
Applying the projection $\tilde{\Pi}$ to the formulae for $\psi_K$ \cref{eq:pressure average value func,eq:global pressure vertex} gives	
\begin{align*}
\tilde{\psi}_K = \begin{cases}
1 - \sum\limits_{\substack{\bdd{a} \in \mathcal{V}_K \\ \omega \supseteq K }} \tilde{\psi}_{\bdd{a}}^{\omega} & \text{on } K \\
0 & \text{otherwise},
\end{cases}
\end{align*}
where the functions $\tilde{\psi}_{\bdd{a}}^{\omega}$ are defined in (i) and we use the fact that $\tilde{\Pi}$ preserves constants. The direct sum decomposition \cref{eq:QI direct sum decomposition} means that any $q \in Q_I^{\perp}$ may be uniquely expressed in the form
\begin{align*}
q = \sum_{\substack{\bdd{a} \in \mathcal{V} \\ \omega \in \Omega_{\bdd{a}}}} q_{\bdd{a}, \omega} + \sum_{K \in \mathcal{T}} q_K,
\end{align*}
where
\begin{align*}
q_{\bdd{a}, \omega} &= q|_{\omega}(\bdd{a}) \tilde{\psi}_{\bdd{a}}^{\omega}, \quad \bdd{a} \in \mathcal{V}, \ \omega \in \Omega_{\bdd{a}}, & 
q_K &=  \left(\frac{1}{|K|} \int_{K} q \ d\bdd{x}\right) \tilde{\psi}_{K}, \quad K \in \mathcal{T}.
\end{align*}
The action of the associated ASM preconditioner on a residual $g \in L^2(\Omega)$ is given by the solution $p \in Q_I^{\perp}$ of the variational problem $\bar{m}(p, q) = (g, q)$  $\forall q \in Q_I^{\perp}$, where
\begin{align}
\label{eq:bar c definition}
\bar{m}(p, q) := \sum_{ \substack{\bdd{a} \in \mathcal{V} \\ \omega \in \Omega_{\bdd{a}} } } m_{\bdd{a}, \omega}(p_{\bdd{a},\omega}, q_{\bdd{a}, \omega}) + \sum_{K \in \mathcal{T}} m_{K}(p_K, q_K).
\end{align}
The bilinear form $\bar{m}(\cdot, \cdot)$ gives rise to a matrix preconditioner $\bar{\bdd{M}}$ for the pressure space defined by
\begin{align}
\label{eq:bilinear form to matrix pressure}
\vec{p}_e^{T} \bar{\bdd{M}} \vec{q}_e &= \bar{m}(\tilde{\Pi} p, \tilde{\Pi} q) \qquad \forall p,q \in Q.
\end{align}

\subsection{Velocity ASM}

We decompose the velocity space $\tilde{\bdd{V}}_E$ as follows:
\begin{align}
\label{eq:tilde VE direct sum decomposition}
\tilde{\bdd{V}}_E = \tilde{\bdd{V}}_{c} \oplus \bigoplus_{ \substack{ \bdd{a} \in \mathcal{V} \\ \unitvec{\mu} \in \mathring{D}_{\bdd{a}} } } \tilde{\bdd{V}}_{\bdd{a},\mu} \oplus \bigoplus_{\gamma \in \mathcal{E}_I} \tilde{\bdd{V}}_{\gamma},
\end{align}
where (i) the global $C^0$ vertex space $\tilde{\bdd{V}}_{c} := \mathrm{span}\{ \bdd{\Pi}_{\bdd{V}} (\phi_{\bdd{a}} \unitvec{e}_1), \ \bdd{\Pi}_{\bdd{V}} (\phi_{\bdd{a}} \unitvec{e}_2) : \bdd{a} \in \mathcal{V}_I \}$, (ii) the $C^1$ vertex spaces $\tilde{\bdd{V}}_{\bdd{a}, \mu} := \mathrm{span}\{ \bdd{\Pi}_{\bdd{V}} (\phi_{\bdd{a}}^{\mu} \bdd{e}_1), \ \bdd{\Pi}_{\bdd{V}}(\phi_{\bdd{a}}^{\mu} \bdd{e}_2)  \}$, $\bdd{a} \in \mathcal{V}$, $\unitvec{\mu} \in \mathring{D}_{\bdd{a}}$, and (iii) the edge spaces $\tilde{\bdd{V}}_{\gamma} := \mathrm{span}\{ \bdd{\Pi}_{\bdd{V}} (\phi_{\gamma, j} \bdd{e}_1), \bdd{\Pi}_{\bdd{V}} (\phi_{\gamma, j} \bdd{e}_2) : 1 \leq j \leq k-3 \}$, $\gamma \in \mathcal{E}_I$. Each of the velocity subspaces is equipped with the inner product $a(\cdot, \cdot)$ restricted to the appropriate space. The direct sum decomposition \cref{eq:tilde VE direct sum decomposition} means that any any $\bdd{u} \in \tilde{\bdd{V}}_E$ may be uniquely expressed in the form
\begin{align}
\label{eq:u subspace decomp}
\bdd{u} = \bdd{u}_c + \sum_{\substack{ \bdd{a} \in \mathcal{V} \\ \unitvec{\mu} \in \mathring{D}_{\bdd{a}} }} \bdd{u}_{\bdd{a}, \mu} + \sum_{\gamma \in \mathcal{E}_I} \bdd{u}_{\gamma},
\end{align}
where
\begin{align*}
\bdd{u}_c &= \sum_{\bdd{a} \in \mathcal{V}_I} \sum_{i=1}^{2} (\bdd{u} \cdot \unitvec{e}_i)(\bdd{a}) \bdd{\Pi}_{\bdd{V}} (\phi_{\bdd{a}} \unitvec{e}_i), \\
\bdd{u}_{\bdd{a}, \mu} &= \sum_{i=1}^{2} \partial_{\mu} (\bdd{u}|_{K_{\mu}} \cdot \unitvec{e}_i)(\bdd{a}) \bdd{\Pi}_{\bdd{V}} (\phi_{\bdd{a}}^{\mu} \unitvec{e}_i), \quad \bdd{a} \in \mathcal{V}, \quad \unitvec{\mu} \in \mathring{D}_{\bdd{a}}, \quad K_{\mu} \subseteq \supp \phi_{\bdd{a}}^{\mu},  \\
\intertext{and for each $\gamma \in \mathcal{E}_I$,}
\bdd{u}_{\gamma} &= \begin{cases}
\bdd{u} - \bdd{u}_c - \sum_{\substack{\bdd{a} \in \mathcal{V} \\ \unitvec{\mu} \in \mathring{D}_{\bdd{a}}}} \bdd{u}_{\bdd{a}, \mu} & \text{on } \gamma \\
\bdd{0} & \text{on the remaining edges in $\mathcal{E}$.}
\end{cases}
\end{align*}

The action of the associated ASM preconditioner on a residual $\bdd{f} \in \bdd{L}^2(\Omega)$ is given by the solution $\bdd{u} \in \tilde{\bdd{V}}_E$ of the variational problem $\bar{a}(\bdd{u}, \bdd{v}) = (\bdd{f}, \bdd{v})$ $\forall \bdd{v} \in \tilde{\bdd{V}}_E$, where
\begin{align}
\label{eq:bar a definition}
\bar{a}(\bdd{u}, \bdd{v}) = a(\bdd{u}_{c}, \bdd{v}_{c}) + \sum_{ \substack{\bdd{a} \in \mathcal{V} \\ \unitvec{\mu} \in \mathring{D}_{\bdd{a}}} } a(\bdd{u}_{\bdd{a}, \mu}, \bdd{v}_{\bdd{a}, \mu}) + \sum_{\gamma \in \mathcal{E}_I} a(\bdd{u}_{\gamma}, \bdd{v}_{\gamma}).
\end{align}
The bilinear form $\bar{a}(\cdot, \cdot)$ gives rise to a matrix preconditioner $\bar{\bdd{A}}$ for the velocity space defined by
\begin{align}
\label{eq:bilinear form to matrix velocity}
\vec{u}_E^T \bar{\bdd{A}} \vec{v}_E &= \bar{a}(\bdd{\Pi}_{\bdd{V}} \bdd{u}, \bdd{\Pi}_{\bdd{V}} \bdd{v}), \quad \forall \bdd{u},\bdd{v} \in \bdd{V}_0.
\end{align}

\subsection{The Preconditioner and Main Result}
\label{sec:proof of main result}

\begin{theorem}
	\label{thm:location of eigenvalues}
	Let $\bdd{P}$ be defined as in \cref{eq:preconditioner form} with $\bar{\bdd{A}}$ given by \cref{eq:bilinear form to matrix velocity} and $\bar{\bdd{M}}$ given by \cref{eq:bilinear form to matrix pressure}. Then, using $\bdd{P}^{-1}$ as a preconditioner for the MINRES method reduces the norm of the residual of MINRES by a factor of at least $\frac{\sqrt{\sigma} - 1}{\sqrt{\sigma} + 1}$ every two iterations, where $\sqrt{\sigma} \leq C(1 + \log^3 k)$ with $C$ independent of $k$ and $h$.
\end{theorem}
\begin{proof}
	Thanks to \cref{thm:spectral equivalences} and the matrix correspondences \cref{eq:bilinear form to matrix pressure,eq:bilinear form to matrix velocity}, there holds
	\begin{align*}
	%\label{eq:spectral equivalence matrix}
	\left[ C_2 \beta^{-2} (1 + \log^3 k) \right]^{-1} \bar{\bdd{A}} \leq \widetilde{\bdd{A}} \leq C_2 \bar{\bdd{A}} \quad \text{and} \quad C_1^{-1} \bar{\bdd{M}} \leq \widetilde{\bdd{M}} \leq C_1 \bar{\bdd{M}},
	\end{align*} 
	where $\widetilde{\bdd{M}}$ is the pressure mass matrix for the space $Q_I^{\perp}$ and $\bdd{A} \leq \bdd{B}$ means that $\bdd{B} - \bdd{A}$ is positive semidefinite. Additionally, the inf-sup condition for the spaces $\tilde{\bdd{V}}_E \times \tilde{Q}_E$ \cref{eq:inf-sup global boundary spaces} and the boundedness of the bilinear form $b(\cdot, \cdot)$ can be expressed in matrix form using \cref{eq:matrix relations 2} and the same arguments in \cite[Theorem 3.22]{Elman14} to arrive at
	\begin{align*}
	{\beta}^2 \leq \frac{ \vec{q}_{e}^{T} \widetilde{\bdd{B}} \widetilde{\bdd{A}}^{-1} \widetilde{\bdd{B}}^{T} \vec{q}_{e} }{ \vec{q}_{e}^{T} \widetilde{\bdd{M}} \vec{q}_{e} } \leq 1, \quad \forall \tilde{Q}_E \ni q =  \vec{q}_e^{T} \vec{\psi}_{e},
	\end{align*}
	where $\beta$ is the discrete inf-sup constant in \cref{eq:inf-sup global spaces}. Thus, \cref{eq:intro evals location} holds with $\delta = \beta^2 [C_2 (1 + \log^3 k)]^{-1} $, $\Delta = C_2$, $\theta = \beta^2 C_1^{-1}$, and $\Delta = C_1$.

	Let $\vec{r}_{n}$ denote the residual on the $n$-th iteration of MINRES with the preconditioner $\bdd{P}^{-1}$. Applying  \cite[Theorem 4.14]{Elman14} and using the fact that the inf-sup constant $\beta$ is bounded below uniformly in $k$ and $h$ gives
	\begin{align}
	\label{eq:residual norm decrease}
	\|\vec{r}_{2n}\|_{\bdd{P}^{-1}} \leq 2\left( \frac{\sqrt{\sigma} - 1}{\sqrt{\sigma} + 1}  \right)^{n} \| \vec{r}_{0} \|_{\bdd{P}^{-1}},
	\end{align}
	where  $\|\vec{r}\|_{\bdd{P}^{-1}}^2 := \vec{r}^{T} \bdd{P}^{-1} \vec{r}$ and $\sqrt{\sigma} \leq C(1 + \log^3 k)$
	with $C$ independent of $k$ and $h$. Since all norms on finite dimension vector spaces are equivalent, \cref{eq:residual norm decrease} holds for any choice of norm at the expense of replacing ``2" by an appropriate constant depending on the choice of norm, which  completes the proof of \cref{thm:location of eigenvalues}. 
\end{proof}
\cref{thm:location of eigenvalues} shows that the performance of the preconditioner deteriorates at most as $\log^3 k$ as the polynomial order is increased, but remains bounded as the mesh is refined provided the shape regularity assumption \cref{eq:shape regularity} is satisfied.

\subsection{Implementation and Cost Analysis of the Preconditioner}

To aid in the implementation and cost analysis of computing the actions of $\bar{\bdd{A}}^{-1}$ and $\bar{\bdd{M}}^{-1}$, we assume, for convenience, the interface degrees of freedom are ordered as follows:
\begin{enumerate}
	\item[(i)] velocity $C^0$ vertex degrees of freedom,
	\item[(ii)] velocity $C^1$ vertex degrees of freedom,
	\item[(iii)] velocity edge degrees of freedom, grouped according to edge,
	\item[(iv)] pressure vertex degrees of freedom,
	\item[(v)] pressure average value degrees of freedom.
\end{enumerate}
This ordering induces a block structure in the matrix $\widetilde{\bdd{A}}$ in which the diagonal subblocks are: $\widetilde{\bdd{A}}_c$, corresponding to the global interaction among \textit{all} the global $C^0$ vertex functions; $\widetilde{\bdd{A}}_{\bdd{a}, \mu}$, the block-diagonal entry corresponding to the $C^1$ vertex functions $\{\phi_{\bdd{a}}^{\mu}\unitvec{e}_1, \phi_{\bdd{a}}^{\mu}\unitvec{e}_2\}$; whilst $\widetilde{\bdd{A}}_{\gamma}$ corresponds to the interactions among the edge functions associated to $\gamma$. The load vectors can be similarly split into subvectors corresponding to the same groupings of degrees of freedom. The block diagonal structure of $\bdd{P}$ is then exploited to compute the action of $\bdd{P}^{-1}$ on a pair of vectors $\vec{f}$, $\vec{g}$ efficiently or in parallel, as described in \cref{alg:matrix preconditioner}.

\begin{algorithm}[ht]
	\caption{Action of Preconditioner}
	\label{alg:matrix preconditioner}
	\begin{algorithmic}[1]
		\Require{ $\widetilde{\bdd{A}}$, $\vec{f}$, $\vec{g}$}
		\Function{}{}
		
		\State{ $\vec{u}_{c} = \widetilde{\bdd{A}}_{c}^{-1} \vec{f}_{c}$  } \Comment{Global velocity $C^0$ vertex function solve}
		
		\For{ $\bdd{a} \in \mathcal{V}, \unitvec{\mu} \in \mathring{D}_{\bdd{a}}$, } \Comment{Block diagonal velocity $C^1$ vertex function solve}
		\State{ $u_{\bdd{a}, \mu} = \widetilde{\bdd{A}}_{\bdd{a}, \mu}^{-1} \vec{f}_{\bdd{a}, \mu}$ }
		\EndFor
		
		\For{ $\gamma \in \mathcal{E}_I$ } \Comment{Block diagonal velocity edge solve}
		\State{ $\vec{u}_{\gamma} = \widetilde{\bdd{A}}_{\gamma}^{-1} \vec{{f}}_{\gamma}$ }
		\EndFor
		
		\For{ $\bdd{a} \in \mathcal{V}$, $\omega \in \Omega_{\bdd{a}}$ }
		\Comment{Diagonal pressure $C^0$ vertex function solve}
		\State{ $p_{\bdd{a}, \omega} = |\omega|^{-1} k^4 {g}_{\bdd{a}, \omega} $ }
		\EndFor
		
		\For{ $K \in \mathcal{T}$ }
		\Comment{ Diagonal pressure average value solve }
		\State{ $p_{K} = |K|^{-1} {g}_{K}$ }
		\EndFor
		
		\State{\Return{ $\vec{u}_{c}$,  $(\vec{u}_{\bdd{a}, \mu})_{\bdd{a}, \mu}$, $(\vec{u}_{\gamma})_{\gamma}$, $(p_{\bdd{a}, \omega})_{\bdd{a}, \omega}$, $(p_K)_K$} }
		\Comment{Return degrees of freedom}
		
		\EndFunction
	\end{algorithmic}
\end{algorithm}

The cost of computing the action of $\bdd{P}^{-1}$ using \cref{alg:matrix preconditioner} comprises of two parts: one-time setup costs and recurring costs associated with each application of \cref{alg:matrix preconditioner}. The setup cost is dominated by eliminating the interior degrees of freedom on each element, which takes $\mathcal{O}(|\mathcal{T}| k^{6})$ operations needed for the subassembly of the Schur complement. The matrices $\bdd{A}_{c}$, $\bdd{A}_{\bdd{a}, \mu}$, $\bdd{a} \in \mathcal{V}$, $\unitvec{\mu} \in \mathring{D}_{\bdd{a}}$, and $\bdd{A}_{\gamma}$, $\gamma \in \mathcal{E}$, need only be factored once at a cost of $\mathcal{O}(|\mathcal{V}|^3 + |\mathcal{E}| k^3)$ operations, giving an overall setup cost of $
\mathcal{O}( |\mathcal{T}|k^6 + |\mathcal{E}|k^3 + |\mathcal{V}|^3  )$.

We now turn to the cost associated  with each application of \cref{alg:matrix preconditioner}. 
Line 2 of \cref{alg:matrix preconditioner} entails the solution of the linear system $\bdd{A}_c$ involving all of the $C^0$ vertex functions which, thanks to the prefactorisation of $\bdd{A}_c$, costs $\mathcal{O}(|\mathcal{V}|^2)$ operations per solve. Lines 3-5 require the solution of a 2x2 matrix on the velocity $C^1$ vertex functions for each vertex $\bdd{a} \in \mathcal{V}$ and derivative degree of freedom $\unitvec{\mu} \in \mathring{D}_{\bdd{a}}$ at a cost of $\mathcal{O}(|\mathcal{V}|)$ operations. Lines 6-8 entail a block diagonal solve over each of the edges which, again thanks to the prefactorisation of $\bdd{A}_{\gamma}$, $\gamma \in \mathcal{E}$, can be applied using $\mathcal{O}( |\mathcal{E}| k^2 )$ operations. Lines 9-11 require $\mathcal{O}(|\mathcal{V}|)$ operations and lines 12-14 require $\mathcal{O}(|\mathcal{T}|)$ operations by analogous arguments. In summary, the overall cost per application of \cref{alg:matrix preconditioner} is $\mathcal{O}( |\mathcal{E}| k^2 + |\mathcal{V}|^2 + |\mathcal{T}| )$, which is comparable to nonoverlapping domain decomposition methods for second order elliptic problems \cite{ToWi06}.

\section{Numerical Examples}
\label{sec:numerics}

\tikzset{boximg/.style={remember picture,red,thick,draw,inner sep=0pt,outer sep=0pt}}

We illustrate the performance of the preconditioner described in \cref{sec:asm theory} in two numerical examples. 

\subsection{Moffatt Eddies}

In the first example, we revisit the Moffatt problem \cite{Moffatt64} considered in \cite{AinCP19StokesII}, in which the domain $\Omega$ is the wedge with a prescribed parabolic flow profile on the top part of the boundary and no flow on the remainder of the boundary:
\begin{align*}
\bdd{u}(x, 0) = \begin{bmatrix}
1-x^2 \\ 0
\end{bmatrix}, \ -1 \leq x \leq 1, \quad \text{and} \quad \bdd{u}=\bdd{0} \text{ on } \Gamma \setminus (-1,1) \times \{0\}.
\end{align*}
The problem is approximated using a pure $p$-version finite element scheme on the fixed mesh shown in \cref{fig:moffatt mesh}. The results in \cite{AinCP19StokesII} show that the $k=13$ solution resolves four to five eddies, equivalent to a $10^{13}$ range of scales.

Let $\lambda_{min}^{\pm}$ and $\lambda_{max}^{\pm}$ denote the extremal eigenvalues of $\bdd{P}^{-1} \bdd{S}$ so that
\begin{align}
\label{eq:extremal evals def}
\sigma(\bdd{P}^{-1} \bdd{S}) \subseteq [ -\lambda_{max}^{-}, -\lambda_{min}^{-} ] \cup \{ 0 \} \cup [ \lambda_{min}^{+}, \lambda_{max}^{+} ].
\end{align}
According to \cref{thm:location of eigenvalues}, $\lambda_{max}^{\pm} \leq C$ and $\lambda_{min}^{\pm} \geq C(1 + \log^3 k)^{-1}$ with constant $C$ independent of $k$ and $h$. \cref{fig:wedge evals} displays the actual values of the extreme eigenvalues. In agreement with theory, $\lambda_{max}^{\pm}$ is uniformly bounded in $k$ and $\lambda_{min}^{+} \geq C(1 + \log^3 k)^{-1}$. However, $\lambda_{min}^{-}$ appears to remain uniformly bounded in $k$, which would mean that, in practice, the contraction factor in \cref{thm:location of eigenvalues} is pessimistic. The residual history for $k \in \{4, 7, 10, 13\}$ for the preconditioned MINRES solver are displayed in \cref{fig:wedge resids random}. The starting vector is taken to be $\vec{x} + \vec{\epsilon}$, where $\vec{x}$ is the true solution of the Schur complement system \cref{eq:Schur complement system} and $\vec{\epsilon}$ is a random perturbation with entries uniformly distributed in $(-1, 1)$. Here, and in the remaining examples, the relative residual is given by $\sqrt{(\vec{r}^T \bdd{P}^{-1} \vec{r})/(\vec{r}_0^T \bdd{P}^{-1} \vec{r}_0)}$, where $\vec{r}_0$ is the initial residual vector, and MINRES is terminated when the relative residual is smaller than $10^{-8}$. It is observed that, as the polynomial order is raised, the iteration counts grow modestly consistent with the results in \cref{thm:location of eigenvalues}. 

\begin{figure}[htb]
	\centering
	\begin{subfigure}[t]{0.32\linewidth}
		\centering
		\includegraphics[width=\linewidth]{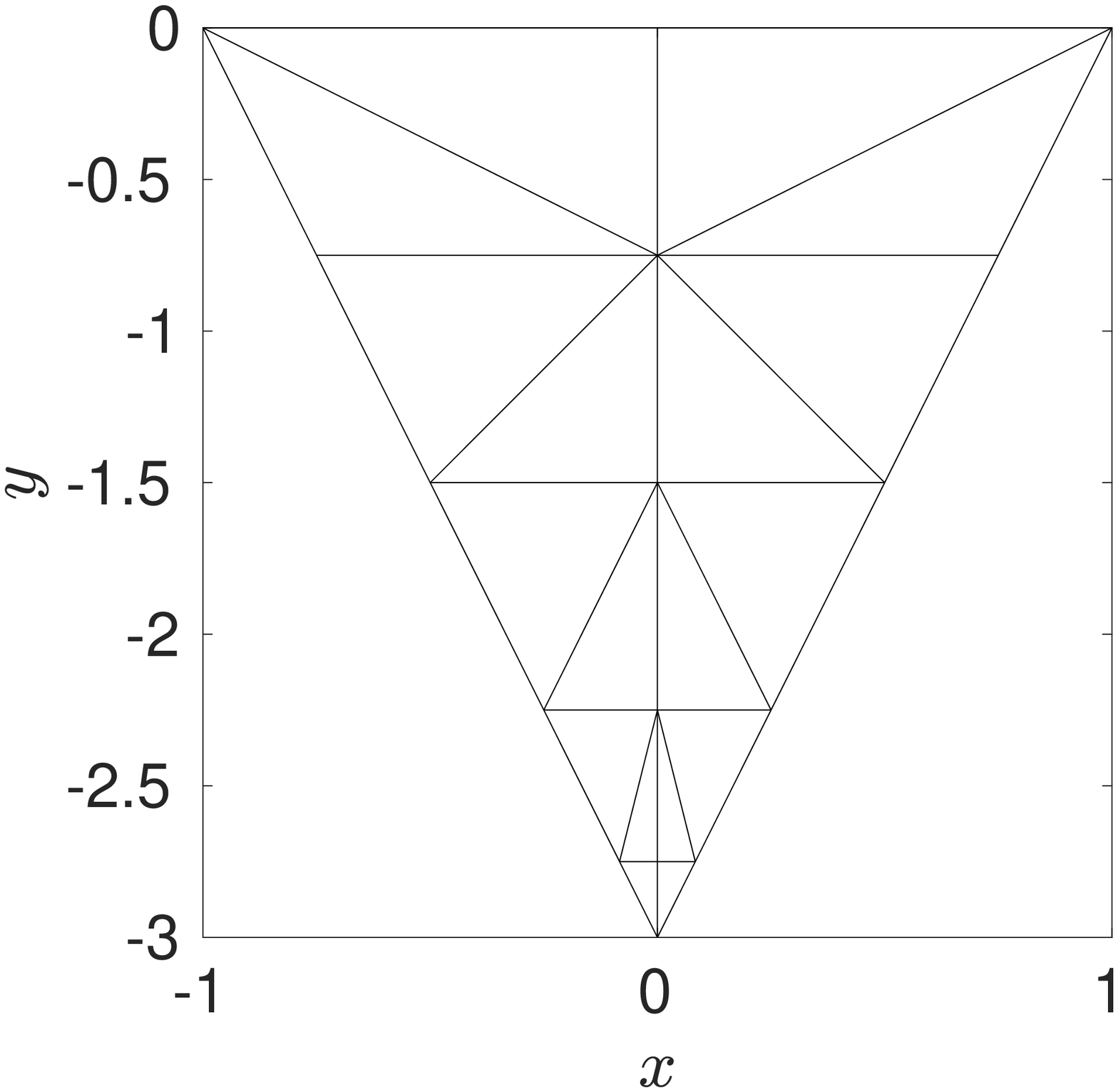}
		\caption{}
		\label{fig:moffatt mesh}
	\end{subfigure}
	\hfill
	\begin{subfigure}[t]{0.32\linewidth}
		\centering
		\includegraphics[width=\linewidth]{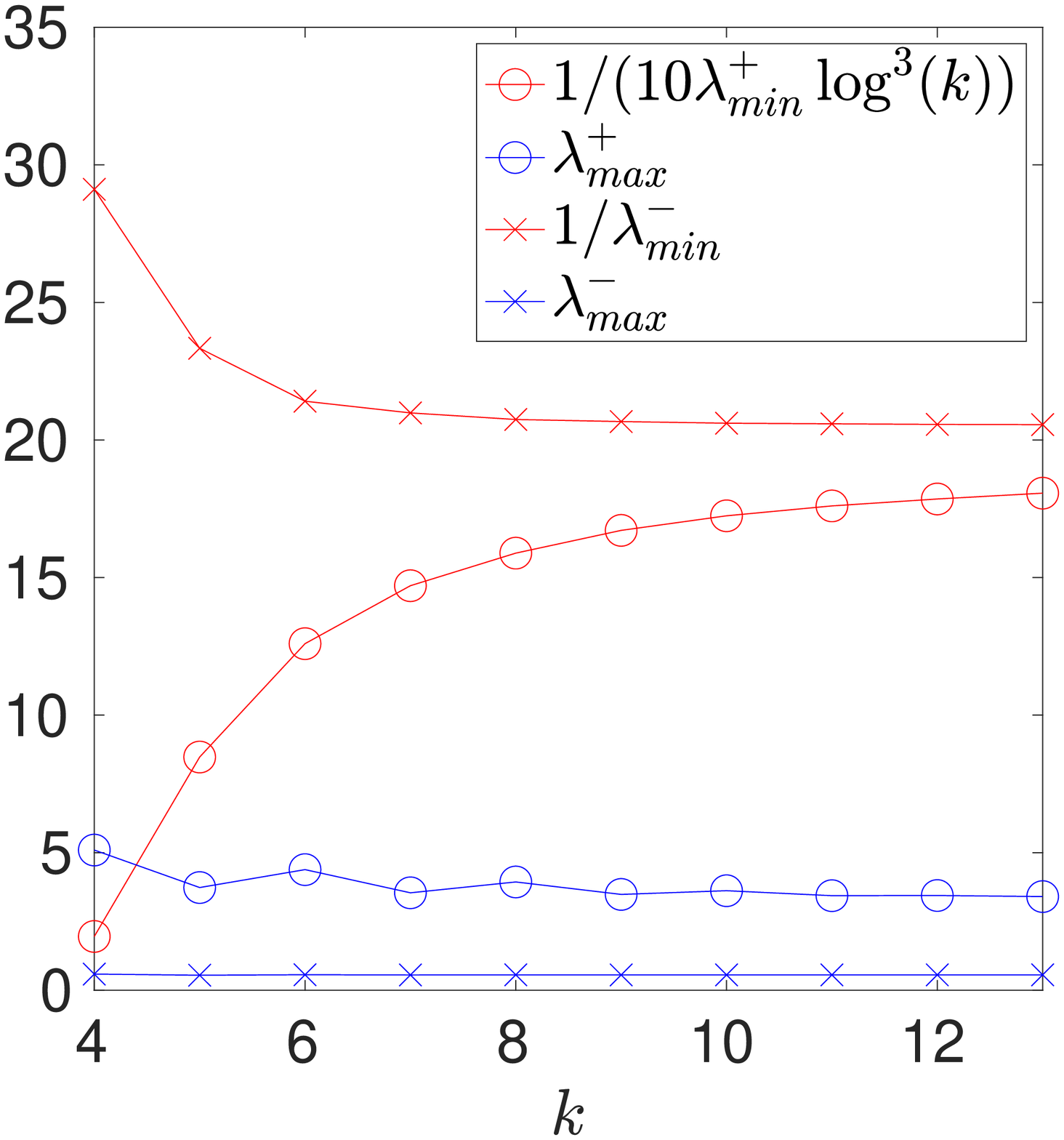}
		\caption{}
		\label{fig:wedge evals}
	\end{subfigure}
	\hfill
	\begin{subfigure}[t]{0.32\linewidth}
		\centering
		\includegraphics[width=\linewidth]{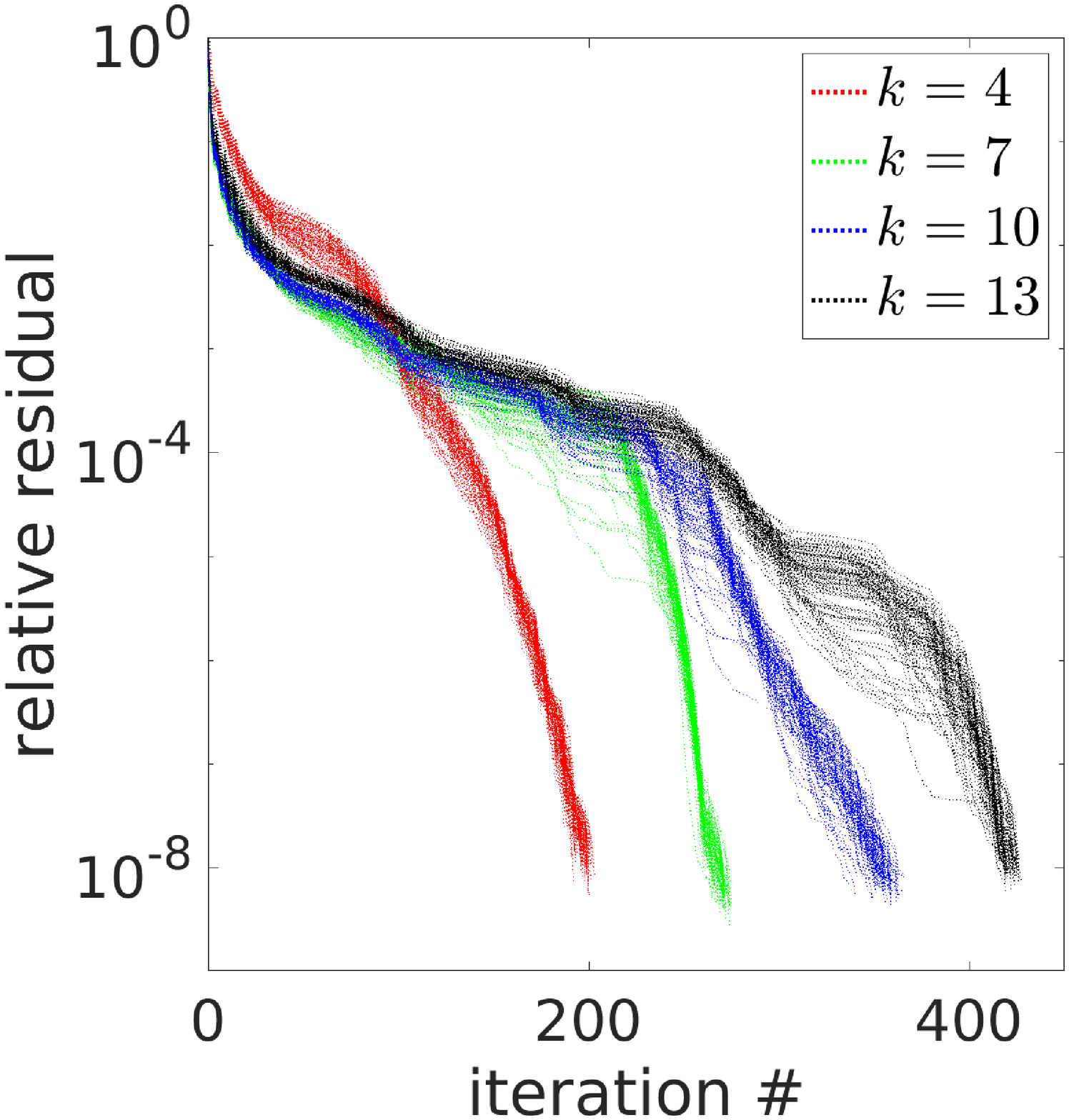}
		\caption{}
		\label{fig:wedge resids random}
	\end{subfigure}
	\caption{(a) 18 element mesh, (b) extremal eigenvalues of $\bdd{P}^{-1} \bdd{S}$, and (c) MINRES convergence history with $k\in \{4, 7, 10, 13\}$, stopping tolerance $10^{-8}$ for a sequence of random initial iterates for Moffatt eddies problem. $1/\lambda_{min}^{+}$ grows as $\log^3 k$ while the other extreme eigenvalues remain bounded as the polynomial degree $k$ is increased.}
\end{figure}

\subsection{T-shaped Domain}

	In the next example, we consider the T-shaped domain example \cite{Ain02} where $\bdd{f} \equiv \bdd{0}$ and boundary conditions are parabolic flow profile on the leftmost and rightmost boundaries of the domain and no flow on the remainder of the boundary:
	\begin{align*}
	\bdd{u}\left(\pm \frac{3}{2}, y\right) &= \begin{bmatrix}
	y(1-y) \\
	0
	\end{bmatrix}, \ 0 \leq y \leq 1, \quad \text{and} \quad \bdd{u} = \bdd{0} \text{ on } \Gamma \setminus \left\{\pm \frac{3}{2} \right\} \times (0,1).
	\end{align*}
	The sequence of meshes is shown in \cref{fig:tshape geo meshes}, in which the elements are geometrically graded and which were proved to give exponential convergence of the finite element solution \cite[\S 7.2]{AinCP19StokesII}. The mesh in \cref{fig:tshape geo mesh1} consists of one layer of elements around the re-entrant corner and the most bottom corners, with a grading factor of $\sigma = 0.08$. We then refine the mesh by successively adding layers of elements such that the innermost layer of elements has a diameter proportional to $\sigma^n$, where $n$ is the number of refinements. For example, the mesh corresponding to three levels is shown in \cref{fig:tshape geo mesh3,fig:tshape geo mesh3 zoom}.
	Observe that, once a mesh contains two or more layers, the shape regularity constant $\kappa$ \cref{eq:shape regularity} changes from 0.1695 to 0.0829 due to the presence of ``needle" elements near the corners. In particular, several estimates in the analysis depend on $\kappa$, and thus we would expect the performance of the preconditioner to be worse for $n \geq 2$ than for $n=1$.

	As with the previous example, the extremal eigenvalues \cref{eq:extremal evals def}, displayed in \cref{fig:tshape evals}, remain bounded independently of the number of levels of geometric refinement, whilst $1/\lambda_{min}^{+}$ increases by a factor of roughly 10 after one level of refinement due to the change in shape regularity mentioned above. This value is an order of magnitude greater than the value of $1/\lambda_{min}^{+}$ observed for the Moffatt ($\kappa = 0.1508$) example and accounts for the increase of the resulting iteration counts observed in the residual histories for $k \in \{4, 7, 10, 13\}$ in \cref{fig:tshape random resids}. Thus, as one might expect, the preconditioner $\bdd{P}^{-1}$ is less effective on meshes containing high aspect ratio elements owing to the fact that the inf-sup constants and inequalities employed in \cref{sec:technical lemmas} all depend on the shape regularity constant appearing in \cref{eq:shape regularity}. Nevertheless, similar to the behavior observed in the previous example, for each fixed $n$, the iteration counts grow modestly in $k$. For each fixed $k$, the iteration counts are bounded in $n$, and remain virtually unchanged for $n \geq 3$.

\begin{figure}[ht]
	\centering
	\begin{subfigure}[t]{0.32\linewidth}
		\centering
		\includegraphics[width=\linewidth]{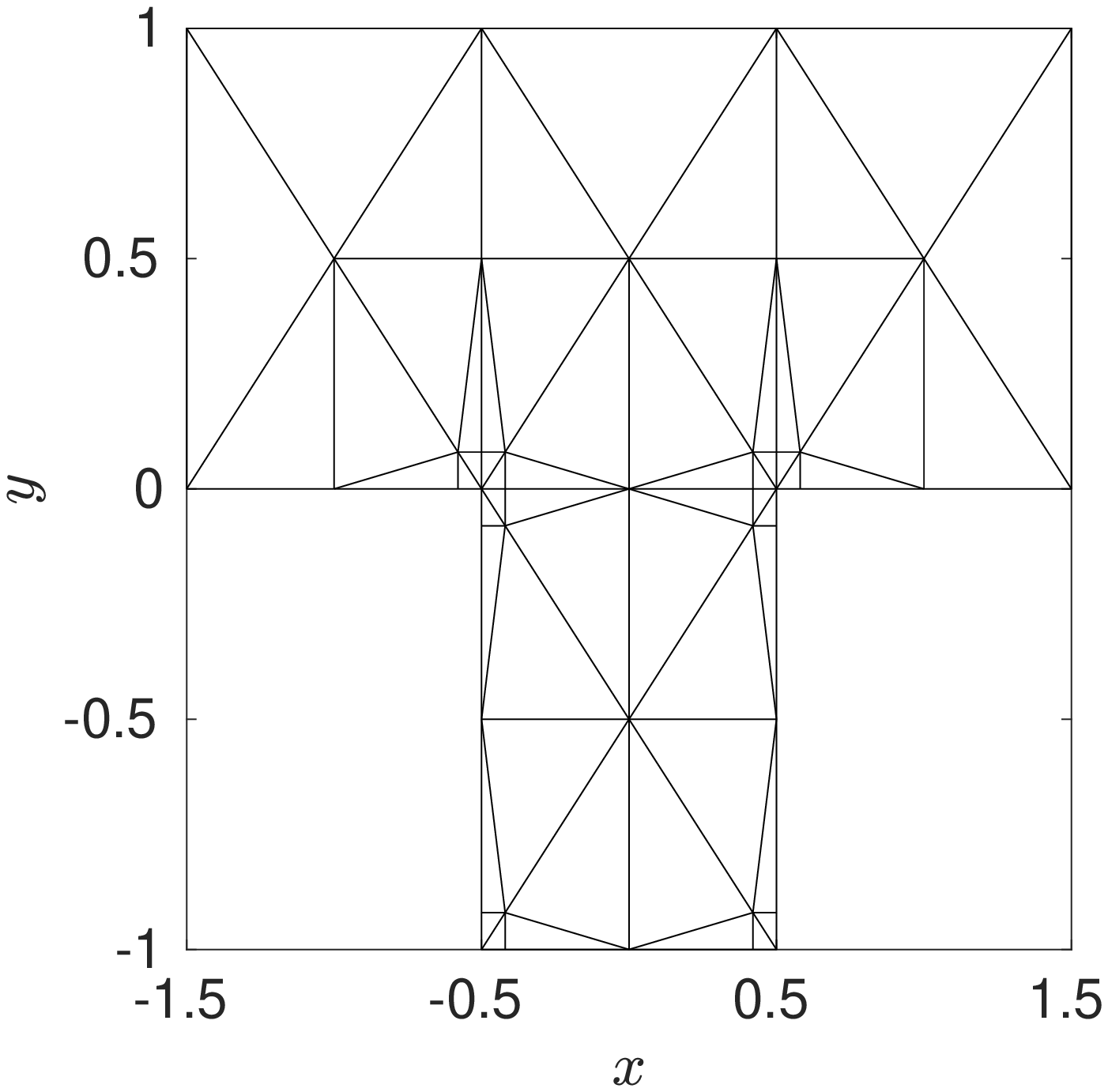}
		\caption{}
		\label{fig:tshape geo mesh1}
	\end{subfigure}
	\hfill
	\begin{subfigure}[t]{0.32\linewidth}
		\centering
		\includegraphics[width=\linewidth]{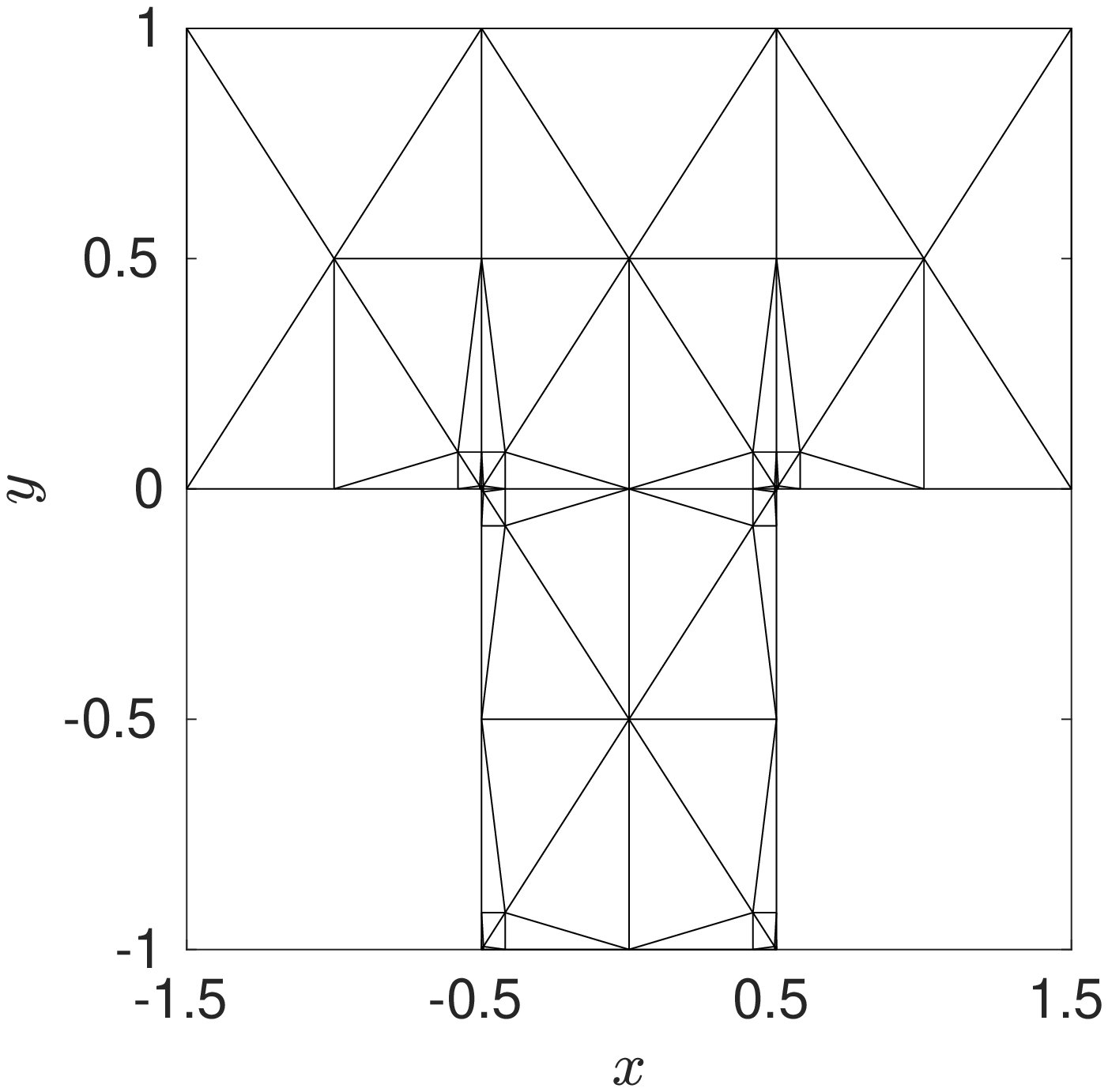}
		\caption{}
		\label{fig:tshape geo mesh3}
	\end{subfigure}
	\hfill
	\begin{subfigure}[t]{0.32\linewidth}
		\centering
		\includegraphics[width=\linewidth]{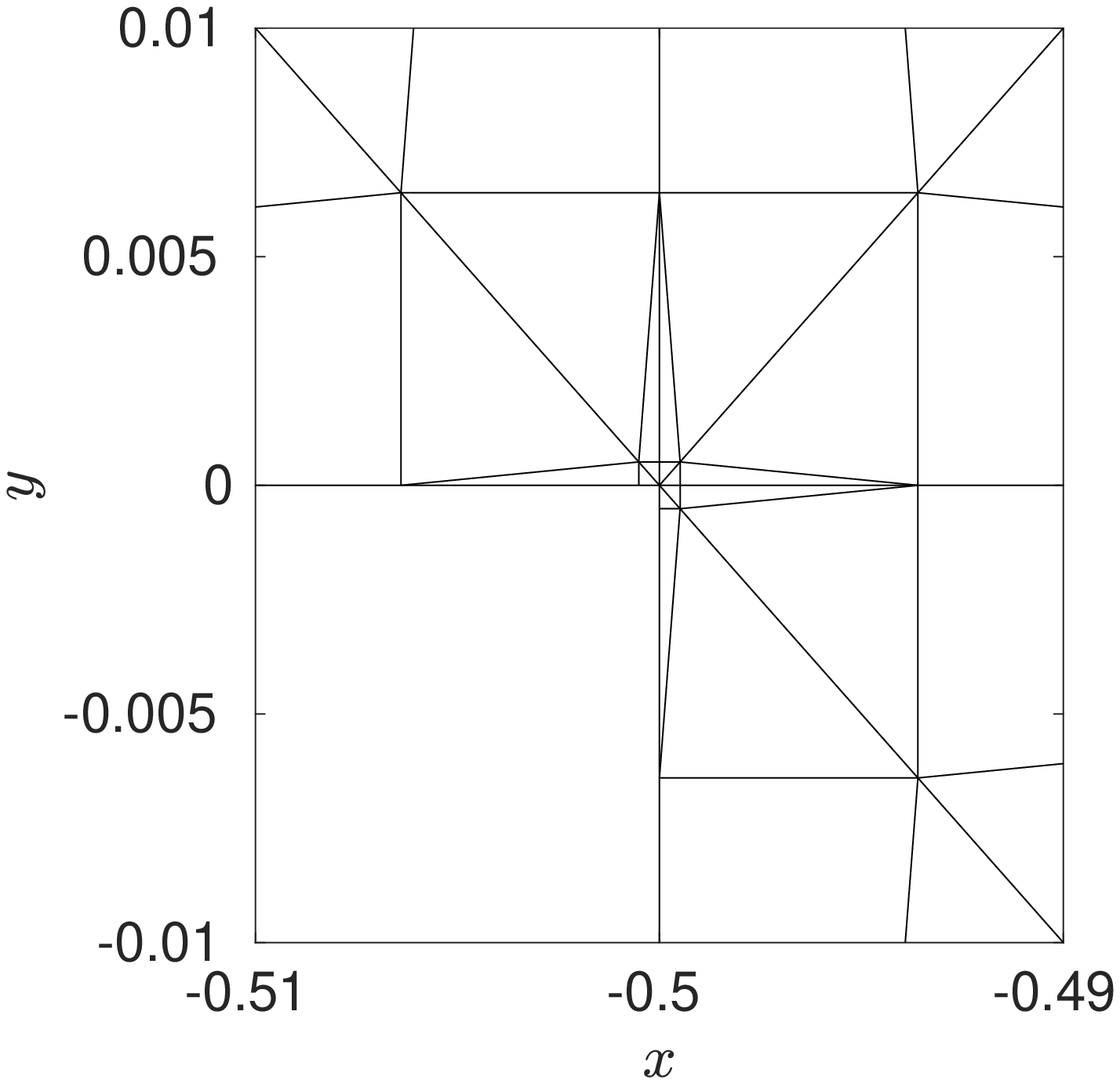}
		\caption{}
		\label{fig:tshape geo mesh3 zoom}
	\end{subfigure}
	\caption{(a) Mesh with $n=1$ layer of elements, (b) Mesh with $n=3$ layers of elements, and (c) Zoom on re-entrant corner of mesh with $n=3$ layers of elements.}
	\label{fig:tshape geo meshes}
\end{figure}

\begin{figure}[ht]
	\centering
	\begin{subfigure}{0.32\linewidth}
		\centering
		\includegraphics[width=\linewidth]{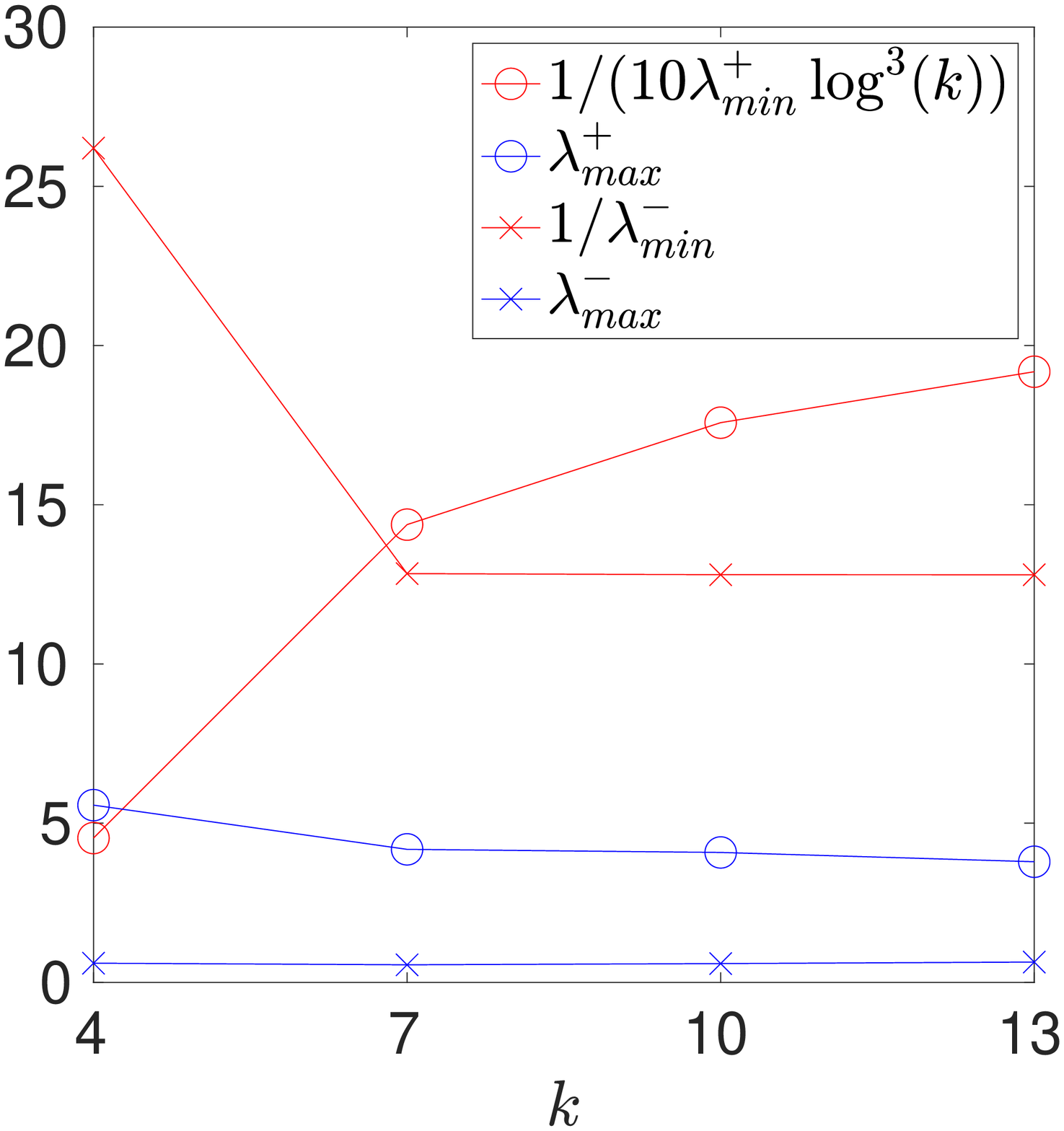}
		\caption{}
		\label{fig:tshape evals n1}
	\end{subfigure}
	\hfill
	\begin{subfigure}{0.32\linewidth}
		\centering
		\includegraphics[width=\linewidth]{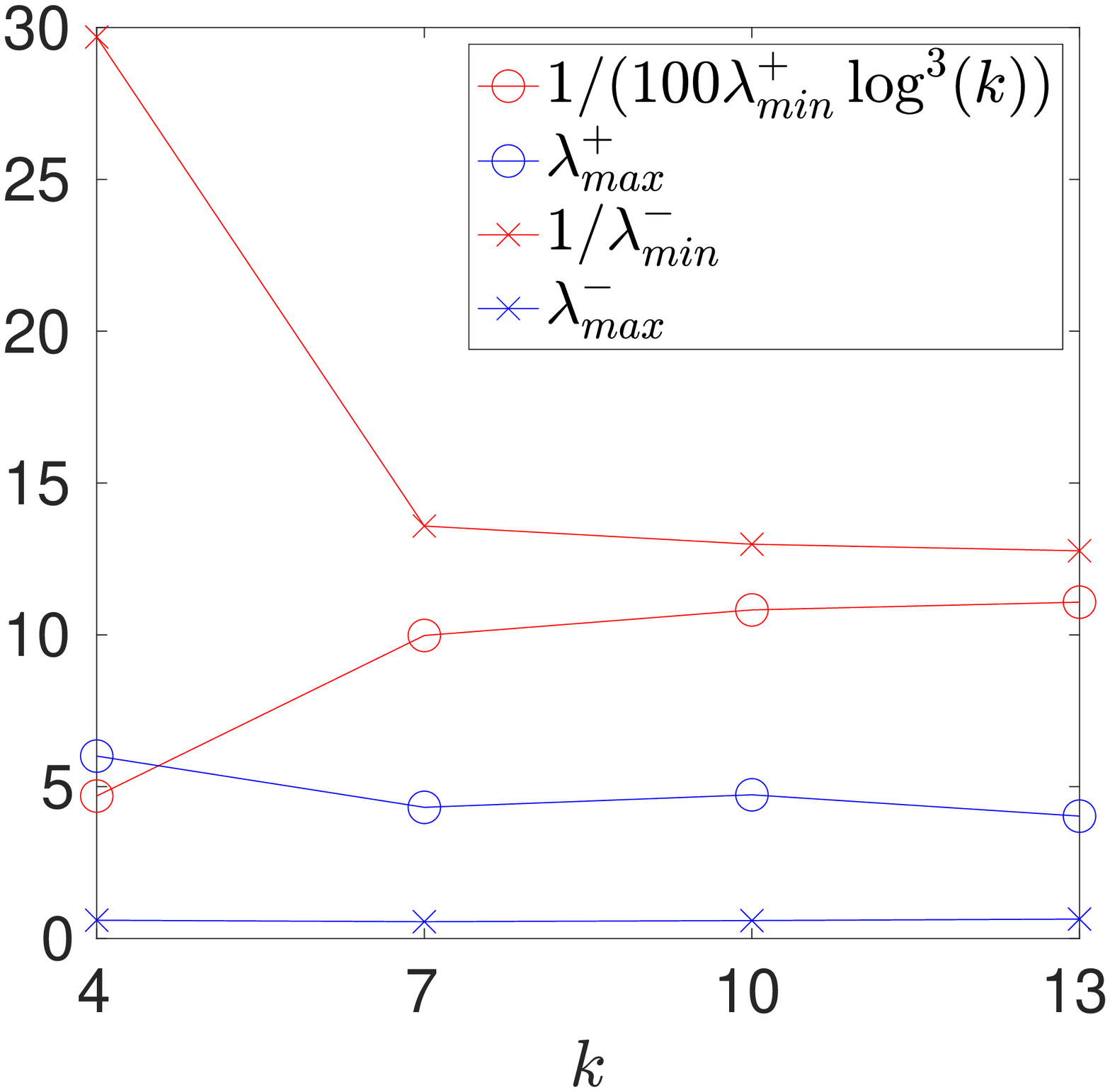}
		\caption{}
		\label{fig:tshape evals n2}
	\end{subfigure} 
	\hfill
	\begin{subfigure}{0.32\linewidth}
		\centering
		\includegraphics[width=\linewidth]{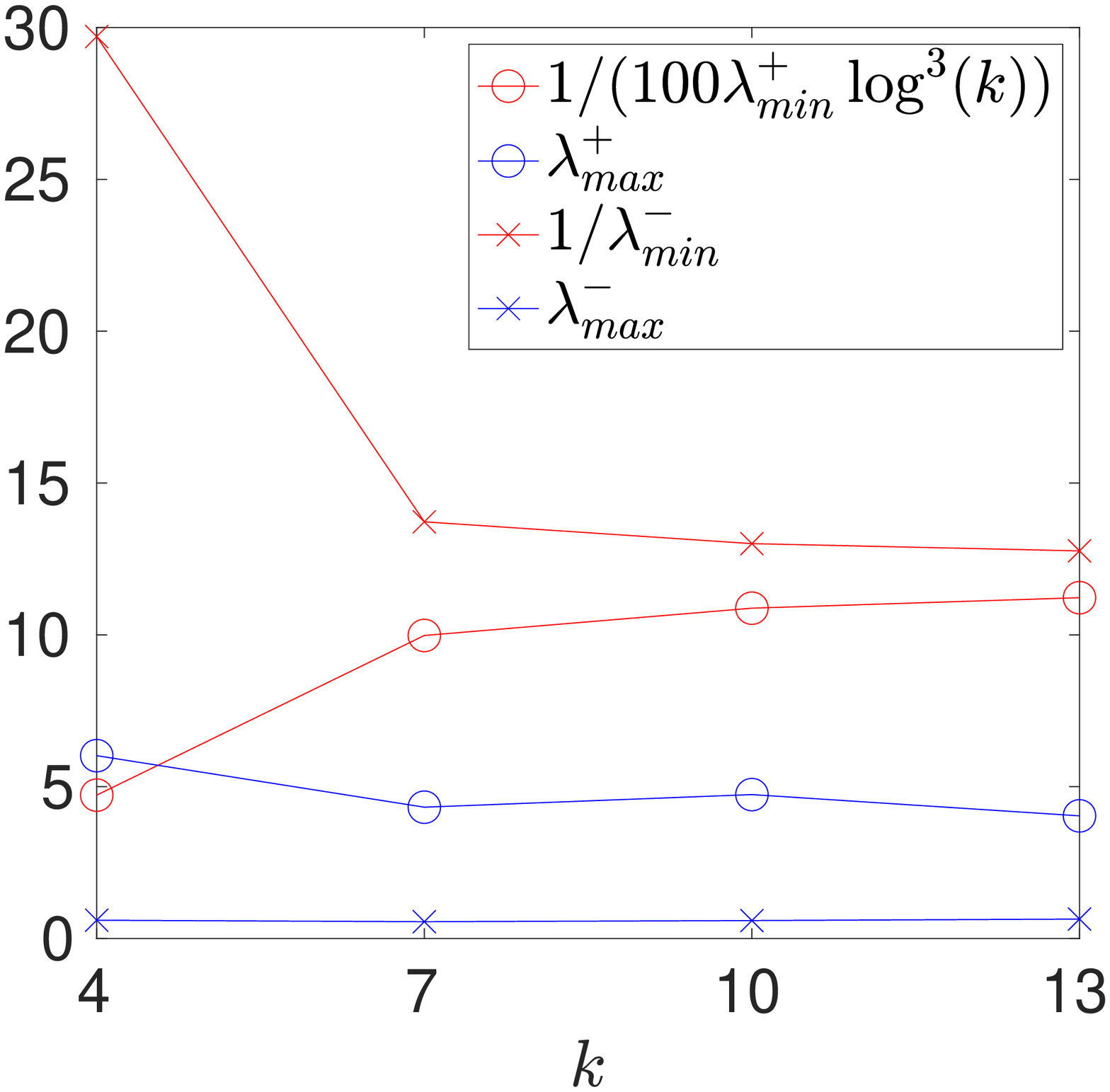}
		\caption{}
		\label{fig:tshape evals n3}
	\end{subfigure}
	\caption{Extremal eigenvalues of $\bdd{P}^{-1} \bdd{S}$ for the T-shape problem with (a) $n=1$, (b) $n = 2$, and  (c) $n=3$ layers of geometrically graded elements at the corners. All of the extreme eigenvalues are uniformly bounded in $n$ for each fixed $k$. In addition, the introduction of small-angle ``needle" elements for $n \geq 2$ greatly increases $1/\lambda_{min}^{+}$.}
	\label{fig:tshape evals}
\end{figure}

\begin{figure}[ht]
	\centering
	\begin{subfigure}{0.32\linewidth}
		\centering
		\includegraphics[width=\linewidth]{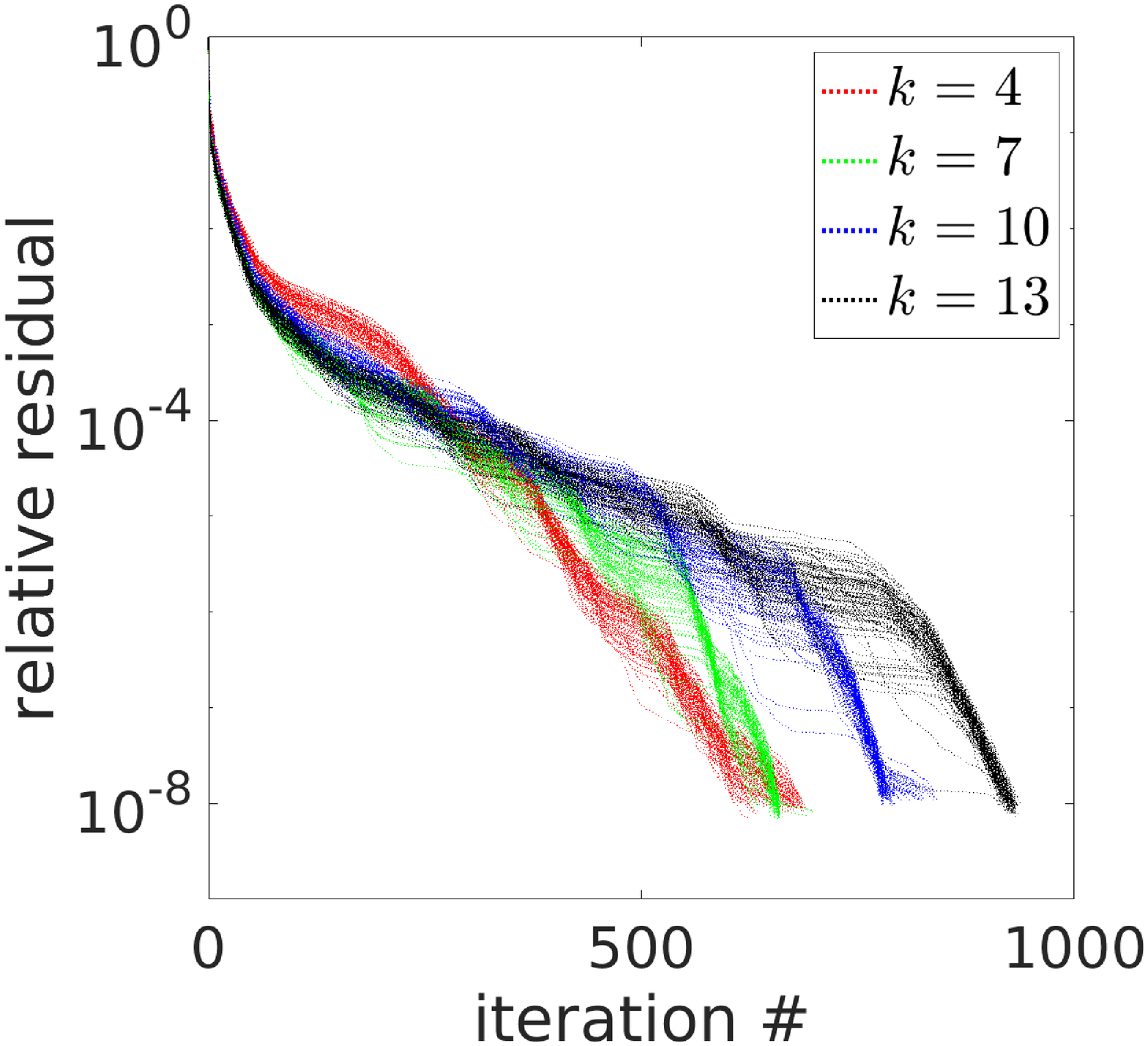}
		\caption{}
		\label{fig:tshape random resid n1}
	\end{subfigure}
	\hfill
	\begin{subfigure}{0.32\linewidth}
		\centering
		\includegraphics[width=\linewidth]{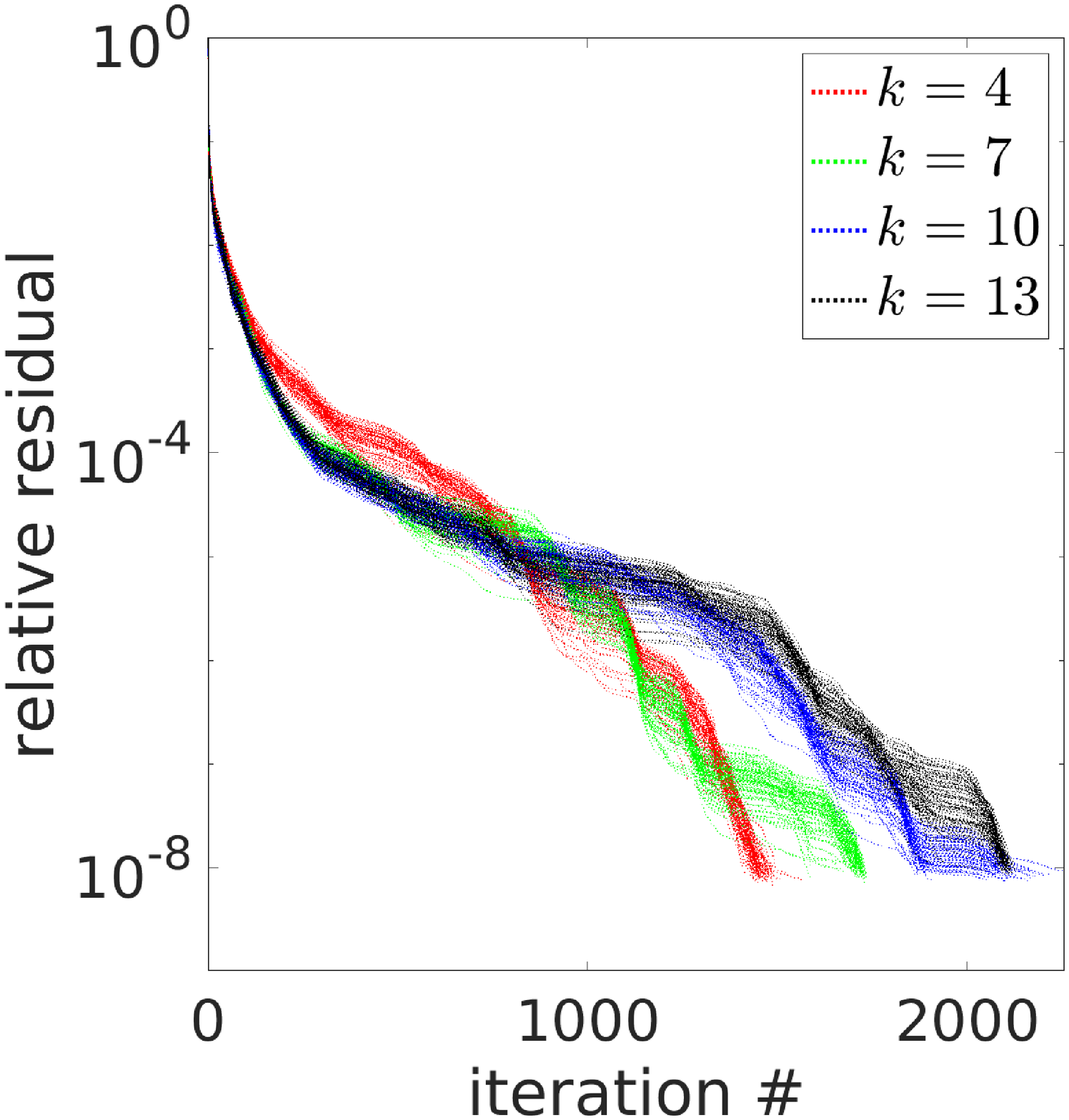}
		\caption{}
		\label{fig:tshape random resid n2}
	\end{subfigure}
	\hfill
	\begin{subfigure}{0.32\linewidth}
		\centering
		\includegraphics[width=\linewidth]{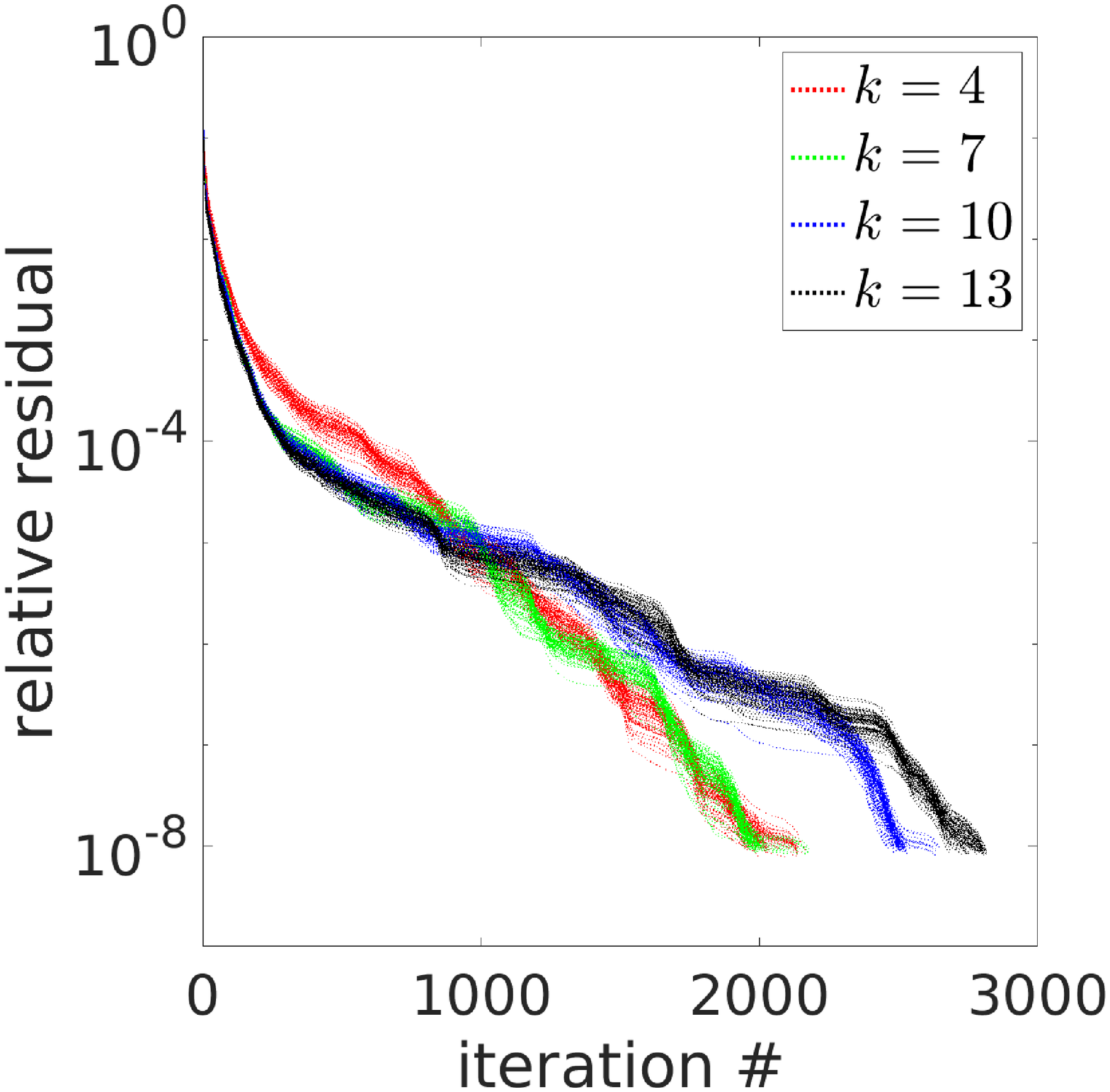}
		\caption{}
		\label{fig:tshape random resid n3}
	\end{subfigure}
	\caption{MINRES convergence history with $k\in \{4, 7, 10, 13\}$, stopping tolerance $10^{-8}$ for a sequence of random initial iterates for the T-shape problem with (a) $n=1$, (b) $n=2$, and (c) $n=3$ layers of geometrically graded elements at the corners.}
	\label{fig:tshape random resids}
\end{figure}

\appendix

\section{Technical Lemmas}
\label{sec:technical lemmas}

In this section, we establish a spectral equivalence of the ASM preconditioners given in \cref{sec:asm theory} to the inner products appearing in the Stokes equations. The main result is the following theorem, which is an immediate consequence of \cref{lem:glorified cs pressure,lem:pressure stability,lem:glorified cs velocity,lem:velocity stable decomp} proved later in this section:
\begin{theorem}
	\label{thm:spectral equivalences}
	There exists positive constants $C_1$ and $C_2$, independent of $k$ and $h$, such that
	\begin{align}
	\label{eq:pressure spectral equivalence}
	C_1^{-1} (p, p) \leq \bar{m}(p, p) \leq C_1 (p, p) \quad \forall p \in Q_I^{\perp},
	\end{align}
	and
	\begin{align}
	\label{eq:velocity spectral equivalence}
	C_2^{-1} a(\bdd{u}, \bdd{u}) \leq \bar{a}(\bdd{u}, \bdd{u}) \leq C_2 \beta^{-2} (1 + \log^3 k) a(\bdd{u}, \bdd{u}) \quad \forall \bdd{u} \in \tilde{\bdd{V}}_E,
	\end{align}
	where $\bar{m}(\cdot, \cdot)$ is defined in \cref{eq:bar c definition}, $\bar{a}(\cdot, \cdot)$ is defined in \cref{eq:bar a definition}, and $\beta$ is the inf-sup constant defined in \cref{eq:inf-sup global spaces}.
\end{theorem}

\subsection{Pressure ASM}

We begin with the pressure ASM. The first lemma establishes a key estimate for the norm of the pressure vertex functions:
\begin{lemma}
	The pressure vertex functions functions $\tilde{\psi}_{\bdd{a}}^{\omega}$, $\bdd{a} \in \mathcal{V}$, $\omega \in \Omega_{\bdd{a}}$, satisfy
	\begin{align}
	\label{eq:psi decay}
	\| \tilde{\psi}_{\bdd{a}}^{\omega} \|_{L^2(K)} \leq C h_K k^{-2} \quad \forall K \in \mathcal{T},
	\end{align}
	with $C$ independent of $k$, $h_K$, and $\bdd{a}$.
\end{lemma}
\begin{proof}
	Let $\bdd{a} \in \mathcal{V}$, $\omega \in \Omega_{\bdd{a}}$. Define the function $\chi \in Q$ by the rule $\chi = \chi_K$ on each element $K \in \mathcal{T}$ where $\chi_K$ is chosen as in \cite[Lemma 4.1]{AinCP19StokesII}. In particular, $\chi_K \in \mathcal{P}_{k-1}(K) \cap L^2_0(K)$ satisfies (i) $\chi_K \equiv 0$ if $K \nsubseteq \omega$, (ii) $\chi_K(\bdd{b}) = \delta_{ \bdd{a} \bdd{b}}$ for $\bdd{b} \in \mathcal{V}_K$, and (iii) $\|\chi_K\|_{L^2(K)} \leq C h_K k^{-2}$ with $C$ independent of $h_K$, $k$, and $\bdd{a}$. By \cref{eq:stokes extension u and p 4}, \cref{eq:stokes extension u and p 5}, and \cref{thm:extension pressure continuity}, $\tilde{\psi}_{\bdd{a}} = \tilde{\Pi} \chi$, and $\| \tilde{\psi}_{\bdd{a}}^{\omega} \|_{L^2(K)} \leq \|\chi\|_{L^2(K)} \leq C h_K k^{-2}$ $\forall K \in \mathcal{T}$.
\end{proof}

We now show that the inner products on the subspaces are coercive:
\begin{lemma}
	\label{lem:pressure solvability}
	There exists a positive constant $C$ independent of $k$ and $h$ such that
	\begin{align*}
	(p, p) &\leq C m_{\bdd{a}, \omega}(p, p) & &\forall p \in \tilde{Q}_{\bdd{a},\omega}, \quad \bdd{a} \in \mathcal{V}, \quad \omega \in \Omega_{\bdd{a}}, \\
	(p, p) &\leq C m_{K}(p, p) & &\forall p \in \tilde{Q}_{K}, \quad K \in \mathcal{T}.
	\end{align*}
\end{lemma}
\begin{proof}
	Let $p \in \tilde{Q}_{\bdd{a}, \omega}$, $\bdd{a} \in \mathcal{V}$, $\omega \in \Omega_{\bdd{a}}$. Then, $p = p(\bdd{a}) \tilde{\psi}_{\bdd{a}}^{\omega}$, and by \cref{eq:psi decay} and shape regularity \cref{eq:shape regularity}, there holds
	\begin{align*}
	(p,p) \leq C \sum_{K \in \mathcal{T} : K \subseteq \omega}  h_K^2 k^{-4} |p(\bdd{a})|^2 \leq C |\omega| k^{-4} = m_{\bdd{a}, \omega}(p, p).
	\end{align*}
	Now let $p \in Q_{K}$, $K \in \mathcal{T}$. Then, $p = (|K|^{-1} \int_{K} p \ d\bdd{x}) \tilde{\psi}_K$ and since $\tilde{\psi}_{\bdd{a}}^{\omega} \in L^2_0(K)$,
	\begin{align*}
	\int_{K} \tilde{\psi}_K^2 \ d\bdd{x} = \int_{K} \left\{ 1 + \left( \sum_{ \substack{\bdd{a} \in \mathcal{V}_K \\ \omega \supseteq K} } \tilde\psi_{\bdd{a}}^{\omega} \right)^2 \right\} \ d\bdd{x} \leq |K| + Ch_K^2 k^{-4} \leq C |K|.
	\end{align*}
	Thus, $(p, p) \leq C m_{K}(p,p)$.
\end{proof}
We are now able to establish the left-hand side of the equivalence \cref{eq:pressure spectral equivalence}:
\begin{lemma}
	\label{lem:glorified cs pressure}
	There exists a constant $C$ independent of $k$ and $h$ such that
	\begin{align}
	\label{eq:glorified cs pressure}
	(p, p) \leq C \bar{m}(p, p) \quad \forall p \in Q_I^{\perp}.
	\end{align}
\end{lemma}
\begin{proof}
	Let $p \in Q_I^{\perp}$. By Cauchy-Schwarz, there holds
	\begin{align*}
	(p, p)_K \leq 4 \left[ \sum_{ \substack{ \bdd{a} \in \mathcal{V}_{K} \\ \omega \supseteq K} } (p_{\bdd{a}, \omega}, p_{\bdd{a}, \omega})_K + (p_{K}, p_{K})_K \right]
	\end{align*}
	where $(p, q)_{K} := \int_{K} p q \ d\bdd{x}$.
	\cref{eq:glorified cs pressure} now follows from \cref{lem:pressure solvability} and summing over the elements.
\end{proof}
The right-hand side of the equivalence \cref{eq:pressure spectral equivalence} is covered by the next result:
\begin{lemma}
	\label{lem:pressure stability}
	There exists a positive constant $C$ independent of $k$ and $h$ such that
	\begin{align}
	\label{eq:pressure stability}
	\bar{m}(p,p) \leq C (p,p) \quad \forall p \in Q_I^{\perp}.
	\end{align}
\end{lemma}
\begin{proof}
	By \cite[Lemma 6.1]{AinJia19}, there holds
	\begin{align*}
	|p|_{K}(\bdd{a})|^2 k^{-4} = |(p|_{K} \circ \bdd{F}_K)(\hat{\bdd{a}})|^2  k^{-4} \leq C \| p \circ \bdd{F}_K \|_{L^2(\hat{T})}^2 \leq C h_K^2 \|p\|_{L^2(K)}^2
	\end{align*}
	with $\hat{\bdd{a}} = \bdd{F}_K^{-1}(\bdd{a})$, and by shape regularity,
	\begin{align*}
	m_{\bdd{a}, \omega}(p_{\bdd{a}, \omega}, p_{\bdd{a}, \omega}) =  |\omega| k^{-4} |p|_{{\omega}}(\bdd{a})|^2   \leq C \sum_{K \in \mathcal{T} : K \subseteq \omega } \|p\|_{L^2(K)}^2.
	\end{align*}
	Summing over $\bdd{a} \in \mathcal{V}$, $\omega \in \Omega_{\bdd{a}}$ and again using shape regularity to bound the overlap $|\{ \omega : \exists \bdd{a} \in \mathcal{V} : K \subseteq \omega \in \Omega_{\bdd{a}} \}|$ gives
	\begin{align*}
	\sum_{ \substack{\bdd{a} \in \mathcal{V} \\ \omega \in \Omega_{\bdd{a}} } } m_{\bdd{a}, \omega}(p_{\bdd{a}, \omega}, p_{\bdd{a}, \omega}) \leq C \sum_{ \substack{\bdd{a} \in \mathcal{V} \\ \omega \in \Omega_{\bdd{a}} } } \sum_{K \in \mathcal{T} : K \subseteq \omega } \|p\|_{L^2(K)}^2 \leq C \|p\|_{L^2(\Omega)}^2.
	\end{align*}
	To bound the remaining $m_K(\cdot, \cdot)$ terms, we use Cauchy-Schwarz:
	\begin{align*}
	\sum_{K \in \mathcal{T}} m_{K}(p_K, p_K) = \sum_{K \in \mathcal{T}} \frac{1}{|K|} \left( \int_{K} p \ d\bdd{x} \right)^2 \leq C \sum_{K \in \mathcal{T}} \|p\|_{L^2(K)}^2 \leq C \|p\|_{L^2(\Omega)}^2,
	\end{align*}
	which completes the proof of \cref{eq:pressure stability}.
\end{proof}

\subsection{Velocity ASM}
We now turn to the velocity space, and start by extending the decomposition \cref{eq:u subspace decomp} as follows. For $\bdd{u} \in \tilde{\bdd{V}}_E$, we define $\bdd{u}_{\gamma} \equiv \bdd{0}$ for $\gamma \in \mathcal{E} \setminus \mathcal{E}_I$ and $\bdd{u}_{\bdd{a}, \mu} \equiv \bdd{0}$ for $\bdd{a} \in \mathcal{V}$, $\unitvec{\mu} \in D_{\bdd{a}} \setminus \mathring{D}_{\bdd{a}}$. Since the inner product on each of the subspace was taken to be $a(\cdot, \cdot)$, we immediately obtain the left-hand side of the equivalence \cref{eq:velocity spectral equivalence}:
\begin{lemma}
	\label{lem:glorified cs velocity}
	For all $\bdd{u} \in \tilde{\bdd{V}}_E$, there holds
	\begin{align}
	\label{eq:glorified cs velocity}
	a(\bdd{u}, \bdd{u}) \leq 10 \bar{a}(\bdd{u}, \bdd{u}).
	\end{align}
\end{lemma}
\begin{proof}
	First recall that there are exactly 2 directional derivative degrees of freedom per velocity component per vertex on any given element, i.e. for $K \in \mathcal{T}$, $| \{ \unitvec{\mu} : K \subseteq \supp \phi_{\bdd{a}}^{\mu} \} | = 2$.
	By Cauchy-Schwarz, there holds
	\begin{align*}
	|\bdd{u}|_{\bdd{H}^1(K)}^2 \leq 10  \left\{ |\bdd{u}_c|_{\bdd{H}^1(K)}^2 + \sum_{\substack{\bdd{a} \in \mathcal{V}_K \\ \unitvec{\mu} : K \in \supp \phi_{\bdd{a}}^{\mu} } } |\bdd{u}_{\bdd{a},\mu}|_{\bdd{H}^1(K)}^2 + \sum_{\gamma \in \mathcal{E}_K} |\bdd{u}_{\gamma}|_{\bdd{H}^1(K)}^2  \right\}, & &\forall K \in \mathcal{T}.
	\end{align*}
	\Cref{eq:glorified cs velocity} now follows by summing over the elements.
\end{proof}

To prove the right-hand side of \cref{eq:velocity spectral equivalence}, we need to establish some properties of the velocity vertex functions:
\begin{lemma}
	\label{lem:velocity vertex functions}
	The $C^0$ velocity vertex functions satisfy the following: For $K \in \mathcal{T}$,
	\begin{align}
	\label{eq:constants in span}
	\mathbb{R}^2 \ni \bdd{c} = \sum_{ \bdd{a} \in \mathcal{V}_K } \sum_{i=1}^{2} (\bdd{c} \cdot \unitvec{e}_i) \bdd{\Pi}_{\bdd{V}} (\phi_{\bdd{a}}\unitvec{e}_i) \quad \text{on } K
	\end{align}
	and
	\begin{align}
	\| \bdd{\Pi}_{\bdd{V}} (\phi_{\bdd{a}}\unitvec{e}_i) \circ \bdd{F}_K \|_{\bdd{H}^1(\hat{T})} &\leq C \quad \forall \bdd{a} \in \mathcal{V}, \ i=1,2, \label{eq:velocity c0 vertex decay}
	\end{align}
	where $C$ depends only on the shape regularity parameter.
	
	Moreover, the $C^1$ velocity vertex functions satisfy
	\begin{align}
	\label{eq:velocity c1 vertex decay}
	\| \bdd{\Pi}_{\bdd{V}} (\phi_{\bdd{a}}^{\mu} \unitvec{e}_i) \circ \bdd{F}_K \|_{\bdd{H}^1(\hat{T})} &\leq C \|D\bdd{F}_K\|_{L^{\infty}(\hat{T})} k^{-2} & & \bdd{a} \in \mathcal{V}, \ \unitvec{\mu} \in D_{\bdd{a}}, \ i=1,2, 
	\end{align}
	where $C$ depends only on the shape regularity parameter.
\end{lemma}
\begin{proof}
	Let $K \in \mathcal{T}$. A simple computation reveals that $	1 = (\lambda_1 + \lambda_2 + \lambda_3)^3 = \sum_{ \bdd{a} \in \mathcal{V}_K } \phi_{\bdd{a}} + 6\lambda_1 \lambda_2 \lambda_3$ on $K$, where $\{ \lambda_i: \ 1 \leq i \leq 3\}$ are the barycentric coordinates on $K$, and hence, for any $\bdd{c} \in \mathbb{R}^2$,
	\begin{align*}
	\bdd{c} = \sum_{\bdd{a} \in \mathcal{V}} \sum_{i=1}^{2} (\bdd{c}\cdot \unitvec{e}_i) (\phi_{\bdd{a}} \unitvec{e}_i) + \underbrace{ 6\bdd{c} \sum_{K \in \mathcal{T}} \lambda_1 \lambda_2 \lambda_3 }_{\in \bdd{V}_I} .
	\end{align*}
	Applying $\bdd{\Pi}_{\bdd{V}}$ to both sides of this identity and noting that \cref{thm:extension velocity continuity} gives $\bdd{\Pi}_{\bdd{V}} : \bdd{V}_I \to \{ \bdd{0} \}$, we obtain
	\begin{align*}
	\bdd{\Pi}_{\bdd{V}} \bdd{c} = \sum_{\bdd{a} \in \mathcal{V}} \sum_{i=1}^{2} (\bdd{c}\cdot \unitvec{e}_i) \bdd{\Pi}_{\bdd{V}} (\phi_{\bdd{a}} \unitvec{e}_i).
	\end{align*}
	Finally, \cref{thm:extension velocity continuity} implies that
	$\mathscr{E}(\bdd{c}, 0) = (\bdd{c}, 0)$ and \cref{eq:constants in span} follows at once.
	
	Now let $K \in \mathcal{T}$ and $i \in \{1,2\}$. Clearly \cref{eq:velocity c0 vertex decay} holds if $\bdd{a} \notin \mathcal{V}_K$ since $\phi_{\bdd{a}} \unitvec{e}_i = \bdd{0}$. Otherwise, if $\bdd{a} \in \mathcal{V}_K$, we apply a scaling argument in conjunction with \cref{eq:stokes extension continuity u dagger} to arrive at
	\begin{align*}
	\frac{1}{|K|^2} \|\bdd{\Pi}_{\bdd{V}}  (\phi_{\bdd{a}} \unitvec{e}_i) \|_{\bdd{L}^2(K)}^2 + | \bdd{\Pi}_{\bdd{V}} (\phi_{\bdd{a}} \unitvec{e}_i) |_{\bdd{H}^1(K)}^2 \leq C \left\{ | \phi_{\bdd{a}}|_{H^{1/2}(\partial K)}^2 + \frac{1}{|\partial K|} \|\phi_{\bdd{a}}\|_{L^2(\partial K)}^2  \right\}, 
	\end{align*}
	where $C$ is a positive constant independent of $k$ and $h_K$. Thus,
	\begin{align*}
	\| \bdd{\Pi}_{\bdd{V}} (\phi_{\bdd{a}}\unitvec{e}_i) \circ \bdd{F}_K \|_{\bdd{H}^1(\hat{T})} &\leq C \| \hat{\phi}_{j} \|_{H^{1/2}(\partial \hat{T})} \leq C \| \hat{\phi}_{j} \|_{H^{1}(\hat{T})} \leq C,
	\end{align*}
	where $\hat{\bdd{a}}_j = \bdd{F}_K^{-1}(\bdd{a})$. For $K \subseteq \supp \phi_{\bdd{a}}^{\mu}$, we argue similarly and use \cref{eq:vertex c1 reference to physical} to obtain
	\begin{align*}
	\| \bdd{\Pi}_{\bdd{V}}  (\phi_{\bdd{a}}^{\mu} \unitvec{e}_i) \circ \bdd{F}_K \|_{\bdd{H}^1(\hat{T})} &\leq C \left\| \begin{bmatrix}
	\unitvec{\mu} & \unitvec{\xi}
	\end{bmatrix}^{-1} \right\| \|D\bdd{F}_K\|_{L^{\infty}(\hat{T})}  \left\| \begin{bmatrix}
	\hat{\phi}_j^{(1,0)} \\
	\hat{\phi}_j^{(0,1)}
	\end{bmatrix} \right\|_{\bdd{H}^{1/2}(\partial \hat{T})},
	\end{align*}
	where $\unitvec{\mu} \neq \unitvec{\xi} \in D_{\bdd{a}}$ is chosen such that $\supp \phi_{\bdd{a}}^{\xi} \supseteq K$, $\hat{\bdd{a}}_j = \bdd{F}_K^{-1}(\bdd{a})$, and $\|\cdot\|$ is any matrix norm. By the definition of $D_{\bdd{a}}$ \cref{eq:derivative directions} and shape regularity,  $\left\| \begin{bmatrix}
		\unitvec{\mu} & \unitvec{\xi}
	\end{bmatrix}^{-1} \right\|$ is uniformly bounded by a constant depending only on $\kappa$ \cref{eq:shape regularity}. Since $\hat{\phi}_{j}^{(1,0)}(\hat{\bdd{a}}) = \hat{\phi}_{j}^{(0,1)}(\hat{\bdd{a}}) = {0}$ for $\hat{\bdd{a}} \in \mathcal{V}_{\hat{T}}$ by the construction \cref{eq:vec J definition}, there holds
	\begin{align*}
	\| \bdd{\Pi}_{\bdd{V}}  (\phi_{\bdd{a}}^{\mu} \unitvec{e}_i) \circ \bdd{F}_K \|_{\bdd{H}^1(\hat{T})} &\leq C \|D\bdd{F}_K\|_{L^{\infty}(\hat{T})} \|J\|_{H^{1/2}_{00}(I)}
	\end{align*}
	where $I = (-1, 1)$, $H^{1/2}_{00}(I)$ is the usual Sobolev space (defined as, e.g. \cite{Lions12}), and
	\begin{align*}
	J(t) = \frac{1}{P^{(3,3)}_{k-3}(-1)} \left( \frac{1+t}{2} \right) \left( \frac{1-t}{2} \right)^2 P^{(3,3)}_{k-3}(t). 
	\end{align*}
	Thanks to \cite[Lemma B.1]{AinCP19Precon}, $\|J\|_{L^2(I)} \leq C k^{-3}$ with $C$ independent of $k$. Using interpolation, and the inverse estimate $\|J'\|_{L^2(I)} \leq C k^2 \|J\|_{L^2(I)}$ \cite[Lemma 5.4]{Bern89}, we obtain $\|J\|_{H^{1/2}_{00}(I)} \leq C \|J\|_{L^2(I)}^{1/2} \|J\|_{H^1(I)}^{1/2} \leq C k^{-2}$, which completes the proof of \cref{eq:velocity c1 vertex decay}.	
\end{proof}

We now use the properties of the vertex functions to prove element-wise stability of the subspace decomposition \cref{eq:u subspace decomp}:
\begin{lemma}
	\label{lem:velocity reference stability}
	For $\bdd{u} \in \tilde{\bdd{V}}_E$ and $K \in \mathcal{T}$, there holds
	\begin{align}
	\label{eq:velocity reference stability}
	| \bdd{u}_c |_{\bdd{H}^1(K)}^2 + \sum_{ \substack{\bdd{a} \in \mathcal{V}_K \\ \unitvec{\mu} : K \in \supp \phi_{\bdd{a}}^{\mu} } } | \bdd{u}_{\bdd{a}, \mu} |_{\bdd{H}^1(K)}^2 + \sum_{ \gamma \in \mathcal{E}_K} | \bdd{u}_{\gamma} |_{\bdd{H}^1(K)}^2 \leq C \beta^{-2}(1 + \log^3 k) |\bdd{u}|_{\bdd{H}^1(K)}^2,
	\end{align}
	where $C$ is independent of $k$, $h_K$ and $\bdd{u}$.
\end{lemma}
\begin{proof}
	Let $\bdd{u} \in \tilde{\bdd{V}}_E$ and $K \in \mathcal{T}$. For any  $\bdd{c} \in \mathbb{R}^2$, we have the decomposition
	\begin{align*}
	\label{eq:u minus c decomp}
	\bdd{u} - \bdd{c} = \bdd{u}_c - \sum_{ \bdd{a} \in \mathcal{V}_K } \sum_{i=1}^{2} (\bdd{c} \cdot \unitvec{e}_i) \bdd{\Pi}_{\bdd{V}} (\phi_{\bdd{a}}\unitvec{e}_i) + \sum_{ \substack{\bdd{a} \in \mathcal{V}_K \\ \unitvec{\mu} : K \in \supp \phi_{\bdd{a}}^{\mu} } } \bdd{u}_{\bdd{a}, \mu} + \sum_{\gamma \in \mathcal{E}_K} \bdd{u}_{\gamma} \quad \text{on } K
	\end{align*}
	thanks to \cref{eq:constants in span}. Thus,
	\begin{align*}
	\hat{\bdd{u}} - \bdd{c} = \hat{\bdd{u}}_c - \sum_{ {\bdd{a}} \in \mathcal{V}_{K} } \sum_{i=1}^{2} (\bdd{c} \cdot \unitvec{e}_i) \bdd{\Pi}_{\bdd{V}} (\phi_{\bdd{a}}\unitvec{e}_i) \circ \bdd{F}_K + \sum_{ \substack{\bdd{a} \in \mathcal{V}_K \\ \unitvec{\mu} : K \in \supp \phi_{\bdd{a}}^{\mu} } } \hat{\bdd{u}}_{\bdd{a}, \mu} + \sum_{\gamma \in \mathcal{E}_K} \hat{\bdd{u}}_{\gamma} \quad \text{on } \hat{T},
	\end{align*}
	where $\hat{\bdd{u}} = \bdd{u} \circ \bdd{F}_K$, $\hat{\bdd{u}}_c = \bdd{u}_c \circ \bdd{F}_K$, etc. We first bound the energy of $\hat{\bdd{u}}_c$.
	Since
	\begin{align*}
	\hat{\bdd{u}}_c &= \sum_{ \bdd{a} \in \mathcal{V}_K } \sum_{i=1}^{2} (\bdd{u}(\bdd{a}) \cdot \unitvec{e}_i) \bdd{\Pi}_{\bdd{V}} (\phi_{\bdd{a}} \unitvec{e}_i) \circ \bdd{F}_K \\
	&= \sum_{ \hat{\bdd{a}} \in \mathcal{V}_{\hat{T}} } \sum_{i=1}^{2} (\hat{\bdd{u}}(\hat{\bdd{a}}) \cdot \unitvec{e}_i) \bdd{\Pi}_{\bdd{V}} (\phi_{\bdd{a}} \unitvec{e}_i) \circ \bdd{F}_K \quad \text{on } \hat{T},
	\end{align*}
	where $\bdd{a} = \bdd{F}_K(\hat{\bdd{a}})$,
	we use \cite[Corollary 6.3]{BCMP91} and \cref{eq:velocity c0 vertex decay} to obtain
	\begin{align}
	\|\hat{\bdd{u}}_c - \bdd{c}\|_{\bdd{H}^1(\hat{T})}^2 &\leq \sum_{\hat{\bdd{a}} \in \mathcal{V}_{\hat{T}}} |\hat{\bdd{u}}  (\hat{\bdd{a}}) - \bdd{c}|^2 \sum_{i=1}^{2} \| \bdd{\Pi}_{\bdd{V}} (\phi_{\bdd{a}} \unitvec{e}_i) \circ \bdd{F}_K \|_{\bdd{H}^1(\hat{T})}^2  \notag \\
	&\leq C (1+\log k) \| \hat{\bdd{u}} - \bdd{c} \|_{\bdd{H}^1(\hat{T})}^2. \label{eq:coarse decay}
	\end{align}

	We now bound the vertex derivative contribution. For $\bdd{a} \in \mathcal{V}_K$, we note that $\partial_{\mu}(\bdd{u}\cdot \unitvec{e}_i)(\bdd{a}) =  \unitvec{\mu}^T D\bdd{F}_K^{-T} D(\bdd{u} \cdot \unitvec{e}_i \circ \bdd{F}_K)(\hat{\bdd{a}})$, $i = 1,2$, where $\hat{\bdd{a}} = \bdd{F}_K^{-1}({\bdd{a}})$.  Applying \cite[Lemma 6.1]{AinJia19} to $D (\bdd{u} \circ \bdd{F}_K )$, and using \cref{eq:velocity c1 vertex decay} and shape regularity gives
	\begin{align}
	\label{eq:c1 decay}
	\|\hat{\bdd{u}}_{\bdd{a},\mu}\|_{\bdd{H}^1(\hat{T})}^2 &\leq C \|D\bdd{F}_K^{-T}\|_{L^{\infty}(K)}^2 \cdot \|D\bdd{F}_K\|_{L^{\infty}(\hat{T})}^{2} \cdot |D \hat{\bdd{u}} (\hat{\bdd{a}})|^2 k^{-4}
	\leq C | \hat{\bdd{u}} |_{\bdd{H}^1(\hat{T})}^2.
	\end{align}
	
	Now, we define 
	\begin{align*}
	\bdd{u}^{\#} := (\bdd{u} - \bdd{c}) - (\bdd{u}_{c} - \bdd{c}) - \sum_{ \substack{\bdd{a} \in \mathcal{V}_K \\ \unitvec{\mu} : K \in \supp \phi_{\bdd{a}}^{\mu} } } \bdd{u}_{\bdd{a},\mu} \quad \text{on } K.
	\end{align*}
	Then, $D^{\alpha} \bdd{u}^{\#}(\bdd{a}) = \bdd{0}$ for $\bdd{a} \in \mathcal{V}_K$, $|\alpha| \leq 1$, and thanks to \cref{eq:coarse decay,eq:c1 decay}, $\hat{\bdd{u}}^{\#} := \bdd{u}^{\#} \circ \bdd{F}_K$ may be estimated as follows:
	\begin{align}
	\label{eq:u sharp estimate}
	\| \hat{\bdd{u}}^{\#} \|_{\bdd{H}^1(\hat{T})}^2  \leq C(1 + \log k)\| \hat{\bdd{u}} - \bdd{c} \|_{\bdd{H}^1(\hat{T})}^2
	\end{align}
	Let $\gamma \in \mathcal{E}_K$. \Cref{eq:stokes ext equiv harmonic ext}, shape regularity \cref{eq:shape regularity}, and the trace theorem give
	\begin{align}
	\label{eq:u gamma intermediate}
	|\hat{\bdd{u}}_{\gamma}|_{\bdd{H}^{1}( \hat{T})} &\leq C \beta^{-1} | \bdd{u}_{\gamma} |_{\bdd{H}^{1/2}(\partial K)} \leq C \beta^{-1} | \hat{\bdd{u}}_{\gamma} |_{\bdd{H}^{1/2}(\partial \hat{T})} \leq C \beta^{-1} \| \hat{\bdd{u}}^{\#} \|_{\bdd{H}_{00}^{1/2}(\hat{\gamma})},
	\end{align}
	where $\hat{\gamma} = \bdd{F}_K^{-1}(\gamma)$. 	Thanks to \cite[Theorem 6.5]{BCMP91} and the trace theorem, we have the estimate
	\begin{align}
	\label{eq:u sharp 00 estimate}
	\| \hat{\bdd{u}}^{\#}  \|_{\bdd{H}^{1/2}_{00}(\hat{\gamma})} \leq C(1 + \log k) \| \hat{\bdd{u}}^{\#} \|_{\bdd{H}^{1/2}(\hat{\gamma})} \leq C(1 + \log k) \| \hat{\bdd{u}}^{\#} \|_{\bdd{H}^{1}(\hat{T})}.
	\end{align}
	Using \cref{eq:u sharp 00 estimate,eq:u sharp estimate,eq:u gamma intermediate} gives
	\begin{align}
	\label{eq:u gamma decay} 
	|\hat{\bdd{u}}_{\gamma} |_{\bdd{H}^1(\hat{T})}^2 \leq C \beta^{-2} (1 + \log^2 k) \| \hat{\bdd{u}}^{\#} \|_{\bdd{H}^1(\hat{T})}^2 \leq C \beta^{-2}(1 + \log^3 k)  \| \hat{\bdd{u}} - \bdd{c} \|_{\bdd{H}^1(\hat{T})}^2, 
	\end{align}	
	Combining \cref{eq:coarse decay,eq:c1 decay,eq:u gamma decay} leads to
	\begin{align*}
	| \hat{\bdd{u}}_c |_{\bdd{H}^1(\hat{T})}^2 + \sum_{ \substack{{\bdd{a}} \in \mathcal{V}_{K} \\ \unitvec{\mu} : K \in \supp \phi_{\bdd{a}}^{\mu} } } |\hat{\bdd{u}}_{\bdd{a}, \mu}|_{\bdd{H}^1(\hat{T})}^2 &+ \sum_{ \gamma \in \mathcal{E}_K} | \hat{\bdd{u}}_{\gamma} |_{\bdd{H}^1(\hat{T})}^2 \\ 
	&\qquad \leq C \beta^{-2} (1 + \log^3 k) \| \hat{\bdd{u}} - \bdd{c} \|_{\bdd{H}^1(\hat{T})}^2, 
	\end{align*}
	where we used that $| \hat{\bdd{u}}_c |_{\bdd{H}^1(\hat{T})} = | \hat{\bdd{u}}_c - \bdd{c} |_{\bdd{H}^1(\hat{T})}$. Taking the infimum over all $\bdd{c} \in \mathbb{R}^2$ and applying the quotient norm equivalence \cite[Theorem 7.2]{Nec11} gives
	\begin{align*}
	| \hat{\bdd{u}}_c |_{\bdd{H}^1(\hat{T})}^2 + \sum_{ \substack{\bdd{a} \in \mathcal{V}_K \\ \unitvec{\mu} : K \in \supp \phi_{\bdd{a}}^{\mu} } } |\hat{\bdd{u}}_{\bdd{a}, \mu}|_{\bdd{H}^1(\hat{T})}^2 + \sum_{ \gamma \in \mathcal{E}_K} | \hat{\bdd{u}}_{\gamma} |_{\bdd{H}^1(\hat{T})}^2 \leq C\beta^{-2}(1 + \log^3 k) |\hat{\bdd{u}}|_{\bdd{H}^1(\hat{T})}^2.
	\end{align*}	
	\Cref{eq:velocity reference stability} now follows from shape regularity \cref{eq:shape regularity}. 
\end{proof}

Summing \cref{eq:velocity reference stability} over the elements leads to the following:
\begin{lemma} %[Stable Decomposition of $\tilde{\bdd{V}}_E$]
	\label{lem:velocity stable decomp}
	There exists a constant $C$ independent of $k$ and $h$ such that
	\begin{align}
	\label{eq:velocity stability}
	\bar{a}(\bdd{u}, \bdd{u}) \leq C\beta^{-2}(1 + \log^3 k) a(\bdd{u}, \bdd{u}) \quad \forall \bdd{u} \in \tilde{\bdd{V}}_E.
	\end{align}
\end{lemma}

\end{document}